\documentclass[11pt]{article}

\usepackage{amsmath, amsfonts, amssymb, amsthm,  graphicx, mathtools, enumerate}

\usepackage[blocks, affil-it]{authblk}

\usepackage[numbers, square]{natbib}
\usepackage[CJKbookmarks=true,
            bookmarksnumbered=true,
			bookmarksopen=true,
			colorlinks=true,
			citecolor=red,
			linkcolor=blue,
			anchorcolor=red,
			urlcolor=blue]{hyperref}
\usepackage[usenames]{color}

\usepackage[letterpaper, left=1.2truein, right=1.2truein, top = 1.2truein, bottom = 1.2truein]{geometry}
\usepackage[ruled, vlined, lined, commentsnumbered]{algorithm2e}

\usepackage{prettyref,soul}

\newtheorem{proposition}{Proposition}[section]
\newtheorem{thm}{Theorem}[section]

\newtheorem{remark}{Remark}[section]
\newrefformat{eq}{(\ref{#1})}
\newrefformat{chap}{Chapter~\ref{#1}}
\newrefformat{sec}{Section~\ref{#1}}
\newrefformat{algo}{Algorithm~\ref{#1}}
\newrefformat{fig}{Fig.~\ref{#1}}
\newrefformat{tab}{Table~\ref{#1}}
\newrefformat{rmk}{Remark~\ref{#1}}
\newrefformat{clm}{Claim~\ref{#1}}
\newrefformat{def}{Definition~\ref{#1}}
\newrefformat{cor}{Corollary~\ref{#1}}
\newrefformat{lmm}{Lemma~\ref{#1}}
\newrefformat{lemma}{Lemma~\ref{#1}}
\newrefformat{prop}{Proposition~\ref{#1}}
\newrefformat{app}{Appendix~\ref{#1}}
\newrefformat{ex}{Example~\ref{#1}}
\newrefformat{exer}{Exercise~\ref{#1}}
\newrefformat{soln}{Solution~\ref{#1}}
\newrefformat{cond}{Condition~\ref{#1}}

\def\text#1{\mbox{\rm #1}}

\newcommand{\argmin}{\mathop{\rm argmin}}

\newcommand{\iprod}[2]{\left \langle #1, #2 \right\rangle}

\newtheorem*{thma}{Condition A}
\newtheorem*{thmb}{Condition B}
\newtheorem*{thmc}{Condition C}
\newtheorem*{thmd}{Condition D}
\newtheorem*{thme}{Condition E}

\newtheorem*{m1}{Condition M1}
\newtheorem*{m2}{Condition M2}
\newtheorem*{m2'}{Condition M2'}
\newtheorem*{m3}{Condition M3}
\newtheorem*{m4}{Condition M4}

\title{Goodness-of-Fit Tests for Random Partitions via Symmetric Polynomials
}
\author{Chao Gao}
\affil{
University of Chicago

chaogao@galton.uchicago.edu
}

\begin{document}
\maketitle

\begin{abstract}
We consider goodness-of-fit tests with i.i.d. samples generated from a categorical distribution $(p_1,...,p_k)$. For a given $(q_1,...,q_k)$, we test the null hypothesis whether $p_j=q_{\pi(j)}$ for some label permutation $\pi$. The uncertainty of label permutation implies that the null hypothesis is composite instead of being singular. In this paper, we construct a testing procedure using statistics that are defined as indefinite integrals of some symmetric polynomials. This method is aimed directly at the invariance of the problem, and avoids the need of matching the unknown labels. The asymptotic distribution of the testing statistic is shown to be chi-squared, and its power is proved to be nearly optimal under a local alternative hypothesis. Various degenerate structures of the null hypothesis are carefully analyzed in the paper. A two-sample version of the test is also studied.
\smallskip

\end{abstract}
% \begin{keyword}[class=AMS]
% \kwd[Primary ]{62H12}
% \kwd[; secondary ]{62C20}
% \end{keyword}
% \begin{keyword}
% \kwd{Convex programming, group Lasso, Minimax rates, Rates of convergence, Sparse CCA (SCCA)}
% \end{keyword}

% \end{frontmatter}

\section{Introduction}

Consider a categorical distribution parameterized by $(p_1,...,p_k)$. We have i.i.d. observations $X_1,...,X_n$ that follow $\mathbb{P}(X_i=j)=p_j$. A classical goodness-of-fit testing problem is to test whether or not $p_j=q_j$ for $j\in[k]$, where $q_1,...,q_k$ are some given numbers. One solution is given by the famous Pearson's chi-squared test \citep{pearson1900x}. In this traditional formulation, it is assumed that the labels $(1,....,k)$ of $(p_1,...,p_k)$ correspond to those of $(q_1,...,q_k)$, so that $p_j$ can be directly compared with $q_j$ for each $j\in[k]$. However, this assumption is not satisfied in some interesting applications. We give three examples below:
\begin{enumerate}
\item \textit{Clustering models.} In a typical probabilistic setting of cluster analysis, the event $\{X_i=j\}$ means that the $i$th item belongs to the $j$th cluster, and $p_j$ is the population frequency of the $j$th cluster. Here, the cluster label $j$ does not carry any real meaning, and is present only for notational convenience. In a cluster analysis setting, the underlying object of interest is the partition of the $n$ items instead of the cluster labels. In other words, what really matters to statisticians is the value of $\mathbb{I}\{X_i=X_{i'}\}$ (the indicator function of the event) for every pair $i\neq i'$. Therefore, a clustering model with population frequency $(p_1,...,p_k)$ is equivalent to that with $(p_{\pi(1)},...,p_{\pi(k)})$ with some permutation $\pi$.
\item \textit{Word frequency analysis.} Consider two text corpora of two different languages. The word frequencies are denoted by $(p_1,...,p_k)$ and $(q_1,...,q_k)$, respectively. An interesting problem in comparative linguistics is to study whether the two languages share common features by comparing $(p_1,...,p_k)$ with $(q_1,...,q_k)$. For languages that are not necessarily etymologically related, the correspondence between words of the two languages are usually unclear or unknown. Therefore, a reasonable comparison of word frequencies between two languages can be conducted through comparing $(p_1,...,p_k)$ with a reordered vector $(q_{\pi(1)},...,q_{\pi(k)})$ for some permutation $\pi$.
\item \textit{Simple substitution cipher.} In cryptography, a simple substitution cypher changes every character in a message to a different character systematically. Let $\{1,...,k\}$ be a finite alphabet of characters, and $(Y_1,...,Y_n)$ denote a message to be encrypted. A simple substitution cypher is defined by a permutation $\sigma$ on the alphabet $\{1,...,k\}$. This results in the encrypted message $(X_1,...,X_n)$ with $X_i=\sigma(Y_i)$ for each $i\in[n]$. Suppose each $Y_i$ is independently distributed by $\mathbb{P}(Y_i=j)=q_j$. Then, each $X_i$ independently follows $\mathbb{P}(X_i=j)=p_j$, where $p_j=q_{\pi(j)}$ with $\pi=\sigma^{-1}$. If only the encrypted message is observed, inference of the probability vector $(q_1,...,q_k)$ is only possible up to an unknown permutation.
\end{enumerate}

%In the classic formulation of null hypothesis $p_j=q_j$ for $j\in[k]$, the labels $1,....,k$ of $p_1,...,p_k$ correspond to those of $q_1,...,q_k$. However, in many applications, the labels of $p_1,...,p_k$ do not carry any real meaning, and they may not correspond to the labels of $q_1,...,q_k$. Consider an example of word frequency in computational linguistics. The top three frequent words in English are ``the", ``to" and ``be". The top three frequent words in Spanish are ``de", ``la" and ``que".  When one compares the word frequencies of English and Spanish, there is no one-to-one correspondence between English words and Spanish words. In other words, one can certainly label both English and Spanish words by $1,2,3,...,k$. However, these labels do not have any particular meaning.

Inspired by the above examples, in this paper, we consider a twist of the traditional formulation of the hypothesis testing problem. We consider the following null hypothesis:
\begin{equation}
H_0: p_j=q_{\pi(j)},\quad\text{for some }\pi\in S_k,\label{eq:intro-null}
\end{equation}
where $(q_1,...,q_k)$ is a known vector and $S_k$ is the set of all permutations of $[k]$. This null hypothesis implies that the labels $1,...,k$ do not have any meaning. For example, the vectors $\left(\frac{1}{2},\frac{1}{4},\frac{1}{4}\right)$ and $\left(\frac{1}{4},\frac{1}{4},\frac{1}{2}\right)$ are considered equivalent. Given i.i.d. observations $X_1,...,X_n$, one can immediately define summary statistics $n_j=\sum_{i=1}^n\mathbb{I}\{X_i=j\}$ for $j\in[k]$, which are sufficient. Since the labels of $n_1,...,n_k$ are irrelevant, these sufficient statistics result in a random partition of the integer $n$. There are two ways to code such a random partition \citep{pitman1995exchangeable}: (i) by the order statistics $n_{(1)}\geq n_{(2)}\geq ...\geq n_{(k)}$; (ii) by the numbers of terms of various sizes $m_l=\sum_{j=1}^k\mathbb{I}\{n_j=l\}$ for $l\in[n]$. It is easy to see that $\sum_{l=1}^nm_l=k$ and $\sum_{l=1}^nlm_l=n$. These two representations are equivalent because one can be derived from the other.

Inference of the probability vector $(p_1,...,p_k)$ up to a label permutation using random partitions have been extensively studied in Bayesian statistics. The problem serves as a foundation for random partition models, cluster analysis and species distribution modeling. Priors that induce various exchangeable properties have been developed for the equivalent class $\{(p_{\pi(1)},...,p_{\pi(k)}):\pi\in S_k\}$. See \cite{ewens1972sampling,pitman1995exchangeable,muller2013random,crane2016ubiquitous,gnedin2006exchangeable,de2015gibbs,kingman1975random,pitman1996random} and references therein. In this paper, we take a frequentist point of view that is complementary to the Bayesian literature, and we do not treat the equivalent class $\{(p_{\pi(1)},...,p_{\pi(k)}):\pi\in S_k\}$ as random. The theory of hypothesis testing is developed within a frequentist decision-theoretic framework.

%Another example is from statistical genetics \citep{ewens1972sampling}. An allele is a type of gene. For example, the alleles A, B, and O in the ABO blood group have frequencies $p_1$, $p_2$ and $p_3$. In some applications, the labels $1$, $2$ and $3$, or equivalently, A, B and O, do not matter. Then, the observations $X_1,...,X_n$ can be summarized as $m_1,...,m_n$, where $m_l$ is the number of alleles appearing exactly $l$ times, for each $l\in[n]$. The vector $(m_1,...,m_n)$ is called an allelic partition that satisfies $\sum_{l=1}^nlm_l=n$ \citep{crane2016ubiquitous}.\nb{re-write this paragraph}

With the unknown permutation $\pi$ in the null hypothesis, the classical chi-squared test by Pearson does not work anymore. Our idea of the test is based on the following class of statistics:
\begin{equation}
\left\{\sum_{j=1}^kf(n_j): f\in\mathcal{F}\right\},\label{eq:invariant}
\end{equation}
where $\mathcal{F}$ is the class of all measurable functions. For each $f\in\mathcal{F}$, the distribution of $\sum_{j=1}^kf(n_j)$ is identical for $p_j=q_{\pi(j)}$ with any $\pi\in S_k$. This is because (\ref{eq:invariant}) is a class of statistics that are invariant to the label permutation $\pi$. That is, $\sum_{j=1}^kf(n_j)=\sum_{j=1}^kf(n_{\pi(j)})$ for any $\pi\in S_k$. Moreover, it is easy to see that these statistics are all functions of the random partition because $\sum_{j=1}^kf(n_j)=\sum_{j=1}^kf(n_{(j)})$.

Choosing an appropriate class of $f$'s is important. We propose to use $k$ functions $f_1,...,f_k$ that satisfy the \textit{identifiability} and the \textit{orthogonality} conditions. The identifiability condition requires that the $k$ equations $\sum_{j=1}^kf_l(p_j)=\sum_{j=1}^kf_l(q_j)$ for $l\in[k]$ hold if and only if $p_j=q_{\pi(j)}$ for some $\pi\in S_k$. With this condition, testing whether the null hypothesis holds is equivalent to testing whether the $k$ equations hold. The orthogonality condition requires that the $k$ vectors $(f_l'(q_1),...,f_l'(q_k))^T$ for $l\in[k]$ are orthogonal to each other. Intuitively speaking, this condition ensures that the information carried by the $k$ statistics $\sum_{j=1}^kf_l(n_j)$ for $l\in[k]$ are mutually exclusive, which is a key ingredient that leads to optimal power under a local alternative.

In this paper, we choose $f_1,...,f_k$ to be indefinite integrals of Lagrange interpolation polynomials. The choice of these polynomials satisfies the above-mentioned \textit{identifiability} and \textit{orthogonality} conditions. We prove that the testing statistic constructed from the $k$ functions is asymptotically distributed by a chi-squared distribution. Moreover, we show that the power of the test is nearly optimal under a local alternative hypothesis within a decision-theoretic framework.

Our approach that uses symmetric polynomials bypasses the problem of unknown permutation $\pi$. It falls into the general umbrella of methods of moments, which are commonly used for problems that impose equivalence relations to the signals through the action of a group of transformations. For example, various method-of-moments techniques have been applied to problems including Gaussian mixture models \citep{hsu2013learning}, mixed membership models \citep{anandkumar2013tensor}, dictionary learning \citep{arora2014new}, topic models \citep{anandkumar2012method,arabshahi2016spectral} and multi-reference alignment \citep{perry2017sample}. Recently, this idea was also applied to the problems of network testing by \cite{gao2017testing,banerjee2017optimal}, where the group action there is row and column permutations of the adjacency matrix of a random network.

The rest of the paper is organized as follows. In Section \ref{sec:def}, we introduce definitions of some useful symmetric polynomials and the related Vandermonde matrix. Before getting into the testing problem for random partitions, we first solve an easier version of the problem with Gaussian observations in Section \ref{sec:gauss} and Section \ref{sec:degen}. The test using random partitions is given in Section \ref{sec:cat}. The optimality of our test is discussed in Section \ref{sec:opt}. In Section \ref{sec:2-sample}, we consider a two-sample version of the problem. Numerical experiments of the proposed testing procedures are given in Section \ref{sec:num}. Finally, Section \ref{sec:disc} is a discussion section, where we briefly analyze the property of the test on the boundary of degeneracy and discuss some open problems. The proofs of all results in the paper are given in Section \ref{sec:proof}.

We close this section by introducing the notation used in the paper. For $a,b\in\mathbb{R}$, let $a\vee b=\max(a,b)$ and $a\wedge b=\min(a,b)$. For an integer $m$, $[m]$ denotes the set $\{1,2,...,m\}$. Given a set $S$, $|S|$ denotes its cardinality, and $\mathbb{I}_S$ is the associated indicator function. We use $\mathbb{P}$ and $\mathbb{E}$ to denote generic probability and expectation whose distribution is determined from the context. The noncentral chi-squared distribution with degrees of freedom $k$ and noncentrality parameter $\delta^2$ is denoted as $\chi_{k,\delta^2}^2$. We will also use $\chi_{k,\delta^2}^2$ for the associated random variables.

\section{Symmetric Polynomials and Vandermonde Matrix}\label{sec:def}

Define a polynomial with roots $\mu_1,...,\mu_k\in\mathbb{R}$ by
$$f(t)=\prod_{j=1}^k(t-\mu_j).$$
It can be organized as
\begin{equation}
f(t)=\sum_{j=0}^k(-1)^{k-j}e_{k-j}(\mu_1,...,\mu_k)t^j.\label{eq:expansion}
\end{equation}
The coefficient before $t^j$ is $(-1)^{k-j}e_{k-j}(\mu_1,...,\mu_k)$, and $e_{k-j}(\mu_1,...,\mu_k)$ is called the elementary symmetric polynomial. For $l\in\{1,...,k\}$, the $l$th elementary symmetric polynomial is
$$e_l(\mu_1,...,\mu_k)=\sum_{1\leq j_1<\cdots<j_l\leq k}\mu_{j_1}\cdots\mu_{j_l}.$$
When $l=0$, we use the convention $e_0(\mu_1,...,\mu_k)=1$.

The elementary symmetric polynomials can be efficiently calculated through Newton's identities. Define the $l$th power sum
\begin{equation}
p_l(\mu_1,...,\mu_k)=\sum_{j=1}^k\mu_j^l.\label{eq:power-sum}
\end{equation}
Newton's identities can be summarized through the formula
\begin{equation}
e_l(\mu_1,...,\mu_k) = \frac{1}{l}\sum_{j=1}^l(-1)^{j-1}e_{l-j}(\mu_1,...,\mu_k)p_j(\mu_1,...,\mu_k),\label{eq:newton}
\end{equation}
for $l=1,...,k$.

Finally, we introduce an interesting relation between elementary symmetric polynomials and Vandermonde matrix (see Chapter 0.9.11 of \cite{horn2012matrix}). Given $\mu_1,...,\mu_k$ that take $k$ distinct values, define a matrix $E(\mu_1,...,\mu_k)\in\mathbb{R}^{k\times k}$, whose $(j,l)$th entry is
\begin{equation}
(-1)^{j-1}\frac{e_{k-j}(\mu_1,...,\mu_{l-1},\mu_{l+1},...,\mu_k)}{\prod_{j\in[k]\backslash\{l\}}(\mu_j-\mu_l)}.\label{eq:def-E}
\end{equation}
The Vandermonde matrix $V(\mu_1,...,\mu_k)\in\mathbb{R}^{k\times k}$ has $\mu_j^{l-1}$ on its $(j,l)$th entry. Interestingly, we have
\begin{equation}
E(\mu_1,...,\mu_k)V(\mu_1,...,\mu_k)=V(\mu_1,...,\mu_k)E(\mu_1,...,\mu_k)=I_k.\label{eq:inv-E}
\end{equation}
This relation implies a formula for the determinant of $E(\mu_1,...,\mu_k)$:
\begin{equation}
\det(E(\mu_1,...,\mu_k))=\frac{1}{\det(V(\mu_1,...,\mu_k))}=\frac{1}{\prod_{1\leq j<l\leq k}(\mu_l-\mu_j)}.\label{eq:det-E}
\end{equation}

\section{The Gaussian Case}\label{sec:gauss}

Before working with categorical distributions, we first study data generated from a Gaussian distribution. This allows us to grasp the mathematical essence of the problem without dealing with the dependence and heteroskedasticity of categorical distributions.
We consider a Gaussian random vector $X\sim N(\theta,n^{-1}I_k)$. The mean vector $\theta\in\mathbb{R}^k$ consists of $k$ numbers $\theta_1,...,\theta_k$. Throughout the paper, we assume $k\geq 2$ and it is a constant that does not vary with $n$. We would like to test whether the $k$ numbers are identical to $\mu_1,...,\mu_k$ after some permutation of labels. To be rigorous, introduce a distance between two vectors $\theta$ and $\mu$,
$$\ell(\theta,\mu)=\min_{\pi\in S_k}\sqrt{\sum_{j=1}^k(\theta_j-\mu_{\pi(j)})^2},$$
where $S_k$ is the set of all permutations on $[k]$. Then, the hypothesis testing problem is
$$H_0: \ell(\theta,\mu)=0,\quad H_1:\ell(\theta,\mu)>0.$$
Throughout this section, we assume $\min_{j\neq l}|\mu_j-\mu_l|>0$. The case $\min_{j\neq l}|\mu_j-\mu_l|=0$ is degenerate and will be studied in the next section.

We use the notation $\mu_{\pi}$ to denote a $k$-dimensional vector whose $j$th entry is $\mu_{\pi(j)}$. Then, the null hypothesis can also be written as
$$\theta\in\{\mu_{\pi}:\pi\in S_k\}.$$
In other words, there is an equivalent class of probability distributions $\{N(\mu_{\pi},n^{-1}I_k):\pi\in S_k\}$. Thus, it is natural to consider summary statistics whose distributions are invariant under this equivalent class. This leads to the class of summary statistics
$$\left\{\sum_{j=1}^kf(X_j):f\in\mathcal{F}\right\},$$
where $\mathcal{F}$ is the set of all measurable functions. For $X\sim N(\theta,n^{-1}I_k)$, it is easy to see that the distribution of $\sum_{j=1}^kf(X_j)$ only depends on the equivalent class $\{N(\theta_k,n^{-1}I_k):\pi\in S_k\}$. This fact holds for an arbitrary $f\in\mathcal{F}$.

Since the degree of freedom of the null hypothesis is $k$, our strategy is to construct a testing procedure based on $\left\{\sum_{j=1}^kf_l(X_j):l\in[k]\right\}$. In other words, we need to choose the $k$ functions $f_1(\cdot),...,f_k(\cdot)$. The following two conditions are proposed:
\begin{enumerate}
\item \textbf{Identifiability}. Assume $\min_{j\neq l}|\mu_j-\mu_l|>0$. Then the equations
\begin{equation}
\sum_{j=1}^kf_l(\theta_j)=\sum_{j=1}^kf_l(\mu_j),\quad l=1,...,k,\label{eq:equations}
\end{equation}
hold, if and only if $\ell(\theta,\mu)=0$.
\item \textbf{Orthogonality}. Assume $\min_{j\neq l}|\mu_j-\mu_l|>0$. Then for any $l,h\in[h]$,
$$\sum_{j=1}^kf_l'(\mu_j)f_h'(\mu_j)=\begin{cases}
1, & l=h, \\
0, & l\neq h.
\end{cases}$$
\end{enumerate}
We give a few remarks regarding the two conditions. The first condition of identifiability is natural. It is required by the structure of the problem, and is necessary for the test to have power under the alternative hypothesis. The second condition implies information independence among the $k$ summary statistics.

The $k$ functions we propose to satisfy the two conditions are
\begin{equation}
f_l(t)=\frac{\int\prod_{j\in[k]\backslash\{l\}}(t-\mu_j)}{\prod_{j\in[k]\backslash\{l\}}(\mu_l-\mu_j)},\quad l=1,...,k.\label{eq:def-f}
\end{equation}
The derivatives $f_l'(t)=\frac{\prod_{j\in[k]\backslash\{l\}}(t-\mu_j)}{\prod_{j\in[k]\backslash\{l\}}(\mu_l-\mu_j)}$ are called Lagrange interpolating polynomials, and it is easy to check that the second condition of orthogonality holds. Now we check the first condition of identifiability. By (\ref{eq:expansion}), we have
$$\prod_{j\in[k]\backslash\{l\}}(t-\mu_j)=\sum_{j=0}^{k-1}(-1)^{k-1-j}e_{k-1-j}(\mu_1,...,\mu_{l-1},\mu_{l+1},...,\mu_k)t^j.$$
This implies
\begin{equation}
f_l(t)=\sum_{j=1}^k(-1)^{k-j}\frac{e_{k-j}(\mu_1,...,\mu_{l-1},\mu_{l+1},...,\mu_k)}{\prod_{j\in[k]\backslash\{l\}}(\mu_l-\mu_j)}\frac{t^j}{j}.\label{eq:expansion-of-f}
\end{equation}
Therefore, the equations (\ref{eq:equations}) can be written as
$$\sum_{j=1}^k(-1)^{k-j}\frac{e_{k-j}(\mu_1,...,\mu_{l-1},\mu_{l+1},...,\mu_k)}{\prod_{j\in[k]\backslash\{l\}}(\mu_l-\mu_j)}\Delta_j=0,\quad l=1,...,k.$$
where $\Delta_j=\frac{1}{j}\sum_{h=1}^k\theta_h^j-\frac{1}{j}\sum_{h=1}^k\mu_h^j$. In view of the definition of the matrix $E(\mu_1,...,\mu_k)$ in (\ref{eq:def-E}), we have a compact organization of the equations
$$E(\mu_1,...,\mu_k)\Delta=0.$$
When the assumption $\min_{j\neq l}|\mu_j-\mu_l|>0$ holds, the matrix $E(\mu_1,...,\mu_k)$ has full rank and is invertible according to (\ref{eq:inv-E}) and (\ref{eq:det-E}), which immediately implies $\Delta=0$. Equivalently,
$$p_j(\theta_1,...,\theta_k)=p_j(\mu_1,...,\mu_k),\quad j=1,...,k.$$
The definition of the power sum $p_j(\cdots)$ is given in (\ref{eq:power-sum}). By Newton's identities (\ref{eq:newton}), we have
$$e_j(\theta_1,...,\theta_k)=e_j(\mu_1,...,\mu_k),\quad j=1,...,k.$$
Finally, the relation between elementary symmetric polynomials and roots in (\ref{eq:expansion}) implies that $\prod_{j=1}^k(t-\theta_j)$ and $\prod_{j=1}^k(t-\mu_j)$ are the same polynomials. Hence, we obtain the conclusion $\ell(\theta,\mu)=0$. The other direction trivially holds. This verifies the condition of identifiability for the functions $f_1,...,f_k$.

\begin{remark}
The computation of the statistic $\sum_{j=1}^kf_l(X_j)$ for each $l\in[k]$ is straightforward, thanks to the formula (\ref{eq:expansion-of-f}). According to (\ref{eq:expansion-of-f}), we can write
$$\sum_{j=1}^kf_l(X_j)=\sum_{i=1}^k(-1)^{k-i}\frac{e_{k-i}(\mu_1,...,\mu_{l-1},\mu_{l+1},...,\mu_k)}{\prod_{i\in[k]\backslash\{l\}}(\mu_l-\mu_i)}\frac{\sum_{j=1}^kX_j^i}{i}.$$
In other words, $\sum_{j=1}^kf_l(X_j)$ is a linear combination of empirical moments $\{\sum_{j=1}^kX_j^i: i\in[k]\}$. To compute the elementary symmetric polynomial $e_{k-i}(\mu_1,...,\mu_{l-1},\mu_{l+1},...,\mu_k)$ in the coefficient, one can recursively apply Newton's identities (\ref{eq:newton}). The overall complexity of computing $\sum_{j=1}^kf_l(X_j)$ requires $O(k^2)$ products.
\end{remark}

We propose the testing statistic
\begin{equation}
T=n\sum_{l=1}^k\left(\sum_{j=1}^kf_l(X_j)-\sum_{j=1}^kf_l(\mu_j)\right)^2.\label{eq:def-T}
\end{equation}
When the value of $T$ is large, the equations (\ref{eq:equations}) are unlikely to hold.
Thus, the null hypothesis will be rejected when $T$ exceeds some threshold. The asymptotic distribution of $T$ can be derived under the null hypothesis.

\begin{thma}
Assume $\mu_1,...,\mu_k$ are $k$ different numbers that do not vary with $n$.
\end{thma}
Some possible extensions beyond Condition A will be discussed in Section \ref{sec:opt}.

\begin{thm}\label{thm:null-gauss}
Under Condition A, $T\leadsto \chi_k^2$ as $n\rightarrow\infty$ under the null hypothesis.
\end{thm}

For a chi-squared random variable $\chi_k^2$, define a number $\chi_k^2(\alpha)$ such that
$$\mathbb{P}\left(\chi_k^2\leq \chi_k^2(\alpha)\right)=1-\alpha.$$
Then, a testing function is
$$\phi_{\alpha}=\mathbb{I}\left\{T>\chi_k^2(\alpha)\right\}.$$
By Theorem \ref{thm:null-gauss}, its asymptotic Type-1 error is $\alpha$. The next result characterizes the regime where the asymptotic power of the test tends to $1$. It is a consequence of the fact that the functions $f_1,...,f_k$ satisfy the identifiability condition.
\begin{thm}\label{thm:power-Gaussian}
Under Condition A, the following two statements are equivalent
\begin{enumerate}
\item $\lim_{n\rightarrow\infty}\sqrt{n}\ell(\theta,\mu)= \infty$;
\item $\lim_{n\rightarrow\infty}\mathbb{P}_{\theta}\left(T>\chi_k^2(\alpha)\right) = 1$, for any constant $\alpha\in (0,1)$,
\end{enumerate}
where the probability $\mathbb{P}_{\theta}$ denotes $N(\theta,n^{-1}I_k)$.
\end{thm}
Theorem \ref{thm:power-Gaussian} shows that $\lim_{n\rightarrow\infty}\sqrt{n}\ell(\theta,\mu)= \infty$ is the necessary and sufficient condition for the asymptotic power of the test to be one. For a local alternative such that $\sqrt{n}\ell(\theta,\mu)=O(1)$, the test will have a non-trivial power between $0$ and $1$. This contiguous regime will be studied in Section \ref{sec:opt}.

\section{Degeneracy of the Problem}\label{sec:degen}

In the last section, we construct a chi-squared test under the assumption that $\min_{j\neq l}|\mu_j-\mu_l|>0$. When $\min_{j\neq l}|\mu_j-\mu_l|=0$, the identifiability condition of the functions $f_1,...,f_k$ defined in (\ref{eq:def-f}) does not hold. We need to construct summary statistics based on new functions in this degenerate case.

Assume there is a partition of the set $[k]$. That is, for some $d\leq k$, we have $\cup_{h=1}^d\mathcal{C}_h=[k]$, and for any $g,h\in[d]$ such that $g\neq h$, $\mathcal{C}_g\cap\mathcal{C}_h=\varnothing$. We assume
$$\mu_j=\nu_h,\text{ for all }j\in\mathcal{C}_h.$$
Moreover, we require that $\min_{g\neq h}|\nu_g-\nu_h|>0$.
To motivate the appropriate functions that we will propose, we consider two extreme cases. The first case is when $d=k$. Then, the condition $\min_{j\neq l}|\mu_j-\mu_l|>0$ still holds, and we can still use the functions $f_1,...,f_k$ defined in (\ref{eq:def-f}). The second case is when $d=1$. This implies $\mu_1=\mu_2=\cdots=\mu_k=\nu_1$. Then, we can use the function
\begin{equation}
g(t)=(t-\nu_1)^2.\label{eq:def-g1}
\end{equation}
This leads to an obvious chi-squared statistic $T_g=n\sum_{j=1}^kg(X_j)$.

For a general $d$, we need to borrow ideas from both extreme cases. We define functions $f_1,...,f_d$ that are modifications from (\ref{eq:def-f}). Define
\begin{equation}
f_h(t)=\frac{\int\prod_{g\in[d]\backslash\{h\}}(t-\nu_g)}{\prod_{g\in[d]\backslash\{h\}}(\nu_h-\nu_g)},\quad h=1,...,d.\label{eq:def-f-degen}
\end{equation}
We also need another function to ensure identifiability. Define
\begin{equation}
g(t)=\frac{\prod_{g=1}^d(t-\nu_g)^2}{\sum_{g=1}^d\prod_{h\in[d]\backslash\{g\}}(t-\nu_h)^2}.\label{eq:def-g-degen}
\end{equation}
The function $g(t)$ is well defined when $d\geq 2$. When $d=1$, we use the definition given by (\ref{eq:def-g1}). The following proposition shows that the functions $f_1,...,f_d,g$ together ensure identifiability via the equations
\begin{equation}
\sum_{j=1}^kf_h(\theta_j)=\sum_{j=1}^kf_h(\mu_j),\quad h=1,...,d,\label{eq:equ-f}
\end{equation}
and
\begin{equation}
\sum_{j=1}^kg(\theta_j)=\sum_{j=1}^kg(\mu_j). \label{eq:equ-g}
\end{equation}

\begin{proposition}\label{prop:iden}
Assume $\min_{g\neq h}|\nu_g-\nu_h|>0$. We have the following conclusions.
\begin{enumerate}
\item
When $d=1$, the equation (\ref{eq:equ-g})
holds if and only if $\ell(\theta,\mu)=0$.
\item
When $2\leq d\leq k-1$, the equations (\ref{eq:equ-f}) and (\ref{eq:equ-g})
hold if and only if $\ell(\theta,\mu)=0$.
\item
When $d=k$, the equation (\ref{eq:equ-f}) holds if and only if $\ell(\theta,\mu)=0$.
\end{enumerate}
\end{proposition}

The first conclusion of the above proposition is obvious. The last conclusion is proved in Section \ref{sec:gauss}. Here we show the second conclusion. Using a similar argument that we used in Section \ref{sec:gauss}, the equations (\ref{eq:equ-f}) can be written as
$$E(\nu_1,...,\nu_d)\Delta=0,$$
where $\Delta\in\mathbb{R}^d$ is a vector with the $h$th entry being $\Delta_h=\frac{1}{h}\sum_{j=1}^k\theta_j^h-\frac{1}{h}\sum_{j=1}^k\mu_j^h$. In other words, we have
$$\sum_{j=1}^k\theta_j^h=\sum_{j=1}^k\mu_j^h,\quad \text{for }h=1,...,d.$$
The equation (\ref{eq:equ-g}) immediately implies that each $\theta_j$ only takes value in the set $\{\nu_1,...,\nu_d\}$. Therefore, there exists a partition $[k]=\cup_{h=1}^d\mathcal{D}_h$ such that $\mathcal{D}_g\cap\mathcal{D}_h=\varnothing$ for all $g\neq h$, and $\theta_j=\nu_g$ for all $j\in\mathcal{D}_g$. This leads to
$$\sum_{j=1}^k\theta_j^h=\sum_{g=1}^d|\mathcal{D}_g|\nu_g^h.$$
We also have
$$\sum_{j=1}^k\mu_j^h=\sum_{g=1}^d|\mathcal{C}_g|\nu_g^h.$$
Hence, we obtain the equations
$$\sum_{g=1}^d|\mathcal{D}_g|\nu_g^h=\sum_{g=1}^d|\mathcal{C}_g|\nu_g^h,\quad\text{for }h=0,1,...,d-1.$$
The equation for $h=0$ holds because $\sum_{g=1}^d|\mathcal{D}_g|=\sum_{g=1}^d|\mathcal{C}_g|=k$.
Again, with matrix notation, these equations can be written as $V(\nu_1,...,\nu_d)r=0$, where $V(\nu_1,...,\nu_d)$ is the Vandermonde matrix, and $r\in\mathbb{R}^d$ is a vector with its $g$th entry being $r_g=|\mathcal{D}_g|-|\mathcal{C}_g|$. When $\min_{g\neq h}|\nu_g-\nu_h|>0$ holds, $V(\nu_1,...,\nu_d)$ has full rank, which leads to $|\mathcal{D}_g|=|\mathcal{C}_g|$ for all $g=1,...,d$. Finally, we can conclude that $\ell(\theta,\mu)=0$. The other direction is obvious.

The above proof actually shows that the function $f_d$ is not needed when $d\leq k-1$. The equation with $h=d$ in (\ref{eq:equ-f}) is redundant for identifiability. The second conclusion of Proposition \ref{prop:iden} would still hold without it. However, we still keep it for the convenience of analyzing the proposed test.

We propose two testing statistics. Define
\begin{equation}
T_f=n\sum_{h=1}^d\frac{1}{|\mathcal{C}_h|}\left(\sum_{j=1}^kf_h(X_j)-\sum_{j=1}^kf_h(\mu_j)\right)^2,\label{eq:def-Tf}
\end{equation}
and
\begin{equation}
T_g=n\sum_{j=1}^kg(X_j).\label{eq:def-Tg}
\end{equation}
We present asymptotic distributions of $T_f$ and $T_g$. Since the case $d=1$ is trivial, we only present results for $d\geq 2$.
\begin{thmb}
Assume $\mu_1,...,\mu_k$ are $k$ numbers that do not vary with $n$. Moreover, $\mu_j=\nu_h$ for all $j\in\mathcal{C}_h$, $h\in[d]$.
\end{thmb}

\begin{thm}\label{thm:null-gauss-degen}
Under Condition B, $T_f\leadsto \chi_d^2$, $T_g\leadsto \chi_k^2$ and $T_g-T_f\leadsto \chi_{k-d}^2$ as $n\rightarrow\infty$ under the null hypothesis.
\end{thm}

\begin{figure}
\centering
\includegraphics[width=16cm]{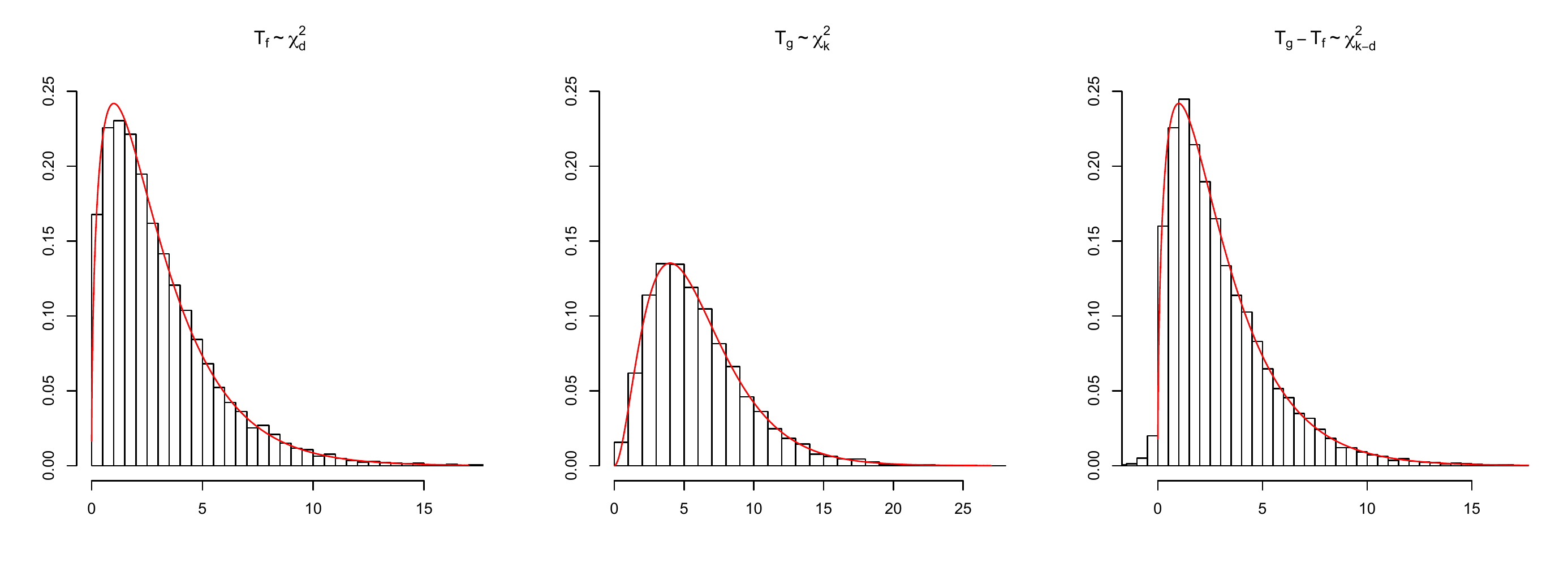}
\caption{histograms of testing statistics with $\mu=(1,3,3,3,5,5)$, $k=6$, $d=3$ and $n=200$.}\label{fig:1}
\end{figure}

Interestingly, Theorem \ref{thm:null-gauss-degen} exhibits an analysis-of-variance type of result. The three statistics all exhibit asymptotic chi-square distributions (see Figure \ref{fig:1}). The statistic $T_g$ dominates $T_f$ in probability under the null hypothesis. An analogous analysis-of-variance type of result continues to hold under a local alternative (see Theorem \ref{thm:M2}).
We define the testing function as
$$\phi_{\alpha}=\mathbb{I}\{T_f>\chi_d^2(\alpha)\}\vee\mathbb{I}\{T_g>\chi_k^2(\alpha)\},$$
where we use the notation $a\vee b$ for $\max(a,b)$. Since under the null hypothesis, $T_g\geq T_f$ in probability, the asymptotic Type-1 error is just the probability of the event $\{T_g>\chi_k^2(\alpha)\}$, which tends to $\alpha$ as $n\rightarrow\infty$. In fact, as we will show later in Section \ref{sec:opt}, the behavior of the testing function mainly depends on the statistic $T_g$ in the contiguous neighborhood of the null hypothesis. The statistic $T_f$ helps to ensure that the testing function has asymptotic power $1$ as soon as $\sqrt{n}\ell(\theta,\mu)\rightarrow\infty$. Without $T_f$, the identifiability property of the test established in Proposition \ref{prop:iden} would break down, and the test would lose power outside of the contiguous neighborhood of the null hypothesis. The next theorem rigorously establishes this fact.

\begin{thm}\label{thm:power-degen-Gaussian}
Under Condition B, the following two statements are equivalent
\begin{enumerate}
\item $\lim_{n\rightarrow\infty}\sqrt{n}\ell(\theta,\mu)=\infty$;
\item $\lim_{n\rightarrow\infty}\mathbb{P}_{\theta}\left(T_f>\chi_d^2(\alpha)\text{ or }T_g>\chi_k^2(\alpha)\right)=1$, for any constant $\alpha\in(0,1)$,
\end{enumerate}
where the probability $\mathbb{P}_{\theta}$ denotes $N(\theta,n^{-1}I_k)$.
\end{thm}

\section{The Case of Categorical Distribution}\label{sec:cat}

Now we are ready to transfer our wisdom from Gaussian distribution to categorical distribution.
Consider i.i.d. observations $X_1,...,X_n$ from a categorical distribution $(p_1,...,p_k)$. To be specific, for each $i\in[n]$ and $j\in[k]$, $\mathbb{P}(X_i=j)=p_j$. Throughout this section, we use $\mathbb{P}_p$ to denote the probability distribution of $X_1,...,X_n$. We would like to test whether the $k$ numbers $p_1,...,p_k$ are identical to some given $q_1,...q_k$ after a permutation of labels. Introduce a distance between the two vectors $p$ and $q$,
\begin{equation}
\ell(p,q)=\min_{\pi\in S_k}2\sqrt{\sum_{j=1}^k(\sqrt{p_j}-\sqrt{q_{\pi(j)}})^2}.\label{eq:lpq}
\end{equation}
The hypothesis testing problem is
$$H_0:\ell(p,q)=0,\quad H_1:\ell(p,q)>0.$$
For each $j\in[k]$, define
$$\hat{p}_j=\frac{1}{n}\sum_{i=1}^n\mathbb{I}\{X_i=j\}.$$
Pearson's chi-squared test \citep{pearson1900x} is defined as $\chi^2=n\sum_{j=1}^k\frac{(\hat{p}_j-q_j)^2}{q_j}$, which is asymptotically distributed as $\chi_{k-1}^2$ when $p=q$. However, this test only works when the null hypothesis is simple. Here, our null hypothesis is composite, and there is uncertainty from the underlying permutation of the labels. 

Our idea is to borrow the solution for the Gaussian case in Section \ref{sec:gauss}. Intuitively, the vector $(\hat{p}_1,...,\hat{p}_k)$ is asymptotically Gaussian after some normalization. However, the normalization step brings extra difficulty for this problem. In the definition of Pearson's chi-squared test, each $\hat{p}_j$ is normalized by $\sqrt{q_j}$ because of the heteroskedasticity and dependence structure of the vector $(\hat{p}_1,...,\hat{p}_k)$. This normalization does not work in our setting because the underlying label is not given, and we do not know which $\sqrt{q_j}$ to use. To overcome this issue, we adopt the technique of variance-stabilizing transformation \citep{anscombe1948transformation}, and directly work with $\sqrt{\hat{p}_j}$. 

This leads to a modification of the definition of the function $f_l(\cdot)$, and the new definition is given by
\begin{equation}
f_l(t)=\frac{\int\prod_{j\in[k]\backslash\{l\}}(\sqrt{t}-\sqrt{q_j})}{\prod_{j\in[k]\backslash\{l\}}(\sqrt{q_l}-\sqrt{q_j})},\quad l=1,...,k.\label{eq:def-f-sqrt}
\end{equation}
The testing statistic is
\begin{equation}
T=4n\sum_{l=1}^k\left(\sum_{j=1}^kf_l(\hat{p}_j)-\sum_{j=1}^kf_l(q_j)\right)^2.\label{eq:T-cat}
\end{equation}
Similar to the discussion in Section \ref{sec:gauss}, when the value of $T$ is large, the equations (\ref{eq:equations}) are unlikely to hold.
Thus, the null hypothesis will be rejected when $T$ exceeds some threshold. The asymptotic distribution of $T$ can be derived under the null hypothesis.

\begin{thmc}
Assume $q_1,...,q_k$ are $k$ different numbers in $(0,1)$ that do not vary with $n$.
\end{thmc}
Some possible extensions beyond Condition C will be discussed in Section \ref{sec:opt}.

\begin{thm}\label{thm:null-gauss-cat}
Under Condition C, $T\leadsto \chi_{k-1}^2$ as $n\rightarrow\infty$ under the null hypothesis.
\end{thm}

For a chi-squared random variable $\chi_{k-1}^2$, define a number $\chi_{k-1}^2(\alpha)$ such that
$$\mathbb{P}\left(\chi_{k-1}^2\leq \chi_{k-1}^2(\alpha)\right)=1-\alpha.$$
Then, a testing function is
$$\phi_{\alpha}=\mathbb{I}\left\{T>\chi_{k-1}^2(\alpha)\right\}.$$
By Theorem \ref{thm:null-gauss-cat}, its asymptotic Type-1 error is $\alpha$. The next result characterizes the regime where the asymptotic power of the test tends to $1$. It is a consequence of the fact that the functions $f_1,...,f_k$ satisfy the identifiability condition, though with slightly different definitions.
\begin{thm}\label{thm:power-Gaussian-cat}
Under Condition C, the following two statements are equivalent
\begin{enumerate}
\item $\lim_{n\rightarrow\infty}\sqrt{n}\ell(p,q)= \infty$;
\item $\lim_{n\rightarrow\infty}\mathbb{P}_{p}\left(T>\chi_{k-1}^2(\alpha)\right) = 1$, for any constant $\alpha\in(0,1)$,
\end{enumerate}
where the probability $\mathbb{P}_p$ is defined in the beginning of this section.
\end{thm}

Next, we study the degenerate case where the $k$ numbers $q_1,...,q_k$ only take $d$ values. There is a partition $[k]=\cup_{h=1}^d\mathcal{C}_h$ such that for any $g\neq h$, $\mathcal{C}_g\cap\mathcal{C}_h=\varnothing$. We assume the following condition.
\begin{thmd}
Asume $q_1,...,q_k$ are $k$ numbers in $(0,1)$ that do not vary with $n$. Moreover, $q_j=r_h$ for all $j\in\mathcal{C}_h$, $h\in[d]$.
\end{thmd}
The approach we take is similar to that in Section \ref{sec:degen}, assisted with the technique of variance-stabilizing transformation. Define
\begin{equation}
f_h(t)=\frac{\int\prod_{g\in[d]\backslash\{h\}}(\sqrt{t}-\sqrt{r_g})}{\prod_{g\in[d]\backslash\{h\}}(\sqrt{r_h}-\sqrt{r_g})},\quad h=1,...,d,\label{eq:def-f-sqrt-degen}
\end{equation}
and
\begin{equation}
g(t)=\frac{\prod_{g=1}^d(\sqrt{t}-\sqrt{r_g})^2}{\sum_{g=1}^d\prod_{h\in[d]\backslash\{g\}}(\sqrt{t}-\sqrt{r_h})^2}.\label{eq:def-g-sqrt-degen}
\end{equation}
Then, define the testing statistics
\begin{equation}
T_f=4n\sum_{h=1}^d\frac{1}{|\mathcal{C}_h|}\left(\sum_{j=1}^kf_h(\hat{p}_j)-\sum_{j=1}^kf_h(q_j)\right)^2,\label{eq:def-Tf-cat}
\end{equation}
and
\begin{equation}
T_g=4n\sum_{j=1}^kg(\hat{p}_j).\label{eq:def-Tg-cat}
\end{equation}
The properties of $T_f$ and $T_g$ are given by the following theorem. Again, the case $d=1$ is trivial, and we only present results for $d\geq 2$.
\begin{thm}\label{thm:null-cat-degen}
Under Condition D, $T_g\leadsto \chi_{k-1}^2$, $T_f\leadsto \chi_{d-1}^2$ and $T_g-T_f\leadsto \chi_{k-d}^2$ as $n\rightarrow\infty$ under the null hypothesis.
\end{thm}
We define the testing function
$$\phi_{\alpha}=\mathbb{I}\{T_f>\chi_{d-1}^2(\alpha)\}\vee\mathbb{I}\{T_g>\chi_{k-1}^2(\alpha)\}.$$
By Theorem \ref{thm:null-cat-degen}, the Type-1 error of this test converges to $\alpha$. Though $T_f$ is dominated by $T_g$ under the null hypothesis, both are needed to ensure the power goes to $1$ under the alternative.
\begin{thm}\label{thm:power-degen-cat}
Under Condition D, the following two statements are equivalent
\begin{enumerate}
\item $\lim_{n\rightarrow\infty}\sqrt{n}\ell(p,q)=\infty$;
\item $\lim_{n\rightarrow\infty}\mathbb{P}_p\left(T_f>\chi_{d-1}^2(\alpha)\text{ or }T_g>\chi_{k-1}^2(\alpha)\right)=1$, for any constant $\alpha\in(0,1)$,
\end{enumerate}
where the probability $\mathbb{P}_p$ is defined in the beginning of this section.
\end{thm}

\section{Optimality of the Test}\label{sec:opt}

In this section, we study the optimality issue of the testing problem from a decision-theoretic perspective. The goal is to understand the fundamental limit of the problem and establish optimality results of the proposed testing procedures. We propose to study the setting where a null hypothesis is tested against a local alternative. This leads to a nontrivial power function and a precise asymptotic characterization of the minimax risk of the test. Depending on whether the data generating process is Gaussian or categorical, and whether the null hypothesis is degenerate or not, the optimality of the test will be studied in four different cases.

\subsection{Gaussian Distribution: Non-Degenerate Case}\label{sec:opt-gauss}

We first consider the non-degenerate situation. That is, we assume that $\mu_1,...,\mu_k$ are $k$ different numbers. In Section \ref{sec:gauss}, we impose the assumption that the $k$ numbers $\mu_1,...,\mu_k$ do not depend on $n$. This assumption can be made significantly weaker. For two indices $j$ and $l$ that are not equal, define
$$\eta_{jl}=\frac{1}{\mu_j-\mu_l}\prod_{h\in[k]\backslash\{j,l\}}\frac{\mu_l-\mu_h}{\mu_j-\mu_h}.$$
It characterizes the relative difference between $\mu_j$ and $\mu_l$ in the background of the set $\{\mu_1,...,\mu_k\}$.
\begin{m1}
Assume $\lim_{n\rightarrow\infty}\max_{j\neq l}\frac{|\eta_{jl}|}{\sqrt{n}}=0$.
\end{m1}
To understand Condition M1, we can interpret $|\eta_{jl}|+|\eta_{lj}|$ as approximately the inverse distance between $\mu_j$ and $\mu_l$. Therefore, we allow the possibility that $|\mu_j-\mu_l|$ converges to $0$, but not as fast as $n^{-1/2}$. Otherwise, the data cannot tell the difference between $\mu_j\neq \mu_l$ and $\mu_j=\mu_l$, which is equivalent to the degenerate case. Recall that the number $k$ is assumed to be a constant that does not vary with $n$ throughout the paper.

Consider the testing problem
\begin{equation}
H_0: \theta\in\Theta_0=\left\{\theta:\ell(\theta,\mu)=0\right\},\quad H_1:\theta\in\Theta_{\delta}=\left\{\theta:\ell(\theta,\mu)=\frac{\delta}{\sqrt{n}}\right\}.\label{eq:problem}
\end{equation}
That is, we test the null hypothesis against its contiguous alternative. The choice of $H_1$ ensures a non-trivial asymptotic power.
We measure the testing error via the minimax risk function
$$R_n(k,\delta)=\inf_{0\leq\phi\leq 1}\left\{\sup_{\theta\in\Theta_0}\mathbb{P}_{\theta}\phi+\sup_{\theta\in\Theta_{\delta}}\mathbb{P}_{\theta}(1-\phi)\right\}.$$
The probability symbol $\mathbb{P}_{\theta}$ stands for $N(\theta,n^{-1}I_k)$. Throughout the paper, we assume $k$ and $\delta$ are fixed constants independent of $n$.

We first present the lower bound.
\begin{thm}\label{thm:lower-gauss}
Under Condition M1, for sufficiently large $n$, we have
$$R_n(k,\delta)\geq \inf_{t>0}\left(\mathbb{P}\left(\chi_k^2>t\right)+\mathbb{P}(\chi_{k,\delta^2}^2\leq t)\right).$$
\end{thm}

Theorem \ref{thm:lower-gauss} gives the benchmark of the problem. Using the proposed testing statistic $T$ defined in (\ref{eq:def-T}), we can achieve this benchmark.
\begin{thm}\label{thm:upper-gauss}
Consider the testing procedure $\phi=\mathbb{I}\{T>t^*\}$, where $T$ is defined in (\ref{eq:def-T}), and
$$t^*=\argmin_{t>0}\left(\mathbb{P}\left(\chi_k^2>t\right)+\mathbb{P}(\chi_{k,\delta^2}^2\leq t)\right).$$
Under Condition M1, we have
$$\sup_{\theta\in\Theta_0}\mathbb{P}_{\theta}\phi+\sup_{\theta\in\Theta_{\delta}}\mathbb{P}_{\theta}(1-\phi)\leq(1+o(1))\inf_{t>0}\left(\mathbb{P}\left(\chi_k^2>t\right)+\mathbb{P}(\chi_{k,\delta^2}^2\leq t)\right),$$
as $n\rightarrow\infty$.
\end{thm}

Theorem \ref{thm:upper-gauss} characterizes both Type-1 and Type-2 error of the test $\phi=\mathbb{I}\{T>t^*\}$. The conclusion holds for any local alternative with $\delta\in(0,\infty)$. It complements the result of Theorem \ref{thm:power-Gaussian}.
Combining Theorem \ref{thm:lower-gauss} and Theorem \ref{thm:upper-gauss}, we conclude that the minimax testing error has the following asymptotic formula
$$R_n(k,\delta)=(1+o(1))\inf_{t>0}\left(\mathbb{P}\left(\chi_k^2>t\right)+\mathbb{P}(\chi_{k,\delta^2}^2\leq t)\right),$$
and this error can be achieved by the test $\phi=\mathbb{I}\{T>t^*\}$ with some carefully chosen $t^*$ only depending on $k$ and $\delta$.

\subsection{Gaussian Distribution: Degenerate Case}\label{sec:opt-gauss-degen}

Now we consider situations of degeneracy. In Section \ref{sec:degen}, it is assumed that $\mu_1,...,\mu_k$ only take $d$ different values. This assumption can be relaxed. Here, we assume the $k$ numbers $\mu_1,...,\mu_k$ can be approximately clustered into $d$ groups. Given $d$ different numbers $\nu_1,...,\nu_d$, for any pair $g\neq h$, define
\begin{equation}
\bar{\eta}_{gh}=\frac{1}{\nu_g-\nu_h}\prod_{l\in[k]\backslash\{g,h\}}\frac{\nu_h-\nu_l}{\nu_g-\nu_l}.\label{eq:bar-eta}
\end{equation}
\begin{m2}
Assume $\lim_{n\rightarrow\infty}\max_{g\neq h}\frac{|\bar{\eta}_{gh}|}{\sqrt{n}}=0$ and there is a partition $\mathcal{C}_1,...,\mathcal{C}_d$ of $[k]$, such that $\limsup_{n\rightarrow\infty}\max_{1\leq g\leq d}\max_{j\in\mathcal{C}_g}\sqrt{n}|\mu_j-\nu_g|=0$.
\end{m2}

This condition says that $\mu_1,...,\mu_k$ can be approximately clustered into $d$ groups. The within-group distance is of a smaller order than $n^{-1/2}$, and the between-group distance is of a larger order than $n^{-1/2}$.

Consider the same local testing problem (\ref{eq:problem}).
The lower bound of the degenerate setting is given by the following theorem.
\begin{thm}\label{thm:lower-gauss-degen}
Under Condition M2, $n\rightarrow\infty$, we have
$$R_n(k,\delta)\geq (1+o(1))\inf_{t>0}\left(\mathbb{P}\left(\chi_k^2>t\right)+\mathbb{P}(\chi_{k,\delta^2}^2\leq t)\right).$$
\end{thm}

This lower bound can be achieved asymptotically using the testing statistics $T_f$ and $T_g$ defined in (\ref{eq:def-Tf}) and (\ref{eq:def-Tg}).

\begin{thm}\label{thm:upper-gauss-degen}
Consider the testing procedure $\phi=\mathbb{I}\{T_f>t^*\}\vee\mathbb{I}\{T_g>t^*\}$, where $T_f$ and $T_g$ are defined in (\ref{eq:def-Tf}) and (\ref{eq:def-Tg}), and
$$t^*=\argmin_{t>0}\left(\mathbb{P}\left(\chi_k^2>t\right)+\mathbb{P}(\chi_{k,\delta^2}^2\leq t)\right).$$
Under Condition $M_2$, we have
$$\sup_{\theta\in\Theta_0}\mathbb{P}_{\theta}\phi+\sup_{\theta\in\Theta_{\delta}}\mathbb{P}_{\theta}(1-\phi)\leq(1+o(1))\inf_{t>0}\left(\mathbb{P}\left(\chi_k^2>t\right)+\mathbb{P}(\chi_{k,\delta^2}^2\leq t)\right),$$
as $n\rightarrow\infty$.
\end{thm}

Theorem \ref{thm:upper-gauss-degen} shows that the test $\phi=\mathbb{I}\{T_f>t^*\}\vee\mathbb{I}\{T_g>t^*\}$ achieves the optimal error asymptotically under a local alternative. As we will show in Theorem \ref{thm:M2}, $T_g\geq T_f$ in probability under a local alternative that $\sqrt{n}\ell(\theta,\mu)=\delta\in(0,\infty)$. Therefore, the test $\phi=\mathbb{I}\{T_f>t^*\}\vee\mathbb{I}\{T_g>t^*\}$ is asymptotically equivalent to $\mathbb{I}\{T_g>t^*\}$, and the latter only uses $T_g$. Though the role of the statistic $T_f$ is negligible for a local alternative, we have already shown in Theorem \ref{thm:power-degen-Gaussian} that as soon as $\sqrt{n}\ell(\theta,\mu)\rightarrow\infty$, the effect of using $T_f$ starts to kick in and it is necessary to use both $T_f$ and $T_g$ for the asymptotic power to approach one.

\subsection{Categorical Distribution: Non-Degenerate Case}

We study the fundamental limit of testing for the categorical distribution. In Section \ref{sec:cat}, we assume $q_1,...,q_k$ are $k$ different numbers that do not depend on $n$. In this section, we consider a condition that is significantly weaker.
Define
$$\zeta_{jl}=\frac{1}{\sqrt{q_j}-\sqrt{q_l}}\prod_{h\in[k]\backslash\{j,l\}}\frac{\sqrt{q_l}-\sqrt{q_h}}{\sqrt{q_j}-\sqrt{q_h}}.$$
Similar to the definition of $\eta_{jl}$, $\zeta_{jl}$ characterizes the relative difference between $\sqrt{q_j}$ and $\sqrt{q_l}$ in the background of the set $\{\sqrt{q_1},...,\sqrt{q_k}\}$.

\begin{m3}
Assume $\lim_{n\rightarrow\infty}\max_{j\neq l}\frac{|\zeta_{jl}|}{\sqrt{n}}=0$ and $\min_{1\leq j\leq k}nq_j(1-q_j)\rightarrow\infty$.
\end{m3}

Compared with Condition M1, the extra requirement $\min_{1\leq j\leq k}nq_j(1-q_j)\rightarrow\infty$ in Condition M3 ensures that each $q_j$ is bounded away from $0$ and $1$ with a gap at least of order $n^{-1}$. If this extra requirement does not hold, $q_j$ would be asymptotically equivalent to $0$ or $1$, which results in a degenerate variance.

Consider the testing problem
\begin{equation}
H_0:p\in\mathcal{P}_0=\{p:\ell(p,q)=0\},\quad H_1:p\in\mathcal{P}_{\delta}=\left\{p:\ell(p,q)=\frac{\delta}{\sqrt{n}}\right\}.\label{eq:local-pq}
\end{equation}
Recall that the distance $\ell(\cdot,\cdot)$ is defined in (\ref{eq:lpq}).

We present the lower bound.
\begin{thm}\label{thm:lower-cat}
Under Condition M3, as $n\rightarrow\infty$, we have
$$R_n(k,\delta)\geq (1+o(1))\inf_{t>0}\left(\mathbb{P}\left(\chi_{k-1}^2>t\right)+\inf_{\{\delta_1,\delta_2:\delta_1^2+\delta_2^2=\delta^2\}}\mathbb{P}(\chi_{k-1,\delta_1^2}^2+\delta_2^2\leq t)\right).$$
\end{thm}
Theorem \ref{thm:lower-cat} gives the benchmark of the problem. Using the testing statistic $T$ defined in (\ref{eq:T-cat}), we can achieve this benchmark.

\begin{thm}\label{thm:upper-cat}
Consider the testing procedure $\phi=\mathbb{I}\{T>t^*\}$, where $T$ is defined in (\ref{eq:T-cat}), and
$$t^*=\argmin_{t>0}\left(\mathbb{P}\left(\chi_{k-1}^2>t\right)+\sup_{\{\delta_1,\delta_2:\delta_1^2+\delta_2^2=\delta^2\}}\mathbb{P}(\chi_{k-1,\delta_1^2}^2+\delta_2^2\leq t)\right).$$
Under Condition M3, we have
$$\sup_{\theta\in\mathcal{P}_0}\mathbb{P}_{\theta}\phi+\sup_{\theta\in\mathcal{P}_{\delta}}\mathbb{P}_{\theta}(1-\phi)\leq (1+o(1))\inf_{t>0}\left(\mathbb{P}\left(\chi_{k-1}^2>t\right)+\sup_{\{\delta_1,\delta_2:\delta_1^2+\delta_2^2=\delta^2\}}\mathbb{P}(\chi_{k-1,\delta_1^2}^2+\delta_2^2\leq t)\right),$$
as $n\rightarrow\infty$.
\end{thm}

The upper bound given by Theorem \ref{thm:upper-cat} does not exactly match the lower bound given by Theorem \ref{thm:lower-cat}. The difference lies in the Type-2 error. For the lower bound, we get $\inf_{\{\delta_1,\delta_2:\delta_1^2+\delta_2^2=\delta^2\}}\mathbb{P}(\chi_{k-1,\delta_1^2}^2+\delta_2^2\leq t)$, while for the upper bound, it is $\sup_{\{\delta_1,\delta_2:\delta_1^2+\delta_2^2=\delta^2\}}\mathbb{P}(\chi_{k-1,\delta_1^2}^2+\delta_2^2\leq t)$. These two quantities are close, because for any $\delta_1$ and $\delta_2$ that satisfy $\delta_1^2+\delta_2^2=\delta^2$, the expectation of $\chi_{k-1,\delta_1^2}^2+\delta_2^2$ is always $k-1+\delta^2$. Therefore, the test using the statistic $T$ is nearly optimal.

\subsection{Categorical Distribution: Degenerate Case}

Finally, we study the categorical distribution with the presence of degeneracy. In Section \ref{sec:cat}, we consider the situation where $q_1,...,q_k$ take $d$ different values. Here, we propose a much weaker condition. Given $d$ different numbers $r_1,...,r_d\in(0,1)$, for any pair $g\neq h$, define
$$\bar{\zeta}_{gh}=\frac{1}{\sqrt{r_g}-\sqrt{r_h}}\prod_{l\in[k]\backslash\{g,h\}}\frac{\sqrt{r_h}-\sqrt{r_l}}{\sqrt{r_g}-\sqrt{r_l}}.$$

\begin{m4}
Assume $\lim_{n\rightarrow\infty}\max_{j\neq l}\frac{|\bar{\zeta}_{jl}|}{\sqrt{n}}=0$, $\min_{1\leq j\leq k}nq_j(1-q_j)\rightarrow\infty$, and there is a partition $\mathcal{C}_1,...,\mathcal{C}_d$ of $[k]$, such that $\limsup_{n\rightarrow\infty}\max_{1\leq g\leq d}\max_{j\in\mathcal{C}_g}\sqrt{n}|\sqrt{q_j}-\sqrt{r_g}|=0$.
\end{m4}
Condition M4 has the same interpretation as Condition M2. The extra requirement $\min_{1\leq j\leq k}nq_j(1-q_j)\rightarrow\infty$ is also needed in Condition M3 to prevent the variance from being degenerate.

The lower bound of the local testing problem (\ref{eq:local-pq}) is given by the next theorem.

\begin{thm}\label{thm:lower-cat-degen}
Under Condition M4, as $n\rightarrow\infty$, we have
$$R_n(k,\delta)\geq (1+o(1))\inf_{t>0}\left(\mathbb{P}\left(\chi_{k-1}^2>t\right)+\inf_{\{\delta_1,\delta_2:\delta_1^2+\delta_2^2=\delta^2\}}\mathbb{P}(\chi_{k-1,\delta_1^2}^2+\delta_2^2\leq t)\right).$$
\end{thm}

For the matching upper bound, we can use the proposed testing statistics $T_f$ and $T_g$ defined in (\ref{eq:def-Tf-cat}) and (\ref{eq:def-Tg-cat}).

\begin{thm}\label{thm:upper-cat-degen}
Consider the testing procedure $\phi=\mathbb{I}\{T_f>t^*\}\vee\mathbb{I}\{T_g>t^*\}$, where $T_f$ and $T_g$ are defined in (\ref{eq:def-Tf-cat}) and (\ref{eq:def-Tg-cat}), and
$$t^*=\argmin_{t>0}\left(\mathbb{P}\left(\chi_{k-1}^2>t\right)+\sup_{\{\delta_1,\delta_2:\delta_1^2+\delta_2^2=\delta^2\}}\mathbb{P}(\chi_{k-1,\delta_1^2}^2+\delta_2^2\leq t)\right).$$
Under Condition M4, we have
$$\sup_{\theta\in\mathcal{P}_0}\mathbb{P}_{\theta}\phi+\sup_{\theta\in\mathcal{P}_{\delta}}\mathbb{P}_{\theta}(1-\phi)\leq (1+o(1))\inf_{t>0}\left(\mathbb{P}\left(\chi_{k-1}^2>t\right)+\sup_{\{\delta_1,\delta_2:\delta_1^2+\delta_2^2=\delta^2\}}\mathbb{P}(\chi_{k-1,\delta_1^2}^2+\delta_2^2\leq t)\right),$$
as $n\rightarrow\infty$.
\end{thm}

\section{Two-Sample Test}\label{sec:2-sample}

Consider two categorical distributions $(p_1,...,p_k)$ and $(q_1,...,q_k)$. Suppose we observe i.i.d. observations $X_1,...,X_n$ from $(p_1,...,p_k)$ and i.i.d. observations $Y_1,...,Y_m$ from $(q_1,...,q_k)$. We assume that $X_1,...,X_n$ are independent of $Y_1,...,Y_m$. The hypothesis testing problem we study in this section is
$$H_0:\ell(p,q)=0, \quad H_1:\ell(p,q)>0,$$
where the distance $\ell(\cdot,\cdot)$ is defined in (\ref{eq:lpq}).
The two-sample testing problem is harder than the one-sample version that we have just studied. The major difficulty is that the definitions of the functions (\ref{eq:def-f-sqrt-degen}) and (\ref{eq:def-g-sqrt-degen}) all depend on the values of $(p_1,...,p_k)$ and $(q_1,...,q_k)$ under the null hypothesis, which is not available anymore in the two-sample scenario.

Our idea is to estimate the unknown $(p_1,...,p_k)$ and $(q_1,...,q_k)$ from the data, and then construct data-driven versions of (\ref{eq:def-f-sqrt-degen}) and (\ref{eq:def-g-sqrt-degen}).

For each $j\in[k]$, define $\hat{p}_j=\frac{1}{n}\sum_{i=1}^n\mathbb{I}\{X_i=j\}$. Next, we will apply a variable clustering procedure to $(\hat{p}_1,...,\hat{p}_k)$. The goal is to find a partition $\underline{\mathcal{C}}_1,...,\underline{\mathcal{C}}_{\underline{d}}$ of $[k]$ according to
$$j\sim l\quad\text{if}\quad \sqrt{n}|\sqrt{\hat{p}_j}-\sqrt{\hat{p}_l}|\leq\lambda_n.$$
Algorithmically, one can first sort the vector $(\hat{p}_1,...,\hat{p}_k)$, and then find the partition sequentially.
There exists a permutation $\sigma\in S_k$, such that we can rank the empirical frequencies as $\hat{p}_{\sigma(1)}\leq \hat{p}_{\sigma(2)}\leq ...\leq \hat{p}_{\sigma(k)}$. Let $\hat{j}_1$ be the largest $j$ such that $\sqrt{n}|\sqrt{\hat{p}_{\sigma(j)}}-\sqrt{\hat{p}_{\sigma(1)}}|\leq\lambda_n$, where $\lambda_n$ is some threshold to be specified later. Then, the first cluster is defined as $\underline{\mathcal{C}}_1=\left\{\sigma(1),\sigma(2),...,\sigma(\hat{j}_1)\right\}$. Similarly, we can define the second cluster as $\underline{\mathcal{C}}_2=\{\sigma(\hat{j}_1+1),...,\sigma(\hat{j}_2)\}$, where $\hat{j}_2$ is the largest $j$ such that $\sqrt{n}|\sqrt{\hat{p}_{\sigma(j)}}-\sqrt{\hat{p}_{\sigma(\hat{j}_1)}}|\leq\lambda_n$. We continue this operation until we obtain a partition $\underline{\mathcal{C}}_1,...,\underline{\mathcal{C}}_{\underline{d}}$ of $[k]$. Here, $\underline{d}$ is the number of clusters estimated from the data. Now, for each $g\in[\underline{d}]$, we find the center of the cluster by $\sqrt{\underline{r}_g}=\frac{1}{|\underline{\mathcal{C}}_g|}\sum_{j\in \underline{\mathcal{C}}_g}\sqrt{\hat{p}_j}$. With the numbers $\underline{r}_1,...,\underline{r}_{\underline{d}}$, we define
$$\underline{f}_h(t)=\frac{\int\prod_{g\in[\underline{d}]\backslash\{h\}}(\sqrt{t}-\sqrt{\underline{r}_g})}{\prod_{g\in[\underline{d}]\backslash\{h\}}(\sqrt{\underline{r}_h}-\sqrt{\underline{r}_g})},\quad h=1,...,\underline{d},$$
and
$$\underline{g}(t)=\frac{\prod_{g=1}^{\underline{d}}(\sqrt{t}-\sqrt{\underline{r}_g})^2}{\sum_{g=1}^{\underline{d}}\prod_{h\in[\underline{d}]\backslash\{g\}}(\sqrt{t}-\sqrt{\underline{r}_h})^2}.$$

We repeat the above procedure on the observations $Y_1,...,Y_m$. For each $j\in[k]$, define $\hat{q}_j=\frac{1}{m}\sum_{i=1}^m\mathbb{I}\{Y_i=j\}$. Then, apply the same variable clustering procedure on $(\hat{q}_1,...,\hat{q}_k)$, and we obtain a partition $\overline{\mathcal{C}}_1,...,\overline{\mathcal{C}}_{\overline{d}}$ of $[k]$. For each $g\in[\overline{d}]$, define $\sqrt{\overline{r}_g}=\frac{1}{|\overline{\mathcal{C}}_g|}\sum_{j\in \overline{\mathcal{C}}_g}\sqrt{\hat{q}_j}$. Analogous definitions of $\underline{f}_h$'s and $\underline{g}$ are given by
$$\overline{f}_h(t)=\frac{\int\prod_{g\in[\overline{d}]\backslash\{h\}}(\sqrt{t}-\sqrt{\overline{r}_g})}{\prod_{g\in[\overline{d}]\backslash\{h\}}(\sqrt{\overline{r}_h}-\sqrt{\overline{r}_g})},\quad h=1,...,\overline{d},$$
and
$$\overline{g}(t)=\frac{\prod_{g=1}^{\overline{d}}(\sqrt{t}-\sqrt{\overline{r}_g})^2}{\sum_{g=1}^{\overline{d}}\prod_{h\in[\overline{d}]\backslash\{g\}}(\sqrt{t}-\sqrt{\overline{r}_h})^2}.$$

Now we can define testing statistics for this problem:
\begin{eqnarray}
\label{eq:rum} T_f  &=& \frac{2nm}{n+m}\sum_{h=1}^{\underline{d}}\frac{1}{|\underline{\mathcal{C}}_h|}\left(\sum_{j=1}^k\underline{f}_h(\hat{p}_j)-\sum_{j=1}^k\underline{f}_h(\hat{q}_j)\right)^2 \\
\nonumber && + \frac{2nm}{n+m}\sum_{h=1}^{\overline{d}}\frac{1}{|\overline{\mathcal{C}}_h|}\left(\sum_{j=1}^k\overline{f}_h(\hat{p}_j)-\sum_{j=1}^k\overline{f}_h(\hat{q}_j)\right)^2,
\end{eqnarray}
and
\begin{equation}
\label{eq:vodka} T_g =  \frac{2nm}{n+m}\left(\sum_{j=1}^k\underline{g}(\hat{q}_j) + \sum_{j=1}^k\overline{g}(\hat{p}_j)\right).
\end{equation}

The asymptotic distributions of the testing statistics under the null distribution are given below.
\begin{thme}
Asume $q_1,...,q_k$ are $k$ numbers in $(0,1)$ that do not vary with $n$. Moreover, there exists a $d\geq 2$ and a partition $[k]=\cup_{h=1}^d\mathcal{C}_h$, such that $q_j=r_h$ for all $j\in\mathcal{C}_h$, $h\in[d]$.
\end{thme}
\begin{thm}\label{thm:two-sample-null}
Assume $\lambda_n$ is a diverging sequence that satisfies $\lambda_n=o(\sqrt{n})$ and we also assume $\frac{m}{n+m}\rightarrow\beta\in(0,1)$. Under Condition E, we have
\begin{eqnarray*}
T_g &\leadsto& \frac{1}{2}\beta\mathcal{X}_1+\frac{1}{2}(1-\beta)\mathcal{X}_2+\mathcal{X}_3,\\
T_f &\leadsto& \mathcal{X}_3, \\
T_g-T_f &\leadsto& \frac{1}{2}\beta\mathcal{X}_1+\frac{1}{2}(1-\beta)\mathcal{X}_2,
\end{eqnarray*}
as $n\rightarrow\infty$ under the null hypothesis, where $\mathcal{X}_1$, $\mathcal{X}_2$ and $\mathcal{X}_3$ are independent random variables distributed as $\chi_{k-d}^2$, $\chi_{k-d}^2$ and $\chi_{d-1}^2$, respectively.
\end{thm}

Let $\mathcal{X}(\alpha)$ be the number that satisfies $\mathbb{P}\left(\frac{1}{2}\beta\mathcal{X}_1+\frac{1}{2}(1-\beta)\mathcal{X}_2+\mathcal{X}_3>\mathcal{X}(\alpha)\right)=\alpha$.
We define the testing function as
$$\phi_{\alpha}=\mathbb{I}\{T_f>\mathcal{X}(\alpha)\}\vee\mathbb{I}\{T_g>\mathcal{X}(\alpha)\}.$$
Theorem \ref{thm:two-sample-null} implies that this test has asymptotic Type-1 error $\alpha$. The next result characterizes the power behavior of the test.

\begin{thm}\label{thm:two-sample-power}
Assume $\lambda_n$ is a diverging sequence that satisfies $\lambda_n=o(\sqrt{n})$ and we also assume $\frac{m}{n+m}\rightarrow\beta\in(0,1)$. Under Condition E, the following two statements are equivalent
\begin{enumerate}
\item $\lim_{n\rightarrow\infty}\sqrt{n}\ell(p,q)=\infty$;
\item $\lim_{n\rightarrow\infty}\mathbb{P}_{p,q}\left(T_f>\mathcal{X}(\alpha)\text{ or }T_g>\mathcal{X}(\alpha)\right)=1$, for any constant $\alpha\in(0,1)$,
\end{enumerate}
where the probability $\mathbb{P}_{p,q}$ stands for the joint distribution of $X_1,...,X_n,Y_1,...,Y_m$.
\end{thm}

Theorem \ref{thm:two-sample-power} assumes Condition E. That is, $q_1,...,q_k$ are fixed numbers do that depend on $n$, and $p_1,...,p_k$ are allowed to vary with $n$. One can also assume an analogous condition for $p_1,...,p_k$ as fixed numbers that satisfy Condition E, and allow $q_1,...,q_k$ to vary with $n$.

\section{Numerical Studies}\label{sec:num}

In this section, we conduct numerical experiments to verify the theoretical properties of the proposed testing procedures. In each of the following scenarios, we compute power functions of $\alpha$-level tests for $\alpha=0.05$ with various sample sizes.

\noindent\textbf{Scenario 1.}
Consider $X\sim N(\theta,n^{-1}I_k)$, and we test the null hypothesis $\ell(\theta,\mu)=0$ with $\mu$ specified as $\mu=(1,2,3,4,5)$.

\noindent\textbf{Scenario 2.}
Consider $X\sim N(\theta,n^{-1}I_k)$, and we test the null hypothesis $\ell(\theta,\mu)=0$ with $\mu$ specified as $\mu=(1,3,3,3,5,5)$.

\noindent\textbf{Scenario 3.}
Consider $X_1,...,X_n\sim (p_1,...,p_k)$, and we test the null hypothesis $\ell(p,q)=0$ with $q$ specified as $q=(0.1,0.2,0.3,0.4)$.

\noindent\textbf{Scenario 4.}
Consider $X_1,...,X_n\sim (p_1,...,p_k)$, and we test the null hypothesis $\ell(p,q)=0$ with $q$ specified as $q=(0.1,0.1,0.4,0.4)$.

\noindent\textbf{Scenario 5.}
Consider $X_1,...,X_n\sim (p_1,...,p_k)$ and $Y_1,...,Y_m\sim (q_1,...,q_k)$, and we test the null hypothesis $\ell(p,q)=0$. We set $p=(0.1,0.1,0.4,0.4)$ and $q$ to be local perturbations of $p$ in a $O(n^{-1/2})$ neighborhood of $p$.

\begin{figure}
\centering
\includegraphics[width=10cm]{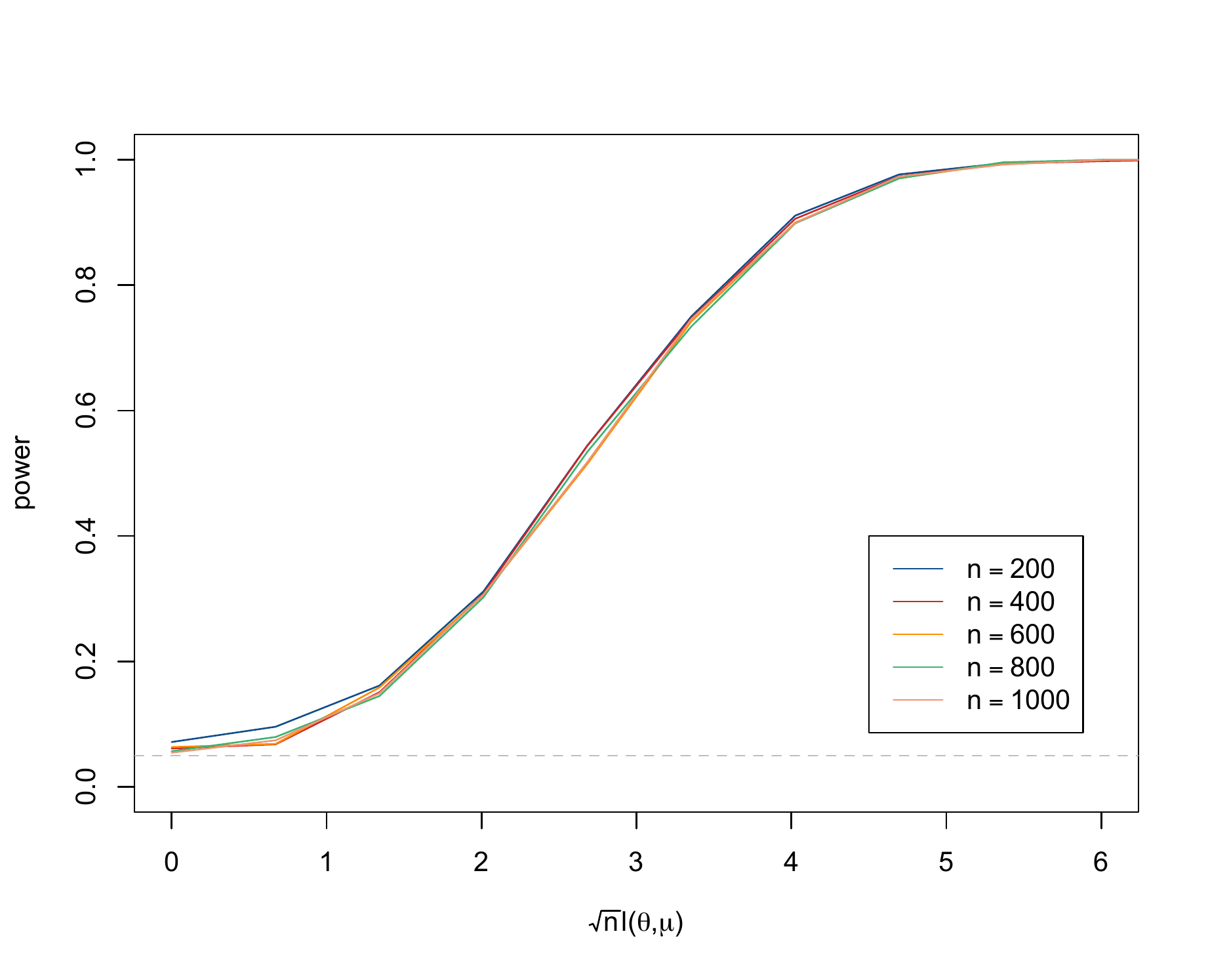}
\caption{Power Curve of Scenario 1}\label{fig:2}
\end{figure}
\begin{figure}
\centering
\includegraphics[width=10cm]{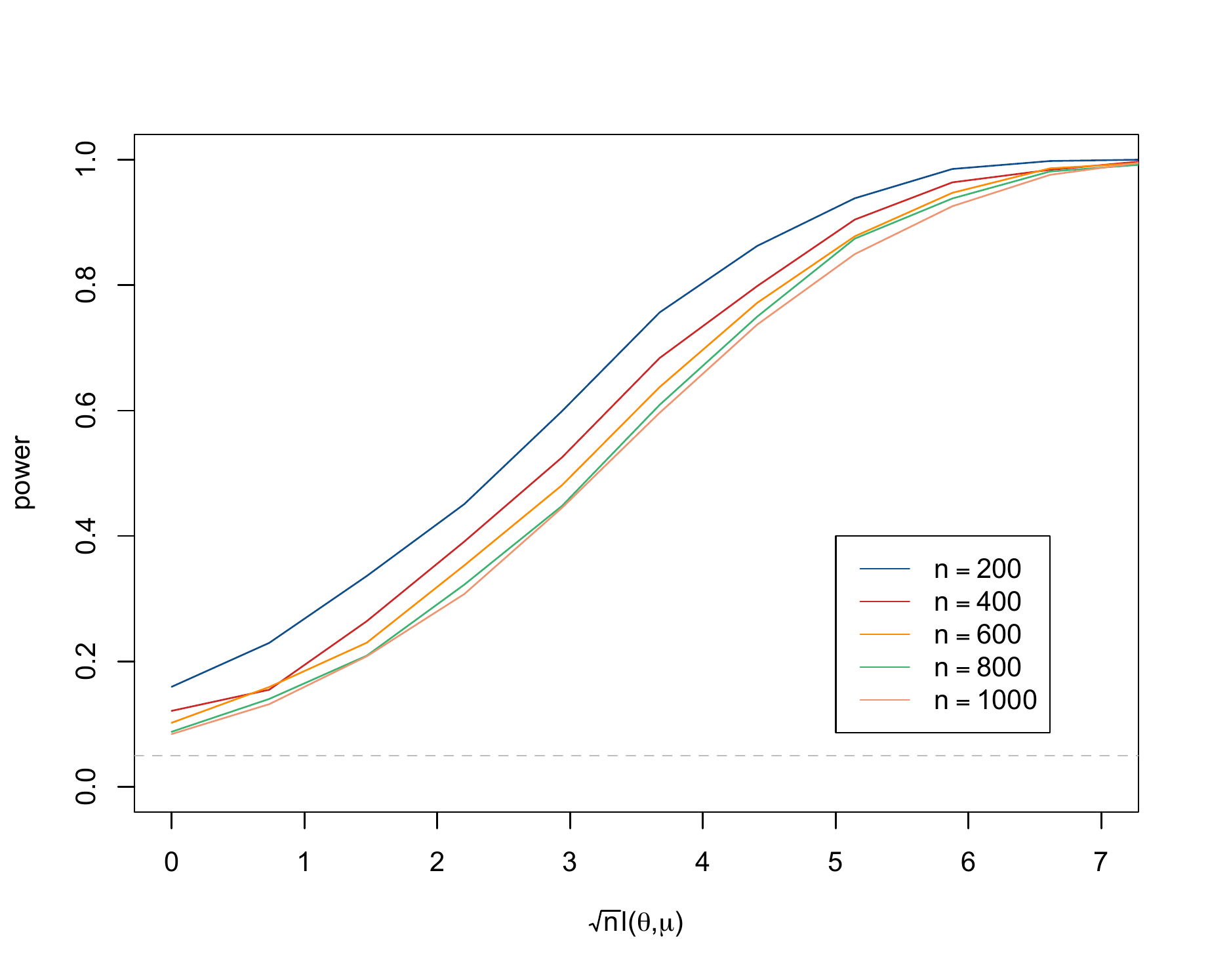}
\caption{Power Curve of Scenario 2}\label{fig:3}
\end{figure}
\begin{figure}
\centering
\includegraphics[width=10cm]{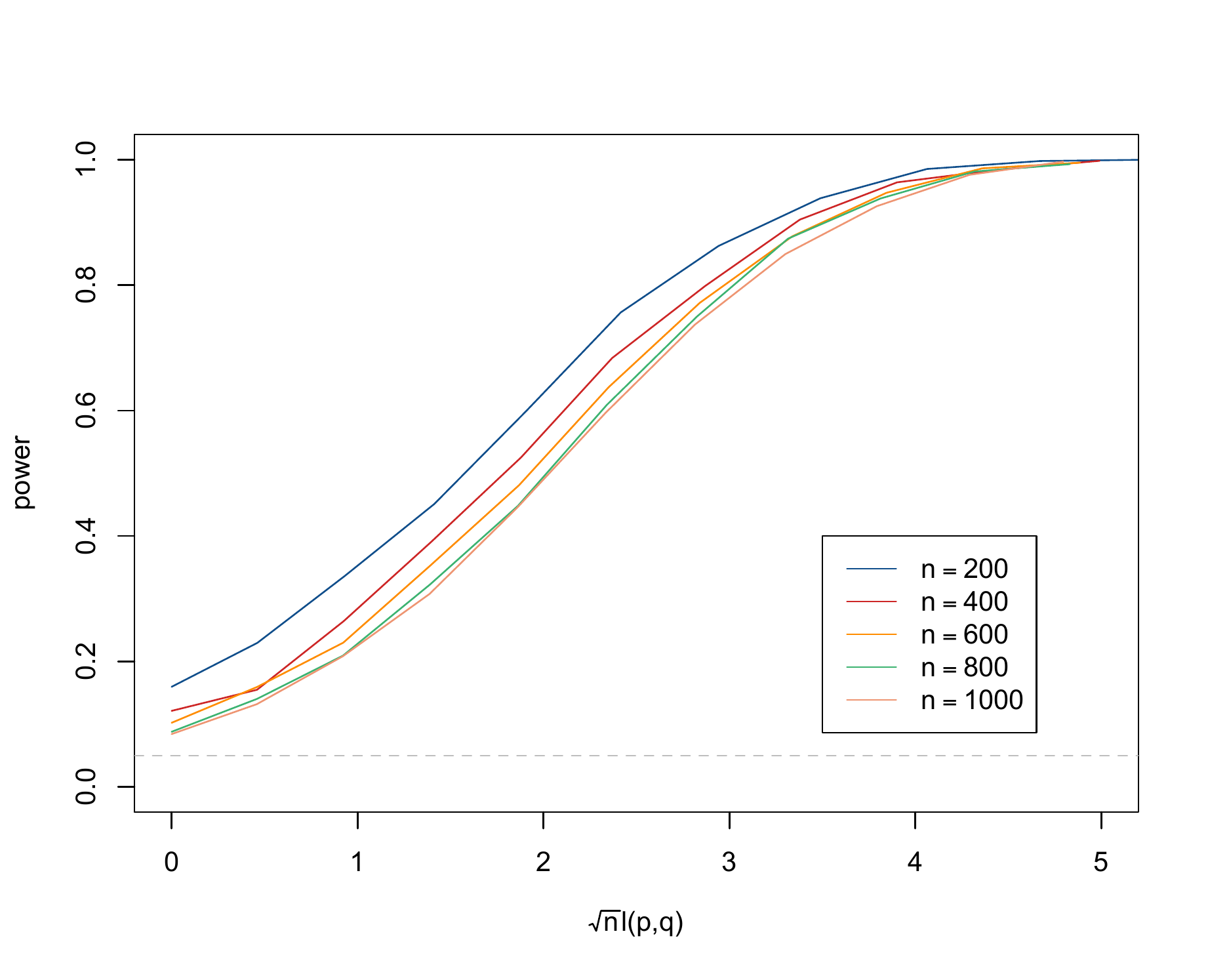}
\caption{Power Curve of Scenario 3}\label{fig:4}
\end{figure}
\begin{figure}
\centering
\includegraphics[width=10cm]{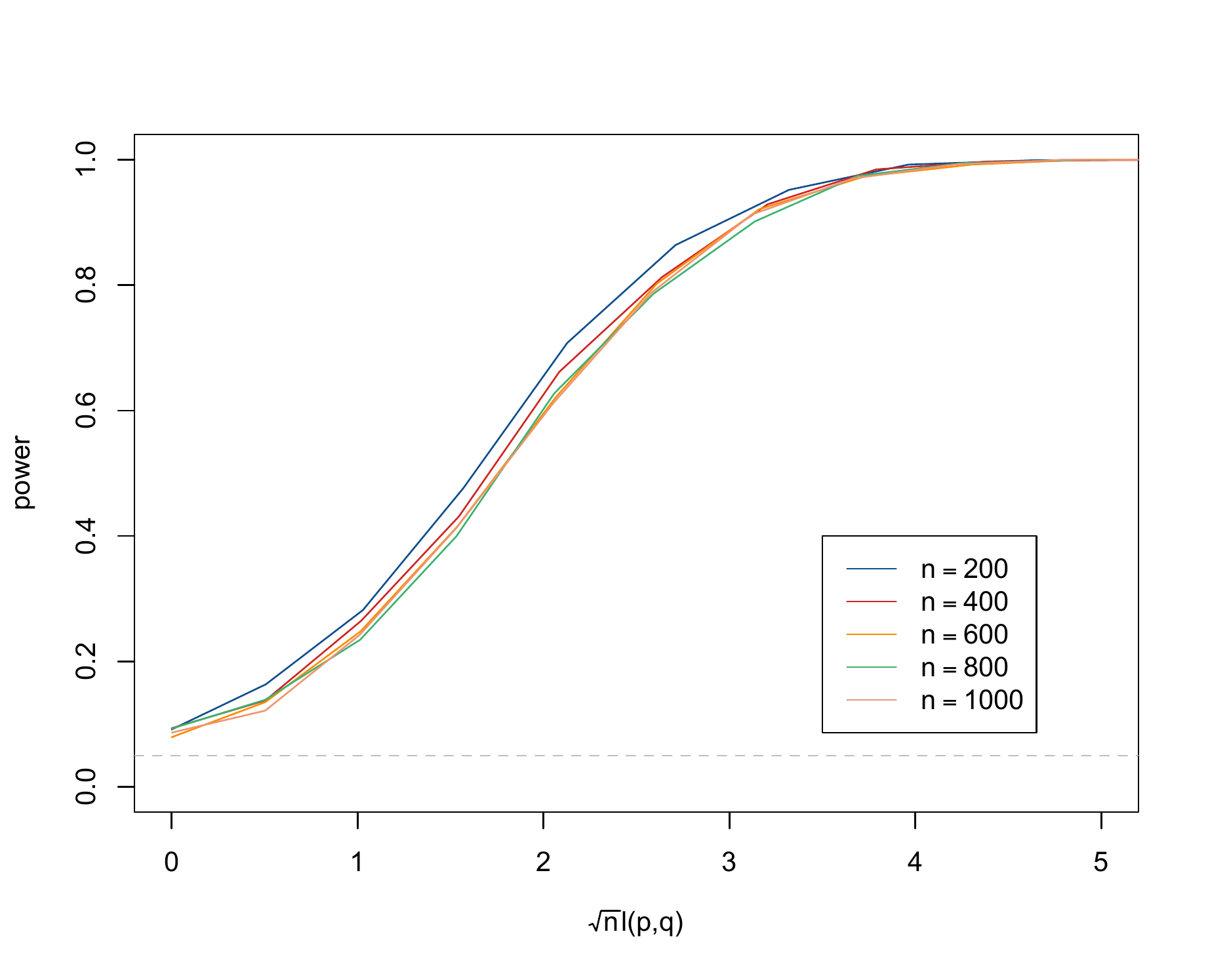}
\caption{Power Curve of Scenario 4}\label{fig:5}
\end{figure}
\begin{figure}
\centering
\includegraphics[width=10cm]{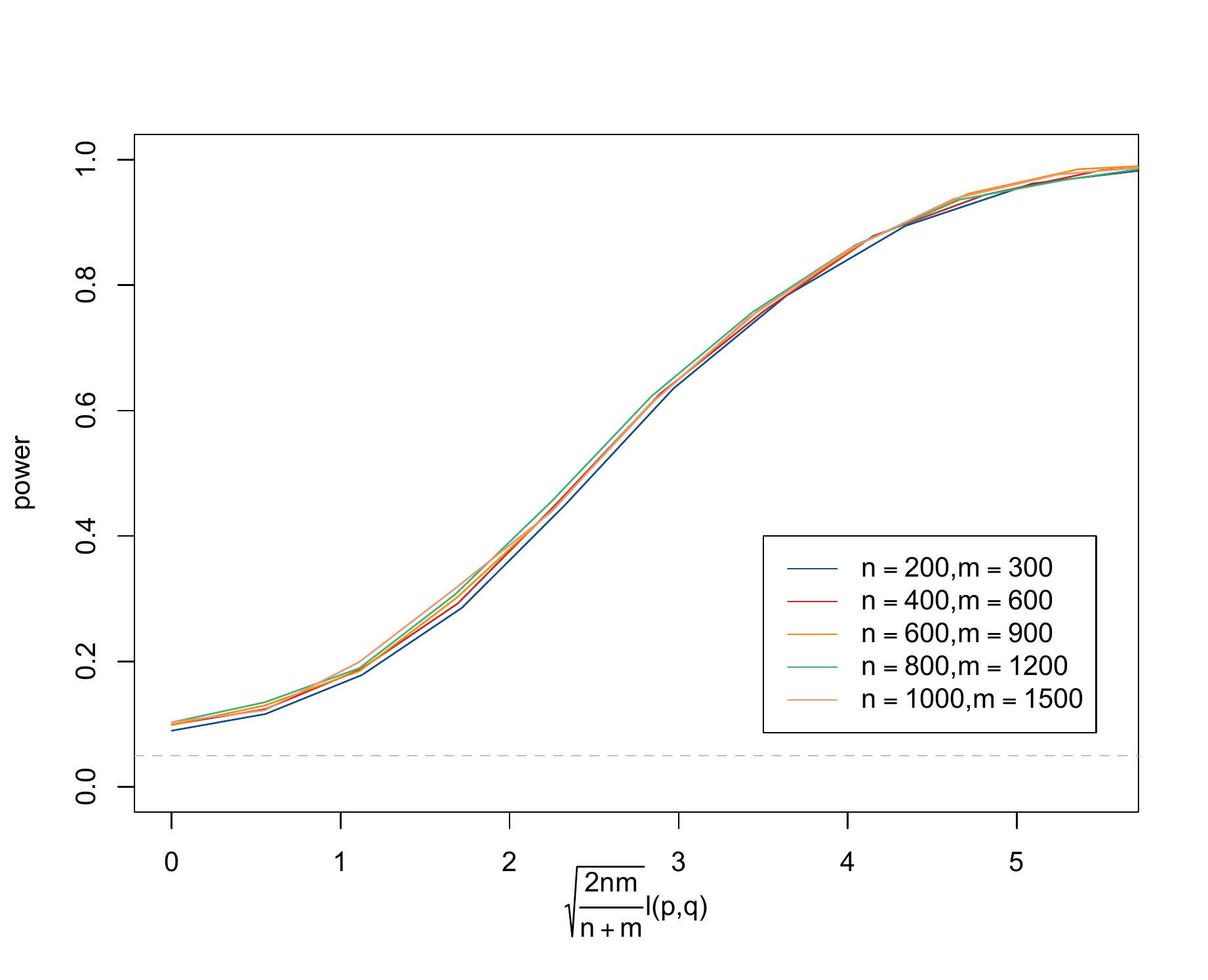}
\caption{Power Curve of Scenario 5}\label{fig:6}
\end{figure}

The numerical results of the five scenarios are summarized in Figures \ref{fig:2}-\ref{fig:6}. The power curves are plotted in the contiguous regimes where $\ell(\theta,\mu)=O(n^{-1/2})$ or $\ell(p,q)=O(n^{-1/2})$. The grey dashed lines correspond to the nominal $0.05$ level of the tests.

In Scenario 1, we vary $\theta$ in a local $O(n^{-1/2})$ neighborhood of the null hypothesis $\mu$. It is clear that the power function is increasing with respect to $\sqrt{n}\ell(\theta,\mu)$. Moreover, with different sample sizes, the curves match well with each other. This verifies the conclusion of Theorem \ref{thm:power-Gaussian} that the magnitude of $\sqrt{n}\ell(\theta,\mu)$ asymptotically determines the power of the test. We observe in Figure \ref{fig:2} that the power is very close to $1$ when $\sqrt{n}\ell(\theta,\mu)$ is greater than $6$. The value of the power at $\sqrt{n}\ell(\theta,\mu)=0$ corresponds to the Type-1 error in the  null hypothesis. The actually Type-1 error is slightly greater than the nominal $0.05$ level, but the approximations are reasonable for relatively large sample sizes.

Scenario 2 considers a harder null hypothesis with a degenerate $\mu=(1,3,3,3,5,5)$. A $0.05$-level test studied in Section \ref{sec:degen} requires both testing statistics $T_f$ and $T_g$. Similar to what is observed in Scenario 1, Figure \ref{fig:3} shows that the power approaches $1$ at about $\sqrt{n}\ell(\theta,\mu)=7$, which is predicted by Theorem \ref{thm:power-degen-Gaussian}. However, when $\sqrt{n}\ell(\theta,\mu)=0$, the actual Type-1 errors are consistently larger than the nominal level $0.05$, especially when the sample sizes are relatively small.

Scenario 3 and Scenario 4 consider categorical distributions with a non-degenerate null $q=(0.1,0.2,0.3,0.4)$ and a degenerate null $q=(0.1,0.1,0.4,0.4)$, respectively. Again, Theorem \ref{thm:power-Gaussian-cat} and Theorem \ref{thm:power-degen-cat} are verified by Figure \ref{fig:4} and Figure \ref{fig:5}. The actual Type-1 errors are also larger than the nominal one, and are closer to $0.05$ with larger sample sizes.

Finally in Scenario 5, we consider experiments of the two-sample test. According to the definitions of the testing statistics in (\ref{eq:rum}) and (\ref{eq:vodka}), it is $\sqrt{\frac{2nm}{n+m}}\ell(p,q)$ that determines the power function asymptotically. Figure \ref{fig:6} shows that different power curves well match each other as functions of $\sqrt{\frac{2nm}{n+m}}\ell(p,q)$. As is predicted by Theorem \ref{thm:two-sample-power}, the power is close to $1$ at a reasonably large value of $\sqrt{\frac{2nm}{n+m}}\ell(p,q)$.

A common theme in the above numerical results is that the actual Type-1 errors are always larger than the nominal one. We will give an explanation of this phenomenon in Section \ref{sec:disc}. Roughly speaking, whenever the null exhibits an ambiguous clustering structure, the asymptotic distribution of the testing statistic under the null is a noncentral chi-square distribution. Though a larger sample size helps to make the noncentrality parameter vanish (Figures \ref{fig:2}-\ref{fig:6}), it still results in an estimate of Type-1 error that is too optimistic with a finite sample size. There are potentially two ways to overcome this difficulty. One is to appeal to a second-order correction, and the other is to estimate the noncentrality parameter in the null distribution. We leave this interesting topic as a future project.

\section{Discussion and Future Directions}\label{sec:disc}

The testing procedures that we propose and analyze in this paper critically depend on the structure of null hypothesis. In Section \ref{sec:gauss}, the mean vector $(\mu_1,...,\mu_k)^T$ is assumed to consist of $k$ distinct numbers, and the testing statistic is constructed based on the functions $f_1,...,f_k$ defined in (\ref{eq:def-f}). In Section \ref{sec:degen}, we assume $(\mu_1,...,\mu_k)^T$ consists of $k$ numbers that take values in $\{\nu_1,...,\nu_d\}$ for some $d\leq k$. For this degenerate setting, we use the functions $f_1,...,f_d$ and $g$ defined in (\ref{eq:def-f-degen}) and (\ref{eq:def-g-degen}) to construct the testing statistics.

Much weaker assumptions are considered in Section \ref{sec:opt}. In Section \ref{sec:opt-gauss}, we allow $|\mu_j-\mu_l|$ to converge to $0$, but require the difference should be of a larger order than $n^{-1/2}$ for every $j\neq l$. This extends the assumption in Section \ref{sec:gauss} that $\mu_1,...,\mu_k$ are $k$ distinct numbers that do not vary with $n$. In Section \ref{sec:opt-gauss-degen}, we consider the setting where $|\nu_g-\nu_h|$ is of a larger order than $n^{-1/2}$ for every $g\neq h$, and $|\mu_j-\nu_h|$ is of a smaller order than $n^{-1/2}$ for every $j\in\mathcal{C}_h$. This setting extends the assumption used in Section \ref{sec:degen}. It turns out that the asymptotic distributions of the proposed testing statistics (Theorem \ref{thm:null-gauss} and Theorem \ref{thm:null-gauss-degen}) are still valid under these more general conditions (see Theorem \ref{thm:Gaussian} and Theorem \ref{thm:M2} in Section \ref{sec:pf-dist}).

However, the conditions in Section \ref{sec:opt-gauss} and Section \ref{sec:opt-gauss-degen} still do not cover all situations. By requiring the within-cluster distance to be of a smaller order than $n^{-1/2}$ and the between-cluster distance to be of a larger order than $n^{-1/2}$, the numbers $\mu_1,...,\mu_k$ enjoy an approximately exact clustering structure, because for each $j\neq l$, we either have $\sqrt{n}|\mu_j-\mu_l|\rightarrow 0$ or $\sqrt{n}|\mu_j-\mu_l|\rightarrow \infty$, depending on whether $j$ and $l$ are in the same cluster or not. A possible situation $\sqrt{n}|\mu_j-\mu_l|\asymp 1$ is excluded.

In this section, we discuss a situation where the clustering structure of the numbers $\mu_1,...,\mu_k$ is ambiguous. Consider a partition $\mathcal{C}_1,...,\mathcal{C}_d$ of $[k]$. Define $\nu_h=\frac{1}{|\mathcal{C}_h|}\sum_{j\in\mathcal{C}_g}\mu_j$. Instead of assuming the within-cluster distance is of a smaller order than $n^{-1/2}$, we consider the situation where $n\sum_{h=1}^d\sum_{j\in\mathcal{C}_h}(\mu_j-\nu_h)^2$ is of a constant order. Moreover, we also assume the between-cluster distance $|\nu_g-\nu_h|$ is of a larger order than $n^{-1/2}$ for every $g\neq h$. This is without loss of generality, because if there is some $g\neq h$, such that $|\nu_g-\nu_h|=O(n^{-1/2})$, then $\mathcal{C}_g$ and $\mathcal{C}_h$ can be combined into a single cluster. Recall the definition of $\bar{\eta}_{gh}$ in (\ref{eq:bar-eta}), we formalize this ambiguous clustering structure into the following condition.

\begin{m2'}
For the partition $\mathcal{C}_1,...,\mathcal{C}_d$ and clustering centers $\nu_1,...,\nu_d$ defined above,
assume $\lim_{n\rightarrow\infty}\max_{g\neq h}\frac{|\bar{\eta}_{gh}|}{\sqrt{n}}=0$, and $\tau^2=\lim_{n\rightarrow\infty}n\sum_{h=1}^d\sum_{j\in\mathcal{C}_h}(\mu_j-\nu_h)^2\in[0,\infty)$.
\end{m2'}

Note that Condition M2 is a special case of Condition M2' when $\tau^2=0$. The next theorem gives the asymptotic distribution of the testing statistics $T_f$ and $T_g$ defined in (\ref{eq:def-Tf}) and (\ref{eq:def-Tg}) under the null hypothesis.

\begin{thm}\label{thm:m2-performance}
Assume Condition M2' holds. Then we have $T_g\leadsto \chi_{k,\tau^2}^2$, $T_f\leadsto \chi_d^2$ and $T_g-T_f\leadsto \chi_{k-d,\tau^2}^2$ as $n\rightarrow\infty$ under the null hypothesis $X\sim N(\mu,n^{-1}I_k)$.
\end{thm}

It is interesting to see that the asymptotic distribution of $T_g$ is a noncentral chi-squared distribution even under the null hypothesis. The noncentrality parameter $\tau^2$ characterizes the within-cluster distance of $\mu_1,...,\mu_k$ with respect to the partition $\mathcal{C}_1,...,\mathcal{C}_d$. Theorem \ref{thm:m2-performance} is reduced to Theorem \ref{thm:null-gauss-degen} when $\tau^2=0$.

Define a number $\chi_{k,\tau^2}^2(\alpha)$ that satisfies $\mathbb{P}(\chi_{k,\tau^2}^2\leq \chi_{k,\tau^2}^2(\alpha))=1-\alpha$. Then, an $\alpha$-level testing function is $\phi_{\alpha}=\mathbb{I}\{T_g>\chi_{k,\tau^2}^2(\alpha)\}\vee\mathbb{I}\{T_f>\chi_{k,\tau^2}^2(\alpha)\}$. Compared with the null hypothesis where $\tau^2=0$, a nonzero $\tau^2$ requires a higher rejection level. This means the test will have less power under a contiguous alternative, compared with the situation where $\tau^2=0$. Suppose $\tilde{\tau}^2=\lim_{n\rightarrow\infty}n\sum_{h=1}^d\sum_{j\in\mathcal{C}_h}(\theta_j-\nu_h)^2\in(0,\infty)$. Then, one can also show that $T_g\leadsto \chi_{k,\tilde{\tau}^2}^2$ under the alternative $X\sim N(\theta,n^{-1}I_k)$. Therefore, the test $\phi_{\alpha}$ starts to have power when $\tilde{\tau}^2$ exceeds $\tau^2$. When $\tilde{\tau}^2$ is close to or even smaller then $\tau^2$, this test will not have any power under the alternative. On the other hand, outside of the contiguous regime where $\sqrt{n}\ell(\theta,\mu)\rightarrow\infty$, we must have $\tilde{\tau}^2=\infty$, and then the test will have asymptotic power $1$.

From what we have just discussed, we can see that the structure of $\mu_1,...,\mu_k$ plays a critical role on the solution of the problem. The discussion also applies to the case of categorical distributions and we can obtain similar conclusions there. Theorem \ref{thm:m2-performance} characterizes the asymptotic distribution of the testing statistics when $\mu_1,...,\mu_k$ exhibit an ambiguous clustering structure, right on the edge of degeneracy. This results in a non-trivial behavior of the power function. Exact characterization of optimality of the testing problem (as what we have done in Section \ref{sec:opt}) on the edge of degeneracy remains open, and we shall consider this problem as a future project.

Finally, we discussed a list of open problems that can be viewed as natural extensions of the results in the paper.
\begin{enumerate}
\item \textit{Growing or infinite support size.} The paper focuses on the case where $k$ is a fixed integer that does not depend on $n$. The case with a growing $k$ or even $k=\infty$ is of potential importance in many high-dimensional data analysis situations. This requires new techniques because for a probability vector $p=(p_1,...,p_k)$ with a growing or an infinite $k$, many $p_j$'s have extremely small values.
\item \textit{Testing a parametric family with permutation invariance.} An extension to the null hypothesis (\ref{eq:intro-null}) is
$$H_0: p(j)=f_{\lambda}(\pi(j))\quad\text{for some }\lambda\in\Lambda\text{ and some }\pi\in S_k.$$
Here, $\{f_{\lambda}(j)\}$ is a discrete distribution with an unknown parameter $\lambda\in\Lambda$. An example is $\text{Poisson}(\lambda)$. Without the permutation $\pi\in S_k$, the null hypothesis becomes $p=f_{\lambda}$ for some $\lambda\in\Lambda$, which is a classical goodness-of-fit test of a parametric family \cite{kulperger1982random,kocherlakota1986goodness}.
\item \textit{Non-asymptotic study of minimax separation.} This paper considers testing procedures that enjoy asymptotic optimality (Section \ref{sec:opt}). An important theoretical problem is to understand the minimax separation $\rho^*$ for which one can consistently test the null $\ell(p,q)=0$ against the alternative $\ell(p,q)>\rho$ if and only if $\rho>\rho^*$. With the permutation invariance, the null hypothesis is a non-convex set, which is in contrast to a convex case that was recently studied by \cite{blanchard2017minimax}.
\item \textit{Other group invariance.} Permutation invariance is a special case of group invariance. A more general question is to consider a null hypothesis that is invariant with respect to other group actions. A recent work \cite{perry2017sample} considered a group of cyclic shifts. It would be interesting to understand the method of invariance in a general group theoretic framework.
\end{enumerate}

\section{Proofs}\label{sec:proof}

In this section, we present the proofs of all results in the paper. In Section \ref{sec:pf-dist}, we derive the asymptotic distributions of the proposed testing statistics in various settings. These results are used to derive Theorem \ref{thm:null-gauss}, Theorem \ref{thm:null-gauss-degen}, Theorem \ref{thm:null-gauss-cat}, Theorem \ref{thm:null-cat-degen}, Theorem \ref{thm:two-sample-null} and Theorem \ref{thm:m2-performance}. Then, in Section \ref{sec:pf-power}, we analyze the powers of the proposed tests, which include the proofs of Theorem \ref{thm:power-Gaussian}, Theorem \ref{thm:power-degen-Gaussian}, Theorem \ref{thm:power-Gaussian-cat}, Theorem \ref{thm:power-degen-cat} and Theorem \ref{thm:two-sample-power}. Finally, in Section \ref{sec:pf-optimal}, we give proofs of all results in Section \ref{sec:opt}.

\subsection{Asymptotic Distribution of the Testing Statistics}\label{sec:pf-dist}

We first present and prove four theorems of the proposed testing statistics in various settings.

\begin{thm}\label{thm:Gaussian}
In addition to Condition M1,
assume
$$\lim_{n\rightarrow\infty}\sqrt{n}\ell(\theta,\mu)= \delta\in[0,\infty).$$
Then, as $n$ tends to infinity, $T\leadsto \chi_{k,\delta^2}^2$.
\end{thm}

\begin{proof}
We first calculate the derivatives of $f_l(t)$. The first derivative is
$$f_l'(t)=\frac{\prod_{j\in[k]\backslash\{l\}}(t-\mu_j)}{\prod_{j\in[k]\backslash\{l\}}(\mu_l-\mu_j)}.$$
Therefore, $f_l'(\mu_l)=1$. For any $j\neq l$, we give a bound for $\sup_{|t-\mu_j|\leq n^{-1/2}\epsilon}|f_l'(t)|$. The following inequality is useful.
\begin{equation}
|\eta_{lh}|+|\eta_{hl}|=\frac{1}{|\mu_l-\mu_h|}\left(\prod_{j\in[k]\backslash\{l,h\}}\left|\frac{\mu_l-\mu_j}{\mu_h-\mu_j}\right| + \prod_{j\in[k]\backslash\{l,h\}}\left|\frac{\mu_h-\mu_j}{\mu_l-\mu_j}\right|\right)\geq\frac{2}{|\mu_l-\mu_h|}.\label{eq:ineq-eta}
\end{equation}
Note that
\begin{eqnarray*}
|f_l'(t)| &=& \frac{|t-\mu_j|}{|\mu_l-\mu_j|}\left|\prod_{h\in[k]\backslash\{l,j\}}\frac{t-\mu_h}{\mu_l-\mu_h}\right| \\
&=& \frac{|t-\mu_j|}{|\mu_l-\mu_j|}\left|\prod_{h\in[k]\backslash\{l,j\}}\left(\frac{t-\mu_j}{\mu_l-\mu_h}+\frac{\mu_j-\mu_h}{\mu_l-\mu_h}\right)\right| \\
&\leq& \frac{|t-\mu_j|}{|\mu_l-\mu_j|}2^{k-2}\left(\left|\prod_{h\in[k]\backslash\{l,j\}}\left(\frac{t-\mu_j}{\mu_l-\mu_h}\right)\right|+\left|\prod_{h\in[k]\backslash\{l,j\}}\left(\frac{\mu_j-\mu_h}{\mu_l-\mu_h}\right)\right|\right) \\
&\leq& 2^{k-2}|t-\mu_j|^{k-1}\left(\frac{|\eta_{lj}|+|\eta_{jl}|}{2}\right)\prod_{h\in[k]\backslash\{l,j\}}\left(\frac{|\eta_{lh}|+|\eta_{hl}|}{2}\right) + 2^{k-2}|t-\mu_j||\eta_{lj}|.
\end{eqnarray*}
Therefore, we have the bound
\begin{equation}
\kappa_1(\epsilon)=\max_{j\neq l}\sup_{|t-\mu_j|\leq n^{-1/2}\epsilon}|f_l'(t)| \leq 2^{k-2}\left[\max_{j\neq l}\left(\frac{\epsilon|\eta_{jl}|}{\sqrt{n}}\right)^{k-2}+\max_{j\neq l}\left(\frac{\epsilon|\eta_{jl}|}{\sqrt{n}}\right)\right]. \label{eq:1std}
\end{equation}
The above bound is useful for $k\geq 3$. For $k=2$, it is easy to see
\begin{equation}
\kappa_1(\epsilon)=\max_{j\neq l}\sup_{|t-\mu_j|\leq n^{-1/2}\epsilon}|f_l'(t)| \leq\max_{j\neq l}\left(\frac{\epsilon|\eta_{jl}|}{\sqrt{n}}\right).\label{eq:1std-2}
\end{equation}

The second derivative of $f_l(t)$ is
$$f_l''(t)=\sum_{j\in[k]\backslash\{l\}}\frac{1}{(\mu_l-\mu_j)}\prod_{h\in[k]\backslash\{l,j\}}\frac{t-\mu_h}{\mu_l-\mu_h}.$$
We give a bound for $\sup_{|t-\mu_l|\leq n^{-1/2}\epsilon}|f_l''(t)|$. Similar calculation gives
\begin{equation}
\kappa_2(\epsilon)=\max_{1\leq l\leq k}\sup_{|t-\mu_l|\leq n^{-1/2}\epsilon}|f_l''(t)| \leq k\max_{j\neq l}|\eta_{jl}| \max_{j\neq l}\left(1+\frac{\epsilon|\eta_{jl}|}{\sqrt{n}}\right)^{k-2}. \label{eq:2ndd}
\end{equation}

Now we are ready to derive the asymptotic distribution of $T$. We write the observation as $X_j=\theta_j+n^{-1/2}Z_j$, with $Z_j\sim N(0,1)$ independently. The condition $\lim_{n\rightarrow\infty}\sqrt{n}\ell(\theta,\mu)= \delta$ implies that there is some $n_0$, such that for any $n>n_0$, we have
$$n\ell^2(\theta,\mu)\leq C_n\delta^2,$$
where $C_n$ is an sequence that tends to infinity arbitrarily slowly. In particular, we require that $C_n$ satisfies $C_n\rightarrow\infty$ and $\frac{C_n^{3/2}\max_{j\neq l}|\eta_{jl}|}{\sqrt{n}}\rightarrow 0$. The existence of such sequence $C_n$ is guaranteed by the assumption $\frac{\max_{j\neq l}|\eta_{jl}|}{\sqrt{n}}\rightarrow 0$. Thus, there exists a $\pi\in S_k$, possibly depending on $n$, such that
$$\max_{1\leq j\leq k}(\theta_j-\mu_{\pi(j)})^2\leq \frac{C_n\delta^2}{n}.$$
Since $k$ does not depend on $n$, $\max_{1\leq j\leq k}Z_j^2\leq C_n$ with probability that goes to $1$. By triangle inequality,
\begin{equation}
\max_{1\leq j\leq k}|X_j-\mu_{\pi(j)}|\leq \frac{\sqrt{C_n}(1+\sqrt{\delta^2})}{\sqrt{n}},\label{eq:bd-X}
\end{equation}
with probability that goes to $1$. We use Taylor expansion. For $j$ such that $\pi(j)=l$, we have
$$f_l(X_j)-f_l(\mu_{\pi(j)})=(X_j-\mu_{\pi(j)})+\frac{1}{2}f_l''(\xi_{jl})(X_j-\mu_{\pi(j)})^2,$$
where we have used the fact that $f'_l(\mu_l)=1$. For $j$ such that $\pi(j)\neq l$, we have
$$f_l(X_j)-f_l(\mu_{\pi(j)})=f_l'(\xi_{jl})(X_j-\mu_{\pi(j)}).$$
Therefore,
\begin{eqnarray*}
&& \left|\sum_{j=1}^kf_l(X_j)-\sum_{j=1}^kf_l(\mu_j)-(X_{\pi^{-1}(l)}-\mu_l)\right| \\
&\leq& \frac{1}{2}|f_l''(\xi_{\pi^{-1}(l)l})|(X_{\pi^{-1}(l)}-\mu_l)^2 + \sum_{j\neq \pi^{-1}(l)}|f_l'(\xi_{jl})||X_j-\mu_{\pi(j)}|.
\end{eqnarray*}
The number $\xi_{jl}$ is between $X_j$ and $\mu_{\pi(j)}$, which implies
\begin{equation}
\max_{j,l}|\xi_{jl}-\mu_{\pi(j)}|\leq \max_{1\leq j\leq k}|X_j-\mu_{\pi(j)}|\leq \frac{\sqrt{C_n}(1+\sqrt{\delta^2})}{\sqrt{n}}.\label{eq:bd-xi}
\end{equation}
Using the bounds (\ref{eq:1std}), (\ref{eq:1std-2}), (\ref{eq:2ndd}), (\ref{eq:bd-X}) and (\ref{eq:bd-xi}), we have
\begin{eqnarray}
\nonumber && \left|\sum_{j=1}^kf_l(X_j)-\sum_{j=1}^kf_l(\mu_j)-(X_{\pi^{-1}(l)}-\mu_l)\right| \\
\label{eq:bd*} &\leq& \frac{1}{2}\frac{C_n(1+\sqrt{\delta^2})^2}{n}\kappa_2\left(\sqrt{C_n}(1+\sqrt{\delta^2})\right) \\
\nonumber && + (k-1)\kappa_1\left(\sqrt{C_n}(1+\sqrt{\delta^2})\right)\frac{\sqrt{C_n}(1+\sqrt{\delta^2})}{\sqrt{n}}.
\end{eqnarray}
Therefore,
\begin{eqnarray*}
&& \left|T-n\sum_{l=1}^k(X_{\pi^{-1}(l)}-\mu_l)^2\right| \\
&\leq& n\sum_{l=1}^k\left|\left(\sum_{j=1}^kf_l(X_j)-\sum_{j=1}^kf_l(\mu_j)\right)^2-(X_{\pi^{-1}(l)}-\mu_l)^2\right| \\
&\leq& 2n\sum_{l=1}^k|X_{\pi^{-1}(l)}-\mu_l|\left|\sum_{j=1}^kf_l(X_j)-\sum_{j=1}^kf_l(\mu_j)-(X_{\pi^{-1}(l)}-\mu_l)\right| \\
&& + n\sum_{l=1}^k\left|\sum_{j=1}^kf_l(X_j)-\sum_{j=1}^kf_l(\mu_j)-(X_{\pi^{-1}(l)}-\mu_l)\right|^2 \\
&\leq& 2k\sqrt{n}\sqrt{C_n}(1+\sqrt{\delta^2})\left|\sum_{j=1}^kf_l(X_j)-\sum_{j=1}^kf_l(\mu_j)-(X_{\pi^{-1}(l)}-\mu_l)\right| \\
&& + kn\left|\sum_{j=1}^kf_l(X_j)-\sum_{j=1}^kf_l(\mu_j)-(X_{\pi^{-1}(l)}-\mu_l)\right|^2.
\end{eqnarray*}
By (\ref{eq:bd*}), the bound for $\left|T-n\sum_{l=1}^k(X_{\pi^{-1}(l)}-\mu_l)^2\right|$ is of order $\frac{C_n^{3/2}\max_{j\neq l}|\eta_{jl}|}{\sqrt{n}}\rightarrow 0$. Finally, it is easy to see that
$$n\sum_{l=1}^k(X_{\pi^{-1}(l)}-\mu_l)^2\sim \chi_{k,\delta_n^2}^2,$$
where $\delta_n^2=n\|\theta-\mu_{\pi}\|^2\rightarrow\delta^2$. Therefore, $T$ converges to $\chi_{k,\delta^2}^2$ in distribution.
\end{proof}

\begin{thm}\label{thm:M2}
In addition to Condition M2,
assume
$$\lim_{n\rightarrow\infty}\sqrt{n}\ell(\theta,\mu)= \delta\in[0,\infty).$$
Then, as $n$ tends to infinity, $T_g\leadsto \chi_{k,\delta^2}^2$, and $T_g\geq T_f$ in probability. Moreover, if $\delta^2=0$,
we have $T_g\leadsto\chi_k^2$, $T_f\leadsto \chi_{d}^2$ and $T_g-T_f\leadsto \chi_{k-d}^2$.
\end{thm}
\begin{proof}
The case $d=1$ is obvious. We only prove the case $d\geq 2$.
Similar to the inequality (\ref{eq:ineq-eta}), we have $|\bar{\eta}_{gh}|+|\bar{\eta}_{hg}|\geq\frac{2}{|\nu_g-\nu_h|}$. By Condition M2, we have
$$\frac{\max_{1\leq g\leq d}\max_{j\in\mathcal{C}_g}|\mu_j-\nu_g|}{\min_{g\neq h}|\nu_g-\nu_h|}=o(1).$$
The observation is $X_j=\theta_j+n^{-1/2}Z_j$ with $Z_j\sim N(0,1)$ independently. Use the notation $L=\max_{1\leq g\leq d}\max_{j\in\mathcal{C}_g}\sqrt{n}|\mu_j-\nu_g|=o(1)$. Under the assumption of the theorem, there exists a sequence $C_n$ that satisfies $C_n\rightarrow\infty$, $C_n^2L\rightarrow 0$ and $\frac{C_n^{3/2}\max_{g\neq h}|\bar{\eta}_{gh}|}{\sqrt{n}}\rightarrow 0$, such that $\max_{1\leq j\leq k}Z_j^2\leq C_n$ with probability tending to $1$. Similar to the bound (\ref{eq:bd-X}), the assumption $\lim_{n\rightarrow\infty}\sqrt{n}\ell(\theta,\mu)= \delta<\infty$ implies the existence of $\pi\in S_k$ such that
$$\max_{1\leq j\leq k}|X_j-\mu_{\pi(j)}|\leq \frac{\sqrt{C_n}(1+\sqrt{\delta^2})}{\sqrt{n}}.$$
We first study the asymptotic distribution of $T_g$. Note that
\begin{equation}
\max_{1\leq g\leq d}\max_{j\in\mathcal{C}_g}|X_{\pi^{-1}(j)}-\nu_g|\leq \frac{\sqrt{C_n}(1+\sqrt{\delta^2})+L}{\sqrt{n}}.\label{eq:ensiferum}
\end{equation}
Together with Condition M2 and the choice of $C_n$, we can immediately deduce
$$\frac{\max_{1\leq g\leq d}\max_{j\in\mathcal{C}_g}|X_{\pi^{-1}(j)}-\nu_g|}{\min_{g\neq h}|\nu_g-\nu_h|}\leq \max_{g\neq h}|\bar{\eta}_{gh}|\frac{\sqrt{C_n}(1+\sqrt{\delta^2})+L}{\sqrt{n}}=o(1).$$
The function $g(t)$ can be written as
$$\frac{1}{g(t)}=\sum_{g=1}^d\frac{1}{(t-\nu_g)^2}.$$
For each $j\in\mathcal{C}_g$, we have
$$\frac{(X_{\pi^{-1}(j)}-\nu_g)^2}{g(X_{\pi^{-1}(j)})}=1+\sum_{h\in[d]\backslash\{g\}}\frac{(X_{\pi^{-1}(j)}-\nu_g)^2}{(X_{\pi^{-1}(j)}-\nu_h)^2}.$$
Thus,
\begin{equation}
\left|\frac{g(X_{\pi^{-1}(j)})}{(X_{\pi^{-1}(j)}-\nu_g)^2}-1\right|\leq \frac{\sum_{h\in[d]\backslash\{g\}}\frac{(X_{\pi^{-1}(j)}-\nu_g)^2}{(X_{\pi^{-1}(j)}-\nu_h)^2}}{1+\sum_{h\in[d]\backslash\{g\}}\frac{(X_{\pi^{-1}(j)}-\nu_g)^2}{(X_{\pi^{-1}(j)}-\nu_h)^2}}\leq \sum_{h\in[d]\backslash\{g\}}\frac{(X_{\pi^{-1}(j)}-\nu_g)^2}{(X_{\pi^{-1}(j)}-\nu_h)^2},\label{eq:ignore-one}
\end{equation}
where the bound on the right hand side above can be bounded by
$$\sum_{h\in[d]\backslash\{g\}}\frac{2(X_{\pi^{-1}(j)}-\nu_g)^2}{(\nu_g-\nu_h)^2-2(X_{\pi^{-1}(j)}-\nu_h)^2}\leq 4d\max_{g\neq h}|\bar{\eta}_{gh}|\frac{\sqrt{C_n}(1+\sqrt{\delta^2})+L}{\sqrt{n}}.$$
Together with (\ref{eq:ensiferum}), we have
\begin{eqnarray*}
&& |g(X_{\pi^{-1}(j)})-(X_{\pi^{-1}(j)}-\nu_g)^2| \\
&=& (X_{\pi^{-1}(j)}-\nu_g)^2\left|\frac{g(X_{\pi^{-1}(j)})}{(X_{\pi^{-1}(j)}-\nu_g)^2}-1\right| \\
&\leq& 4d\max_{g\neq h}|\bar{\eta}_{gh}|\left(\frac{\sqrt{C_n}(1+\sqrt{\delta^2})+L}{\sqrt{n}}\right)^3.
\end{eqnarray*}
Therefore
\begin{eqnarray}
\nonumber && \left|T_g-n\sum_{h=1}^d\sum_{j\in\mathcal{C}_h}(X_{\pi^{-1}(j)}-\nu_h)^2\right| \\
\nonumber &\leq& n\sum_{h=1}^d\sum_{j\in\mathcal{C}_h}\left|g(X_{\pi^{-1}(j)})-(X_{\pi^{-1}(j)}-\nu_g)^2\right| \\
\label{eq:nice-analysis} &\leq& 4kd\max_{g\neq h}|\bar{\eta}_{gh}|\frac{\left(\sqrt{C_n}(1+\sqrt{\delta^2})+L\right)^{3}}{\sqrt{n}}=o(1).
\end{eqnarray}
For each $j\in\mathcal{C}_h$,
\begin{eqnarray*}
&& n|(X_{\pi^{-1}(j)}-\nu_h)^2-(X_{\pi^{-1}(j)}-\mu_j)^2| \\
&\leq & n|\nu_h-\mu_j||X_{\pi^{-1}(j)}-\nu_h+X_{\pi^{-1}(j)}-\mu_j| \\
&\leq& L\left(2\sqrt{C_n}(1+\sqrt{\delta^2})+L\right)=o(1).
\end{eqnarray*}
Thus, 
\begin{equation}
\left|T_g-n\sum_{h=1}^d\sum_{j\in\mathcal{C}_h}(X_{\pi^{-1}(j)}-\mu_j)^2\right|\label{eq:bd-Tg}
\end{equation}
has a bound that tends to $0$.
Observe that
$$n\sum_{h=1}^d\sum_{j\in\mathcal{C}_h}(X_{\pi^{-1}(j)}-\mu_j)^2\sim \chi_{k,\delta_n^2}^2,$$
where
$$\delta_n^2=n\sum_{j=1}^k(\theta_j-\mu_{\pi(j)})^2\leadsto \delta^2.$$
Thus, $T_g\leadsto \chi_{k,\delta^2}^2$.

Next we derive the asymptotic distribution of $T_f$. Similar to (\ref{eq:1std}), (\ref{eq:1std-2}) and (\ref{eq:2ndd}), we also have
\begin{equation}
\kappa_1(\epsilon)=\max_{g\neq h}\sup_{|t-\nu_g|\leq n^{-1/2}\epsilon}|f_h'(t)| \leq 2^{d-2}\left[\max_{g\neq h}\left(\frac{\epsilon|\bar{\eta}_{gh}|}{\sqrt{n}}\right)^{d-2}+\max_{g\neq h}\left(\frac{\epsilon|\bar{\eta}_{gh}|}{\sqrt{n}}\right)\right], \label{eq:1std-}
\end{equation}
for $d\geq 3$,
\begin{equation}
\kappa_1(\epsilon)=\max_{g\neq h}\sup_{|t-\nu_g|\leq n^{-1/2}\epsilon}|f_h'(t)| \leq \max_{g\neq h}\left(\frac{\epsilon|\bar{\eta}_{gh}|}{\sqrt{n}}\right),\label{eq:1std-2-}
\end{equation}
for $d=2$,
and
\begin{equation}
\kappa_2(\epsilon)=\max_{1\leq h\leq d}\sup_{|t-\nu_h|\leq n^{-1/2}\epsilon}|f_h''(t)| \leq d\max_{g\neq h}|\bar{\eta}_{gh}| \max_{g\neq h}\left(1+\frac{\epsilon|\bar{\eta}_{gh}|}{\sqrt{n}}\right)^{d-2}. \label{eq:2ndd-}
\end{equation}
For any $j\in\mathcal{C}_g$,
\begin{eqnarray*}
&& f_g(X_{\pi^{-1}(j)}) - f_g(\mu_j) \\
 &=& f_g(X_{\pi^{-1}(j)}) - f_g(\nu_g) + f_g(\nu_g) - f_g(\mu_j) \\
&=& f_g'(\nu_g)(X_{\pi^{-1}(j)}-\mu_j) + \frac{1}{2}f_g''(\xi_{jg})(X_{\pi^{-1}(j)}-\nu_g)^2 - \frac{1}{2}f_g''(\bar{\xi}_{jg})(\mu_j-\nu_g)^2.
\end{eqnarray*}
For any $j\in\mathcal{C}_h$ with any $h\neq g$,
$$f_g(X_{\pi^{-1}(j)}) - f_g(\mu_j)=f_g'(\tilde{\xi}_{jg})(X_{\pi^{-1}(j)}-\mu_j).$$
By the fact that $f_g'(\nu_g)=1$, we have
\begin{eqnarray*}
&& \left|\sum_{j=1}^kf_g(X_j)-\sum_{j=1}^kf_g(\mu_j)-\sum_{j\in\mathcal{C}_g}(X_{\pi^{-1}(j)}-\mu_j)\right| \\
&\leq& \frac{1}{2}\sum_{j\in\mathcal{C}_g}|f_g''(\xi_{jg})|(X_{\pi^{-1}(j)}-\nu_g)^2 + \frac{1}{2}\sum_{j\in\mathcal{C}_g}|f_g''(\bar{\xi}_{jg})|(\mu_j-\nu_g)^2 \\
&& \sum_{h\in[d]\backslash\{g\}}\sum_{j\in\mathcal{C}_h}|f_g'(\tilde{\xi}_{jg})||X_{\pi^{-1}(j)}-\mu_j|.
\end{eqnarray*}
The number $\xi_{jg}$ is between $X_{\pi^{-1}(j)}$ and $\nu_g$, the number $\bar{\xi}_{jg}$ is between $\mu_j$ and $\nu_g$, and the number $\tilde{\xi}_{jg}$ is between $X_{\pi^{-1}(j)}$ and $\mu_j$. Thus,
$$|\xi_{jg}-\nu_g|\leq |X_{\pi^{-1}(j)}-\nu_g|\leq \frac{\sqrt{C_n}(1+\sqrt{\delta^2})+L}{\sqrt{n}},$$
$$|\bar{\xi}_{jg}-\nu_g|\leq |\mu_j-\nu_g|\leq \frac{L}{\sqrt{n}},$$
and
$$|\tilde{\xi}_{jg}-\mu_j|\leq |X_{\pi^{-1}(j)}-\mu_j|\leq \frac{\sqrt{C_n}(1+\sqrt{\delta^2})}{\sqrt{n}}.$$
Using the bounds (\ref{eq:1std-}), (\ref{eq:1std-2-}) and (\ref{eq:2ndd-}), we can deduce
\begin{eqnarray*}
&& \left|\sum_{j=1}^kf_g(X_j)-\sum_{j=1}^kf_g(\mu_j)-\sum_{j\in\mathcal{C}_g}(X_{\pi^{-1}(j)}-\mu_j)\right| \\
&\leq& k\frac{(\sqrt{C_n}(1+\sqrt{\delta^2})+L)^2}{n}\kappa_2\left(\sqrt{C_n}(1+\sqrt{\delta^2})+L\right) \\
\nonumber && + k\kappa_1\left(\sqrt{C_n}(1+\sqrt{\delta^2})\right)\frac{\sqrt{C_n}(1+\sqrt{\delta^2})}{\sqrt{n}}.
\end{eqnarray*}
Similar to the proof of Theorem (\ref{eq:ineq-eta}), we can show that
\begin{equation}
\left|T_f-n\sum_{h=1}^d\frac{1}{|\mathcal{C}_h|}\left(\sum_{j\in\mathcal{C}_h}(X_{\pi^{-1}(j)}-\mu_j)\right)^2\right|\label{eq:bd-Tf}
\end{equation}
has a bound of order  $\frac{C_n^{3/2}\max_{j\neq l}|\eta_{jl}|}{\sqrt{n}}\rightarrow 0$. Note that when $\delta^2=0$,
$$n\sum_{h=1}^d\frac{1}{|\mathcal{C}_h|}\left(\sum_{j\in\mathcal{C}_h}(X_{\pi^{-1}(j)}-\mu_j)\right)^2\sim \chi_{d}^2.$$
Thus, $T_f\leadsto\chi_{d}^2$.

Finally, we derive the asymptotic distribution for $T_g-T_f$. The bounds for (\ref{eq:bd-Tg}) and (\ref{eq:bd-Tf}) imply that
$$\left|T_g-T_f-n\sum_{h=1}^d\sum_{j\in\mathcal{C}_h}(X_{\pi^{-1}(j)}-\mu_j)^2+n\sum_{h=1}^d\frac{1}{|\mathcal{C}_h|}\left(\sum_{j\in\mathcal{C}_h}(X_{\pi^{-1}(j)}-\mu_j)\right)^2\right|$$
has a bound that tends to zero.
Thus, the asymptotic distribution of $T_g-T_f$ is the same as that of
\begin{eqnarray*}
&& n\sum_{h=1}^d\sum_{j\in\mathcal{C}_h}(X_{\pi^{-1}(j)}-\mu_j)^2 - n\sum_{h=1}^d\frac{1}{|\mathcal{C}_h|}\left(\sum_{j\in\mathcal{C}_h}(X_{\pi^{-1}(j)}-\mu_j)\right)^2 \\
&=& n\sum_{h=1}^d\sum_{j\in\mathcal{C}_h}\left(X_{\pi^{-1}(j)}-\frac{1}{|\mathcal{C}_h|}\sum_{j\in\mathcal{C}_h}X_{\pi^{-1}(j)}-\left(\mu_j-\frac{1}{|\mathcal{C}_h|}\sum_{j\in\mathcal{C}_h}\mu_j\right)\right)^2,
\end{eqnarray*}
which is $\chi_{k-d}^2$ when $\delta^2=0$. Therefore, $T_g-T_f\leadsto \chi_{k-d}^2$. Without the condition $\delta^2=0$, we can still claim $T_g\geq T_f$ in probability.
\end{proof}

\begin{thm}\label{thm:M3}
For $\pi=\argmin_{\pi\in S_k}\|\sqrt{p}-\sqrt{q_{\pi}}\|$, define
\begin{equation}
\delta_1^2=4n\sum_{l=1}^k(1-p_l)\left(\sqrt{p_l}-\sqrt{q_{\pi(l)}}\right)^2,\label{eq:def-delta1}
\end{equation}
and
\begin{equation}
\delta_2^2=4n\sum_{l=1}^kp_l\left(\sqrt{p_l}-\sqrt{q_{\pi(l)}}\right)^2.\label{eq:def-delta2}
\end{equation}
Assume $\limsup_{n\rightarrow\infty}(\delta_1^2+\delta_2^2)<\infty$.
Then, under Condition M3, $T-\delta_2^2\leadsto \chi_{k-1,\delta_1^2}^2$, as $n$ tends to infinity.
\end{thm}
\begin{proof}
The proof is almost the same as that of Theorem \ref{thm:Gaussian}, and therefore we will omit some overlapping details. Largely speaking, we can replace the $t, \mu_j, \theta_j, X_j$ by $\sqrt{t}, \sqrt{q_j}, \sqrt{p_j}, \sqrt{\hat{p}_j}$, and most parts in the proof of Theorem \ref{thm:Gaussian} will go through. Here are a few different details. We write $\sqrt{\hat{p}_j}=\sqrt{p}_j+n^{-1/2}Z_j/2$, with $Z_j=2\sqrt{n}(\sqrt{\hat{p}_j}-\sqrt{p}_j)$. Condition M3 implies that $\max_{1\leq j\leq k}Z_j^2=O_P(1)$. Thus, the inequality (\ref{eq:bd-X}) in the proof of Theorem \ref{thm:Gaussian} can be replaced by $\max_{1\leq j\leq k}|\sqrt{\hat{p}_j}-\sqrt{q_{\pi(j)}}|\leq \frac{\sqrt{C_n}(1+\sqrt{\delta^2})}{\sqrt{n}}$. Then, following the same argument in the proof of Theorem \ref{thm:Gaussian}, we have
$$\left|T-4n\sum_{l=1}^k(\sqrt{\hat{p}_{l}}-\sqrt{q_{\pi(l)}})^2\right|=o_P(1),$$
and it is sufficient to study the asymptotic distribution of $4n\sum_{l=1}^k(\sqrt{\hat{p}_{l}}-\sqrt{q_{\pi(l)}})^2$. Let $\Delta$ be a vector with the $l$th entry being $2\sqrt{n}(\sqrt{\hat{p}_{l}}-\sqrt{q_{\pi(l)}})$. Then, we have $4n\sum_{l=1}^k(\sqrt{\hat{p}_{l}}-\sqrt{q_{\pi(l)}})^2=\|Z+\Delta\|^2$. Under Condition M3, $Z\leadsto N(0,I_k-\sqrt{p}\sqrt{p}^T)$ by Lindeberg's central limit theorem together with an argument of delta's method. Therefore, there exists a random vector $W$ that satisfies $W\leadsto N(0,I_k)$ and $Z=(I_k-\sqrt{p}\sqrt{p}^T)W$. This gives
\begin{eqnarray*}
 \|Z+\Delta\|^2 &=& \|(I_k-\sqrt{p}\sqrt{p}^T)W+(I_k-\sqrt{p}\sqrt{p}^T)\Delta+\sqrt{p}\sqrt{p}^T\Delta\|^2 \\
 &=& \|(I_k-\sqrt{p}\sqrt{p}^T)W+(I_k-\sqrt{p}\sqrt{p}^T)\Delta\|^2 + \|\sqrt{p}\sqrt{p}^T\Delta\|^2,
\end{eqnarray*}
where $\|(I_k-\sqrt{p}\sqrt{p}^T)W+(I_k-\sqrt{p}\sqrt{p}^T)\Delta\|^2\leadsto \chi_{k-1,\delta_1^2}^2$ and $\|\sqrt{p}\sqrt{p}^T\Delta\|^2=\delta_2^2$.
\end{proof}

\begin{thm}\label{thm:M4}
For $\pi=\argmin_{\pi\in S_k}\|\sqrt{p}-\sqrt{q_{\pi}}\|$, define
$$\delta_1^2=4n\sum_{l=1}^k(1-p_l)\left(\sqrt{p_l}-\sqrt{q_{\pi(l)}}\right)^2,$$
and
$$\delta_2^2=4n\sum_{l=1}^kp_l\left(\sqrt{p_l}-\sqrt{q_{\pi(l)}}\right)^2.$$
Assume $\limsup_{n\rightarrow\infty}(\delta_1^2+\delta_2^2)<\infty$.
Then, under Condition M4, $T_g-\delta_2^2\leadsto \chi_{k-1,\delta_1^2}^2$, as $n$ tends to infinity. Moreover, $T_g\geq T_f$ in probability. Furthermore, when $\delta_1^2+\delta_2^2=0$, $T_g\leadsto \chi_{k-1}^2$, $T_f\leadsto\chi_{d-1}^2$ and $T_g-T_f\leadsto \chi_{k-d}^2$.
\end{thm}
\begin{proof}
The proof is largely the same as that of Theorem \ref{thm:M3}. We only need to replace the $t, \mu_j, \theta_j, \nu_h, X_j$ in the proof of Theorem \ref{thm:M3} by $\sqrt{t}, \sqrt{q_j}, \sqrt{p_j}, \sqrt{r_h}, \sqrt{\hat{p}_j}$. Then, by the same argument, we have
$$\left|T_g-4n\sum_{h=1}^d\sum_{j\in\mathcal{C}_h}(\sqrt{\hat{p}_{\pi^{-1}(j)}}-\sqrt{q_j})^2\right|=o_P(1),$$
and
$$\left|T_g-T_f-4n\sum_{h=1}^d\sum_{j\in\mathcal{C}_h}(\sqrt{\hat{p}_{\pi^{-1}(j)}}-\sqrt{q_j})^2+4n\sum_{h=1}^d\frac{1}{|\mathcal{C}_h|}\left(\sum_{j\in\mathcal{C}_h}(\sqrt{\hat{p}_{\pi^{-1}(j)}}-\sqrt{q_j})\right)^2\right|=o_P(1).$$
The same argument in the proof of Theorem \ref{thm:M3} implies that $T_g-\delta_2^2\leadsto \chi_{k-1,\delta_1^2}^2$. The conclusion $T_g\geq T_f$ in probability can be deduced by
\begin{eqnarray*}
&& 4n\sum_{h=1}^d\sum_{j\in\mathcal{C}_h}(\sqrt{\hat{p}_{\pi^{-1}(j)}}-\sqrt{q_j})^2-4n\sum_{h=1}^d\frac{1}{|\mathcal{C}_h|}\left(\sum_{j\in\mathcal{C}_h}(\sqrt{\hat{p}_{\pi^{-1}(j)}}-\sqrt{q_j})\right)^2 \\
&=& 4n\sum_{h=1}^d\sum_{j\in\mathcal{C}_h}\left(\sqrt{\hat{p}_{\pi^{-1}(j)}}-\sqrt{q_j}-\frac{1}{|\mathcal{C}_h|}\sum_{j\in\mathcal{C}_h}(\sqrt{\hat{p}_{\pi^{-1}(j)}}-\sqrt{q_j})\right)^2 \geq 0.
\end{eqnarray*}

Now we derive the results under the null distribution.
Recall the definition of $Z_j$ in the proof of Theorem \ref{thm:M3}. The asymptotic distributions of $T_g$, $T_f$ and $T_g-T_f$ are the same of those of
$$\sum_{j=1}^kZ_j^2,\quad \sum_{h=1}^d\frac{1}{|\mathcal{C}_h|}\left(\sum_{j\in\mathcal{C}_h}Z_j\right)^2,\quad \sum_{j=1}^kZ_j^2-\sum_{h=1}^d\frac{1}{|\mathcal{C}_h|}\left(\sum_{j\in\mathcal{C}_h}Z_j\right)^2,$$
respectively under the null hypothesis.
According to the argument in the proof of Theorem \ref{thm:M3}, $Z=(I_k-\sqrt{q}\sqrt{q}^T)W$ with $W\leadsto N(0,I_k)$. Therefore, $\sum_{j=1}^kZ_j^2\leadsto \chi_{k-1}^2$.

Define a $k\times d$ matrix $Q$ with $Q_{jh}=\frac{1}{\sqrt{|\mathcal{C}_h|}}$ if $j\in\mathcal{C}_h$ and $Q_{jh}=0$ if $j\notin \mathcal{C}_h$. It is easy to see that $QQ^T$ is a projection matrix and $Q^TQ=I_d$. Define a vector $\gamma\in\mathbb{R}^d$ whose $h$th entry is $\gamma_h=\sqrt{|\mathcal{C}_h|r_h}$. It is easy to see that $\gamma$ is a unit vector. Moreover, we have $\sqrt{q}=Q\gamma$. With the new notation, we get
$$\sum_{h=1}^d\frac{1}{|\mathcal{C}_h|}\left(\sum_{j\in\mathcal{C}_h}Z_j\right)^2=\|Q^TZ\|^2.$$
The covariance of $Q^TZ$ is
$$Q^T(I_k-\sqrt{q}\sqrt{q}^T)Q=I_d-\gamma\gamma^T.$$
Therefore, $\|Q^TZ\|^2\leadsto \chi_{d-1}^2$. Finally,
$$\sum_{j=1}^kZ_j^2-\sum_{h=1}^d\frac{1}{|\mathcal{C}_h|}\left(\sum_{j\in\mathcal{C}_h}Z_j\right)^2=\|Z\|^2-\|Q^TZ\|^2=Z^T(I_k-QQ^T)Z=W^T(I_k-QQ^T)W.$$
Therefore, its asymptotic distribution is $\chi_{k-d}^2$.
\end{proof}

The results of Theorem \ref{thm:null-gauss}, Theorem \ref{thm:null-gauss-degen}, Theorem \ref{thm:null-gauss-cat} and Theorem \ref{thm:null-cat-degen} are special cases of Theorem \ref{thm:Gaussian}, Theorem \ref{thm:M2}, Theorem \ref{thm:M3} and Theorem \ref{thm:M4}. Next, we give proofs of Theorem \ref{thm:two-sample-null} and Theorem \ref{thm:m2-performance}.

\begin{proof}[Proof of Theorem \ref{thm:two-sample-null}]
Without loss of generality, we can assume that $p_1=q_1\leq p_2=q_2\leq...\leq p_k=q_k$. This is just to simplify the notation. In general, such a rearrangement can always be done with extra notation of permutations. Then, $\mathcal{C}_g=\{j_g+1,j_g+2,...,j_{g+1}\}$ for $g\in[d]$. According to the assumption, $\min_{g\neq h}\min_{j\in\mathcal{C}_g}\min_{l\in\mathcal{C}_h}\sqrt{n}|\sqrt{p_j}-\sqrt{p_l}|=o(1)$. Moreover, it is easy to see that
$\max_{j\in[k]}\sqrt{n}|\sqrt{\hat{p}_j}-\sqrt{p_j}|=O_P(1)$ and $\max_{j\in[k]}\sqrt{m}|\sqrt{\hat{q}_j}-\sqrt{q_j}|=O_P(1)$. This leads to the conclusion
$$\mathbb{P}\left(\underline{\mathcal{C}}_g=\overline{\mathcal{C}}_g=\mathcal{C}_g\text{ for all }g\in[d]\text{ and }\underline{d}=\overline{d}=d\right)\rightarrow 1,$$
under Condition E.

From now on, the analysis is on the event $\{\underline{\mathcal{C}}_g=\overline{\mathcal{C}}_g=\mathcal{C}_g\text{ for all }g\in[d]\text{ and }\underline{d}=\overline{d}=d\}$. Define $\underline{Z}_j=2\sqrt{n}(\sqrt{\hat{p}_j}-\sqrt{p_j})$ and $\overline{Z}_j=2\sqrt{m}(\sqrt{\hat{q}_j}-\sqrt{q_j})$ for $j\in[k]$. The definition implies that $\max_{j\in[k]}|\underline{Z}_j|=O_P(1)$ and $\max_{j\in[k]}|\overline{Z}_j|=O_P(1)$. The definitions of $\underline{r}_g$ and $\overline{r}_g$ give
$$2\sqrt{n}(\sqrt{\underline{r}_g}-\sqrt{r_g})=\frac{1}{|\mathcal{C}_g|}\sum_{j\in\mathcal{C}_g}\underline{Z}_j\quad\text{and}\quad 2\sqrt{m}(\sqrt{\overline{r}_g}-\sqrt{r_g})=\frac{1}{|\mathcal{C}_g|}\sum_{j\in\mathcal{C}_g}\overline{Z}_j.$$
Given that $p_j=q_j=r_g$ for all $j\in\mathcal{C}_g$, we have $\sqrt{n}|\sqrt{\hat{q}_j}-\sqrt{\underline{r}_g}|=O_P(1)$ and $\sqrt{n}|\sqrt{\hat{p}_j}-\sqrt{\overline{r}_g}|=O_P(1)$ for all $j\in\mathcal{C}_g$. We also have $|\sqrt{\hat{q}_j}-\sqrt{\underline{r}_h}|^{-1}=O_P(1)$ and $|\sqrt{\hat{p}_j}-\sqrt{\overline{r}_h}|^{-1}=O_P(1)$ for all $j\in\mathcal{C}_g$ and $h\neq g$.

We first analyze $\underline{g}(t)$. By its definition,
$$\frac{1}{\underline{g}(t)}=\sum_{h=1}^d\frac{1}{(\sqrt{t}-\sqrt{\underline{r}_h})^2}.$$
Thus, for any $j\in\mathcal{C}_g$,
$$\frac{(\sqrt{\hat{q}_j}-\sqrt{\underline{r}_g})^2}{\underline{g}({\hat{q}_j})}=1+\sum_{h\in[d]\backslash\{g\}}\frac{(\sqrt{\hat{q}_j}-\sqrt{\underline{r}_g})^2}{(\sqrt{\hat{q}_j}-\sqrt{\underline{r}_h})^2}.$$
Similar to the argument in (\ref{eq:ignore-one}), we get
$$\left|\frac{\underline{g}({\hat{q}_j})}{(\sqrt{\hat{q}_j}-\sqrt{\underline{r}_g})^2}-1\right|\leq \sum_{h\in[d]\backslash\{g\}}\frac{(\sqrt{\hat{q}_j}-\sqrt{\underline{r}_g})^2}{(\sqrt{\hat{q}_j}-\sqrt{\underline{r}_h})^2}=O_P(n^{-1}).$$
With some rearragangements, we get
$$\left|\frac{2nm}{n+m}\sum_{j\in[k]}\underline{g}({\hat{q}_j})-\frac{2nm}{n+m}\sum_{g=1}^d\sum_{j\in\mathcal{C}_g}(\sqrt{\hat{q}_j}-\sqrt{\underline{r}_g})^2\right|=o_P(1).$$
A similar argument also gives
$$\left|\frac{2nm}{n+m}\sum_{j\in[k]}\overline{g}({\hat{p}_j})-\frac{2nm}{n+m}\sum_{g=1}^d\sum_{j\in\mathcal{C}_g}(\sqrt{\hat{p}_j}-\sqrt{\overline{r}_g})^2\right|=o_P(1).$$
Therefore, we obtain the following approximation
\begin{eqnarray*}
&&\left|T_g-\frac{2nm}{m+n}\sum_{g=1}^d\sum_{j\in\mathcal{C}_g}\left(\frac{1}{2\sqrt{n}}\underline{Z}_j-\frac{1}{2\sqrt{m}}\frac{1}{|\mathcal{C}_g|}\sum_{j\in\mathcal{C}_g}\overline{Z}_j\right)^2\right.\\
&&\left.- \frac{2nm}{m+n}\sum_{g=1}^d\sum_{j\in\mathcal{C}_g}\left(\frac{1}{2\sqrt{m}}\overline{Z}_j-\frac{1}{2\sqrt{n}}\frac{1}{|\mathcal{C}_g|}\sum_{j\in\mathcal{C}_g}\underline{Z}_j\right)^2\right| = o_P(1).
\end{eqnarray*}
Since
\begin{eqnarray*}
&& \sum_{j\in\mathcal{C}_g}\left(\frac{1}{2\sqrt{n}}\underline{Z}_j-\frac{1}{2\sqrt{m}}\frac{1}{|\mathcal{C}_g|}\sum_{j\in\mathcal{C}_g}\overline{Z}_j\right)^2 \\
&=& \sum_{j\in\mathcal{C}_g}\left(\frac{1}{2\sqrt{n}}\underline{Z}_j-\frac{1}{2\sqrt{n}}\frac{1}{|\mathcal{C}_g|}\sum_{j\in\mathcal{C}_g}\underline{Z}_j\right)^2 + |\mathcal{C}_g|\left(\frac{1}{2\sqrt{n}}\frac{1}{|\mathcal{C}_g|}\sum_{j\in\mathcal{C}_g}\underline{Z}_j-\frac{1}{2\sqrt{m}}\frac{1}{|\mathcal{C}_g|}\sum_{j\in\mathcal{C}_g}\overline{Z}_j\right)^2,
\end{eqnarray*}
and
\begin{eqnarray*}
&& \sum_{j\in\mathcal{C}_g}\left(\frac{1}{2\sqrt{m}}\overline{Z}_j-\frac{1}{2\sqrt{n}}\frac{1}{|\mathcal{C}_g|}\sum_{j\in\mathcal{C}_g}\underline{Z}_j\right)^2 \\
&=& \sum_{j\in\mathcal{C}_g}\left(\frac{1}{2\sqrt{m}}\overline{Z}_j-\frac{1}{2\sqrt{m}}\frac{1}{|\mathcal{C}_g|}\sum_{j\in\mathcal{C}_g}\overline{Z}_j\right)^2 + |\mathcal{C}_g|\left(\frac{1}{2\sqrt{n}}\frac{1}{|\mathcal{C}_g|}\sum_{j\in\mathcal{C}_g}\underline{Z}_j-\frac{1}{2\sqrt{m}}\frac{1}{|\mathcal{C}_g|}\sum_{j\in\mathcal{C}_g}\overline{Z}_j\right)^2,
\end{eqnarray*}
we have
\begin{eqnarray}
\nonumber && \left|T_g - \frac{m}{2(n+m)}\sum_{g=1}^d\sum_{j\in\mathcal{C}_g}\left(\underline{Z}_j-\frac{1}{|\mathcal{C}_g|}\sum_{j\in\mathcal{C}_g}\underline{Z}_j\right)^2-\frac{n}{2(n+m)}\sum_{g=1}^d\sum_{j\in\mathcal{C}_g}\left(\overline{Z}_j-\frac{1}{|\mathcal{C}_g|}\sum_{j\in\mathcal{C}_g}\overline{Z}_j\right)^2\right. \\
\label{eq:camero} && \left.- \sum_{g=1}^d|\mathcal{C}_g|\left(\frac{1}{|\mathcal{C}_g|}\sum_{j\in\mathcal{C}_g}\left(\sqrt{\frac{m}{m+n}}\underline{Z}_j-\sqrt{\frac{n}{m+n}}\overline{Z}_j\right)\right)^2\right|=o_P(1).
\end{eqnarray}

Next, we analyze $\underline{f}_h(t)$. By its definition,
$$\frac{d\underline{f}_h(t)}{d\sqrt{t}}=\frac{\prod_{g\in[d]\backslash\{h\}}(\sqrt{t}-\sqrt{\underline{r}_g})}{\prod_{g\in[d]\backslash\{h\}}(\sqrt{\underline{r}_h}-\sqrt{\underline{r}_g})}.$$
Therefore, we have
$$\max_{g\in[d]}\sup_{\sqrt{n}|\sqrt{t}-\sqrt{r_g}|\leq\lambda_n}\left|\frac{d\underline{f}_g(t)}{d\sqrt{t}}-1\right|=o_P(1)\quad\text{and}\quad \max_{g\in[d]\backslash\{h\}}\sup_{\sqrt{n}|\sqrt{t}-\sqrt{r_g}|\leq\lambda_n}\left|\frac{d\underline{f}_h(t)}{d\sqrt{t}}\right|=o_P(1).$$
Using Taylor expansion, we get
$$\sum_{j=1}^k\underline{f}_h(\hat{p}_j)-\sum_{j=1}^k\underline{f}_h(\hat{q}_j)=\sum_{j\in\mathcal{C}_h}(\sqrt{\hat{p}_j}-\sqrt{\hat{q}_j})+o_P(1)\sum_{j=1}^k|\sqrt{\hat{p}_j}-\sqrt{\hat{q}_j}|.$$
Then, we have
$$\left|\frac{2nm}{n+m}\sum_{h=1}^{d}\frac{1}{|\mathcal{C}_h|}\left(\sum_{j=1}^k\underline{f}_h(\hat{p}_j)-\sum_{j=1}^k\underline{f}_h(\hat{q}_j)\right)^2-\frac{2nm}{n+m}\sum_{h=1}^{d}\frac{1}{|\mathcal{C}_h|}\left(\sum_{j\in\mathcal{C}_h}(\sqrt{\hat{p}_j}-\sqrt{\hat{q}_j})\right)^2\right|=o_P(1).$$
The same argument also leads to
$$\left|\frac{2nm}{n+m}\sum_{h=1}^{d}\frac{1}{|\mathcal{C}_h|}\left(\sum_{j=1}^k\overline{f}_h(\hat{p}_j)-\sum_{j=1}^k\overline{f}_h(\hat{q}_j)\right)^2-\frac{2nm}{n+m}\sum_{h=1}^{d}\frac{1}{|\mathcal{C}_h|}\left(\sum_{j\in\mathcal{C}_h}(\sqrt{\hat{p}_j}-\sqrt{\hat{q}_j})\right)^2\right|=o_P(1).$$
Hence, we have the following approximation,
\begin{equation}
\left|T_f-\sum_{g=1}^d|\mathcal{C}_g|\left(\frac{1}{|\mathcal{C}_g|}\sum_{j\in\mathcal{C}_g}\left(\sqrt{\frac{m}{m+n}}\underline{Z}_j-\sqrt{\frac{n}{m+n}}\overline{Z}_j\right)\right)^2\right|=o_P(1).\label{eq:mustang}
\end{equation}

According to the argument in the proof of Theorem \ref{thm:M3}, $\underline{Z}=(I_k-\sqrt{p}\sqrt{p}^T)\underline{W}$ with $\underline{W}\leadsto N(0,I_k)$. Similarly, we also have $\overline{Z}=(I_k-\sqrt{q}\sqrt{q}^T)\overline{W}$ with $\overline{W}\leadsto N(0,I_k)$. Note that $\overline{W}$ is independent of $\underline{W}$.
Recall the definition of the matrix $Q$ and the vector $\gamma$ in the proof of Theorem \ref{thm:M4}. Then,
\begin{eqnarray*}
\sum_{g=1}^d\sum_{j\in\mathcal{C}_g}\left(\underline{Z}_j-\frac{1}{|\mathcal{C}_g|}\sum_{j\in\mathcal{C}_g}\underline{Z}_j\right)^2 &=& \underline{Z}^T(I_k-QQ^T)\underline{Z}, \\
\sum_{g=1}^d\sum_{j\in\mathcal{C}_g}\left(\overline{Z}_j-\frac{1}{|\mathcal{C}_g|}\sum_{j\in\mathcal{C}_g}\overline{Z}_j\right)^2 &=& \overline{Z}^T(I_k-QQ^T)\overline{Z}, \\
\sum_{g=1}^d|\mathcal{C}_g|\left(\frac{1}{|\mathcal{C}_g|}\sum_{j\in\mathcal{C}_g}\left(\sqrt{\frac{m}{m+n}}\underline{Z}_j-\sqrt{\frac{n}{m+n}}\overline{Z}_j\right)\right)^2 &=& \left\|Q^T\left(\sqrt{\frac{m}{m+n}}\underline{Z}-\sqrt{\frac{n}{m+n}}\overline{Z}\right)\right\|^2.
\end{eqnarray*}
Furthermore, we have
\begin{eqnarray*}
\underline{Z}^T(I_k-QQ^T)\underline{Z} &=& \underline{W}^T(I_k-QQ^T)\underline{W}, \\
\overline{Z}^T(I_k-QQ^T)\overline{Z} &=& \overline{W}^T(I_k-QQ^T)\overline{W}, \\
\left\|Q^T\left(\sqrt{\frac{m}{m+n}}\underline{Z}-\sqrt{\frac{n}{m+n}}\overline{Z}\right)\right\|^2 &=& \left\|(I_k-\gamma\gamma^T)Q^T\left(\sqrt{\frac{m}{m+n}}\underline{W}-\sqrt{\frac{n}{m+n}}\overline{W}\right)\right\|^2.
\end{eqnarray*}
Therefore, the three terms above are asymptotically independent, and their asymptotic distributions are $\chi_{k-d}^2$, $\chi_{k-d}^2$ and $\chi_{d-1}^2$ under the null, respectively.
\end{proof}

\begin{proof}[Proof of Theorem \ref{thm:m2-performance}]
We will borrow notation and arguments used in the proof of Theorem \ref{thm:M2}. For example, we keep using the notation $L=\max_{1\leq g\leq d}\max_{j\in\mathcal{C}_g}\sqrt{n}|\mu_j-\nu_g|$. However, under Condition M2', we have $L=O(1)$ instead of $L=o(1)$. Let $C_n$ be a diverging sequence that satisfies $C_n\rightarrow\infty$ and $\frac{C_n^{3/2}\max_{g\neq h}|\bar{\eta}_{gh}|}{\sqrt{n}}\rightarrow 0$. Then, we can use the same analysis in the proof of Theorem \ref{thm:M2} that leads to (\ref{eq:nice-analysis}) and (\ref{eq:bd-Tf}). Note that the only difference is $L=O(1)$, and it will not affect the conclusions of (\ref{eq:nice-analysis}) and (\ref{eq:bd-Tf}). We still have
$$\left|T_g-n\sum_{h=1}^d\sum_{j\in\mathcal{C}_h}(X_{\pi^{-1}(j)}-\nu_h)^2\right|=o_P(1),$$
and
$$\left|T_f-n\sum_{h=1}^d\frac{1}{|\mathcal{C}_h|}\left(\sum_{j\in\mathcal{C}_h}(X_{\pi^{-1}(j)}-\mu_j)\right)^2\right|=o_P(1).$$
By the fact that
\begin{eqnarray*}
&& n\sum_{h=1}^d\sum_{j\in\mathcal{C}_h}(X_{\pi^{-1}(j)}-\nu_h)^2 - n\sum_{h=1}^d\frac{1}{|\mathcal{C}_h|}\left(\sum_{j\in\mathcal{C}_h}(X_{\pi^{-1}(j)}-\mu_j)\right)^2 \\
&=& n\sum_{h=1}^d\sum_{j\in\mathcal{C}_h}(X_{\pi^{-1}(j)}-\nu_h)^2 - n\sum_{h=1}^d|\mathcal{C}_h|\left(\frac{1}{|\mathcal{C}_h|}\sum_{j\in\mathcal{C}_h}(X_{\pi^{-1}(j)}-\nu_h)\right)^2 \\
&=& n\sum_{h=1}^d\sum_{j\in\mathcal{C}_h}\left(X_{\pi^{-1}(j)}-\frac{1}{|\mathcal{C}_h|}\sum_{j\in\mathcal{C}_h}X_{\pi^{-1}(j)}\right)^2,
\end{eqnarray*}
we also have
$$\left|T_g-T_f-n\sum_{h=1}^d\sum_{j\in\mathcal{C}_h}\left(X_{\pi^{-1}(j)}-\frac{1}{|\mathcal{C}_h|}\sum_{j\in\mathcal{C}_h}X_{\pi^{-1}(j)}\right)^2\right|=o_P(1).$$
Therefore, under the null hypothesis $X\sim N(\mu,n^{-1}I_k)$, we have $T_g\leadsto\chi_{k,\tau^2}^2$, $T_f\leadsto\chi_{d}^2$ and $T_g-T_f\leadsto \chi_{k-d,\tau^2}^2$.
\end{proof}

\subsection{Power Analysis}\label{sec:pf-power}

In this section, we give proofs of Theorem \ref{thm:power-Gaussian}, Theorem \ref{thm:power-degen-Gaussian}, Theorem \ref{thm:power-Gaussian-cat}, Theorem \ref{thm:power-degen-cat} and Theorem \ref{thm:two-sample-power}.

\begin{proof}[Proof of Theorem \ref{thm:power-Gaussian}]
We first assume $n\ell(\theta,\mu)^2\rightarrow\infty$ and derive $T\rightarrow\infty$ in probability.
Note that for each $\pi\in S_k$,
$$n\sum_{j=1}^k(\theta_j-\mu_{\pi(j)})^2\leq 2n\sum_{j=1}^k(X_j-\theta_j)^2 + 2n\sum_{j=1}^k(X_j-\mu_{\pi(j)})^2.$$
Therefore,
$$n\ell(\theta,\mu)^2\leq 2\sum_{j=1}^kZ_j^2 + 2n\ell(X,\mu)^2,$$
where $Z_j\sim N(0,1)$. The fact that $2\sum_{j=1}^kZ_j^2=O_P(1)$ and the assumption $n\ell(\theta,\mu)^2\rightarrow\infty$ implies that $n\ell(X,\mu)^2\rightarrow\infty$ in probability. Suppose we can show $T=O_P(1)$ implies $n\ell(X,\mu)^2=O_P(1)$, then $n\ell(X,\mu)^2\rightarrow\infty$ in probability must implies $T\rightarrow\infty$ in probability.

Now we suppose a bound $T\leq B=O(1)$, and it is sufficient to derive a bound for $n\ell(X,\mu)^2$. For each $j=1,...,k$, we shorthand the power sums $p_j(X_1,...,X_k)$ and $p_j(\mu_1,...,\mu_k)$ by $p_j(X)$ and $p_j(\mu)$. Similarly, the elementary symmetric polynomials $e_j(X_1,...,X_k)$ and $e_j(\mu_1,...,\mu_k)$ are shorthanded by $e_j(X)$ and $e_j(\mu)$. Define a vector $\Delta\in\mathbb{R}^k$ with the $j$th entry being $\Delta_j=\frac{1}{j}\sum_{h=1}^kX_h^j-\frac{1}{j}\sum_{h=1}^k\mu_h^j$. Recall the definition of the matrix $E(\mu_1,...,\mu_k)$. Then,
$$T=n\|E(\mu_1,...,\mu_k)\Delta\|^2.$$
We use $\lambda_{\max}(\cdot)$ and $\lambda_{\min}(\cdot)$ to denote the largest and the smallest eigenvalues. By the fact that $V(\mu_1,...,\mu_k)E(\mu_1,...,\mu_k)=I_k$, we have
$$T\geq n\lambda_{\min}(E(\mu_1,...,\mu_k)^TE(\mu_1,...,\mu_k))\|\Delta\|^2\geq \frac{n\|\Delta\|^2}{\lambda_{\max}(V(\mu_1,...,\mu_k)^TV(\mu_1,...,\mu_k))}.$$
The bound $T\leq B$ then leads to
\begin{equation}
\|\Delta\|^2 \leq \frac{\lambda_{\max}(V(\mu_1,...,\mu_k)^TV(\mu_1,...,\mu_k))B}{n}=O\left(\frac{B}{n}\right).\label{eq:porche}
\end{equation}
Therefore, $|p_j(X)-p_j(\mu)|^2=O\left(B/n\right)$ for each $j\in[k]$. By Newton's identities, we can deduce $|e_j(X)-e_j(\mu)|^2=O\left(B/n\right)$ for each $j\in[k]$. Define
$$f(t)=\prod_{j=1}^k(t-\mu_j),\quad \hat{f}(t)=\prod_{j=1}^k(t-X_j).$$
The relation between the two polynomials and the elementary symmetric polynomials is given in (\ref{eq:expansion}). Using (\ref{eq:expansion}), we give a bound for $|f(X_l)|$.
$$
|f(X_l)| = |f(X_l)-\hat{f}(X_l)| \leq \sum_{j=0}^k|e_{k-j}(X)-e_{k-j}(\mu)||X_l|^j.
$$
Since $|X_l|^2\leq p_2(X)\leq p_2(\mu)+|p_2(X)-p_2(\mu)|=O(1)$, we have $|f(X_l)|^2=O(B/n)$.
The following proposition is useful and will be proved in the end.
\begin{proposition}\label{prop:f-lb}
For any $\mu_1,...,\mu_k$, we have
$$|f(t)|\geq \min_{1\leq j\leq k}|t-\mu_j|\prod_{1\leq j<l\leq k}\left|\frac{\mu_j-\mu_l}{2}\right|.$$
\end{proposition}
By this inequality, we have
\begin{equation}
\max_{1\leq l\leq k}\min_{1\leq j\leq k}(X_l-\mu_j)^2=O\left(\frac{B}{n}\right).\label{eq:already-close}
\end{equation}
Therefore, there exists a sequence $\sigma(1),...,\sigma(k)$ such that
$$\max_{1\leq j\leq k}(X_j-\mu_{\sigma(j)})^2=O\left(\frac{B}{n}\right).$$
Since
$$\prod_{j=1}^k|t-\mu_{\sigma(j)}|\leq 2^k\prod_{j=1}^k|t-X_j|+2^k\prod_{j=1}^k|X_j-\mu_{\sigma(j)}|=2^k|\hat{f}(t)|+O\left(\left(\frac{B}{n}\right)^{k/2}\right),$$
and
$$|\hat{f}(\mu_l)|=|\hat{f}(\mu_l)-f(\mu_l)|\leq  \sum_{j=0}^k|e_{k-j}(X)-e_{k-j}(\mu)||\mu_l|^j=O\left(\sqrt{\frac{B}{n}}\right),$$
we have
$$\prod_{j=1}^k|\mu_l-\mu_{\sigma(j)}|=O\left(\sqrt{\frac{B}{n}}\right),$$
which holds for every $l=1,...,k$. The fact that $\mu_1,...,\mu_k$ are $k$ different fixed number implies $\sigma$ must be an element of $S_k$. Hence, the bound (\ref{eq:already-close}) implies $n\ell(X,\mu)^2=O(B)$, and the proof of one direction is complete.

For the other direction, it is sufficient to show that $n\ell(\theta,\mu)=O(1)$ implies $T=O_P(1)$. This can be shown using the same argument in the proof of Theorem \ref{thm:Gaussian}.
\end{proof}

\begin{proof}[Proof of Proposition \ref{prop:f-lb}]
We first consider the case $k=2$, where $f(t)=(t-\mu_1)(t-\mu_2)$. Suppose $|t-\mu_1|\leq |t-\mu_2|$, then $|t-\mu_2|\geq \frac{|\mu_1-\mu_2|}{2}$. Thus, $|f(t)|\geq \frac{|\mu_1-\mu_2|}{2}\min\{|t-\mu_1|,|t-\mu_2|\}$. The same argument also works when $|t-\mu_1|> |t-\mu_2|$. When $k=3$,
\begin{eqnarray*}
|f(t)| &\geq& |t-\mu_3|\frac{|\mu_1-\mu_2|}{2}\min\{|t-\mu_1|,|t-\mu_2|\} \\
&=& \frac{|\mu_1-\mu_2|}{2}\min\left\{|t-\mu_1||t-\mu_3|,|t-\mu_2||t-\mu_3|\right\}.
\end{eqnarray*}
The inequality for $k=2$ can be used to lower bound both $|t-\mu_1||t-\mu_3|$ and $|t-\mu_2||t-\mu_3|$. This gives the desired result for $k=3$. A standard mathematical induction argument leads to inequality for all $k$.
\end{proof}

\begin{proof}[Proof of Theorem \ref{thm:power-degen-Gaussian}]
According to the argument that we have used in the proof of Theorem \ref{thm:power-Gaussian}, we need to show $T_f=O_P(1)$ and $T_g=O_P(1)$ imply $n\ell(X,\mu)^2=O_P(1)$ for the proof of the first direction.

Suppose $T_f\leq B_1=O(1)$ and $T_g\leq B_2=O(1)$. It is sufficient to derive a bound for $n\ell(X,\mu)^2$. We first derive an inequality for $g(t)$.
Since
$$\max_{1\leq g\leq d}\prod_{h\in[d]\backslash\{g\}}(t-\nu_h)^2\leq\sum_{g=1}^d\prod_{h\in[d]\backslash\{g\}}(t-\nu_h)^2\leq d\max_{1\leq g\leq d}\prod_{h\in[d]\backslash\{g\}}(t-\nu_h)^2,$$
we have
\begin{equation}
\frac{1}{d}\min_{1\leq g\leq d}(t-\nu_g)^2\leq g(t)\leq \min_{1\leq g\leq d}(t-\nu_g)^2.
\end{equation}
Therefore, $T_g\leq B_2$ implies that $\sum_{j=1}^k\min_{1\leq g\leq d}(X_j-\nu_g)^2\leq \frac{dB_2}{n}$.
This implies the existence of a sequence $\sigma(1),...,\sigma(k)$ such that $\max_{1\leq j\leq k}(X_j-\nu_{\sigma(j)})^2\leq \frac{dB_2}{n}=O(B_2/n)$. It further implies $\max_{1\leq h\leq d}\max_{1\leq j\leq k}|X_j^h-\nu_{\sigma(j)}^h|=O(\sqrt{B_2/n})$. Define $\hat{\mathcal{C}}_g=\{j\in[k]:\sigma(j)=g\}$ for each $g\in[d]$. Then
$$\sum_{h=1}^{d-1}\left(\sum_{j=1}^kX_j^h-\sum_{g=1}^d|\hat{\mathcal{C}}_g|\nu_g^h\right)^2=O\left(\frac{B_2}{n}\right).$$
Using the same argument in deriving (\ref{eq:porche}), we can also get the bound
$$\sum_{h=1}^{d-1}\left(\sum_{j=1}^kX_j^h-\sum_{g=1}^d|{\mathcal{C}}_g|\nu_g^h\right)^2=O\left(\frac{B_1}{n}\right).$$
The inequalities in the last two displays, together with the equality $\sum_{g=1}^d|{\hat{\mathcal{C}}}_g|=\sum_{g=1}^d|{{\mathcal{C}}}_g|$, give
$$\sum_{h=0}^{d-1}\left(\sum_{g=1}^d|{\hat{\mathcal{C}}}_g|\nu_g^h-\sum_{g=1}^d|{\mathcal{C}}_g|\nu_g^h\right)^2=O\left(\frac{B_1+B_2}{n}\right).$$
Define a vector $r\in\mathbb{R}^d$, with its $g$th entry being $|{\hat{\mathcal{C}}}_g|-|{\mathcal{C}}_g|$. Then,
$$\sum_{h=0}^{d-1}\left(\sum_{g=1}^d|{\hat{\mathcal{C}}}_g|\nu_g^h-\sum_{g=1}^d|{\mathcal{C}}_g|\nu_g^h\right)^2=\|V(\nu_1,...,\nu_d)r\|^2\geq \lambda_{\min}(V(\nu_1,...,\nu_d)^TV(\nu_1,...,\nu_d))\|r\|^2.$$
When $\nu_1,...,\nu_d$ are $d$ different numbers, we have $\lambda_{\min}(V(\nu_1,...,\nu_d)^TV(\nu_1,...,\nu_d))>0$, and thus $\|r\|^2=O\left(\frac{B_1+B_2}{n}\right)$. Since $\|r\|^2$ is an integer, we must have $\|r\|^2=0$, which gives $|{\hat{\mathcal{C}}}_g|=|{\mathcal{C}}_g|$ for any $g\in[d]$. From this we can deduce that  $n\ell(X,\mu)^2=O(B_2)$.

For the other direction, it is sufficient to show that $n\ell(\theta,\mu)^2=O(1)$ implies $T_f=O_P(1)$ and $T_g=O_P(1)$. This can be shown using the same argument in the proof of Theorem \ref{thm:M2}.
\end{proof}

\begin{proof}[Proofs of Theorem \ref{thm:power-Gaussian-cat} and Theorem \ref{thm:power-degen-cat}]
The proofs are the same as those of Theorem \ref{thm:power-Gaussian} and Theorem \ref{thm:power-degen-Gaussian}.
\end{proof}

\begin{proof}[Proof of Theorem \ref{thm:two-sample-power}]
First of all, we have $\max_{j\in[k]}\sqrt{n}|\sqrt{\hat{q}_j}-\sqrt{q_j}|=O_P(1)$. This gives that $\mathbb{P}(\overline{\mathcal{C}}_g=\mathcal{C}_g\text{ for all }g\in[d]\text{ and }\overline{d}=d)\rightarrow 1$. From now on, the analysis is on the event $\{\overline{\mathcal{C}}_g=\mathcal{C}_g\text{ for all }g\in[d]\text{ and }\overline{d}=d\}$.
Since we also have $\max_{j\in[k]}\sqrt{n}|\sqrt{\hat{p}_j}-\sqrt{p_j}|=O_P(1)$, the statement $n\ell(p,q)^2\rightarrow\infty$ is equivalent to $n\ell(\hat{p},q)^2\rightarrow\infty$ in probability. Therefore, we only need to establish the equivalence between $n\ell(\hat{p},q)^2\rightarrow\infty$ in probability and the power of the test goes to one.

In the first direction of the proof, we suppose that $T_f\leq B_1=O_P(1)$ and $T_g\leq B_2=O_P(1)$, and we will show $n\ell(\hat{p},{q})^2=O_P(1)$. The bound $T_g\leq B_2=O_P(1)$ implies that
$$\frac{2nm}{n+m}\sum_{j=1}^k\overline{g}(\hat{p}_j)\leq B_2.$$
By the definition of $\overline{g}(\cdot)$, we have $\overline{g}(t)\geq d^{-1}\min_{g\in[d]}(\sqrt{t}-\sqrt{\overline{r}_g})^2$. This implies the bound
$$\sum_{j=1}^k\min_{g\in[d]}\left(\sqrt{\hat{p}_j}-\sqrt{\overline{r}_g}\right)^2\leq \frac{dB_2(n+m)}{nm}=O_P(n^{-1}).$$
Since $\max_{g\in[d]}\sqrt{n}|\sqrt{\overline{r}_g}-\sqrt{r_g}|=O_P(1)$, we deduce
$$\sum_{j=1}^k\min_{g\in[d]}\left(\sqrt{\hat{p}_j}-\sqrt{{r}_g}\right)^2=O_P(n^{-1}).$$
Then, there must exist $\sigma(1),...,\sigma(k)$ such that $\max_{j\in[k]}|\sqrt{\hat{p}_j}-\sqrt{r_{\sigma(j)}}|^2=O_P(n^{-1})$. It further implies that $\max_{h\in[d]}\max_{j\in[k]}|(\sqrt{\hat{p}_j})^h-(\sqrt{r_{\sigma(j)}})^h|=O_P(n^{-1/2})$. Define $\hat{\mathcal{C}}_g=\{j\in[k]:\sigma(j)=g\}$ for each $g\in[d]$. Then, we have
\begin{equation}
\sum_{h=1}^{d-1}\left(\sum_{j=1}^k(\sqrt{\hat{p}_j})^h-\sum_{g=1}^d|\hat{\mathcal{C}}_g|(\sqrt{r_{g}})^h\right)^2=O_P(n^{-1}).\label{eq:amg-c63}
\end{equation}
Note that
\begin{eqnarray*}
T_f &\geq& \frac{2nm}{n+m}\sum_{h=1}^{{d}}\frac{1}{|{\mathcal{C}}_h|}\left(\sum_{j=1}^k\overline{f}_h(\hat{p}_j)-\sum_{j=1}^k\overline{f}_h(\hat{q}_j)\right)^2 \\
&\geq& \frac{2nm}{d(n+m)}\sum_{h=1}^{{d}}\left(\sum_{j=1}^k\overline{f}_h(\hat{p}_j)-\sum_{j=1}^k\overline{f}_h(\hat{q}_j)\right)^2 \\
&=& \frac{2nm}{d(n+m)}\|E(\sqrt{\overline{r}_1},...,\sqrt{\overline{r}_d})\Delta\|^2,
\end{eqnarray*}
where $\Delta$ is a $d$-dimensional vector with $\Delta_h=\frac{1}{h}\sum_{j=1}^k(\sqrt{\hat{p}_j})^h-\frac{1}{h}\sum_{j=1}^k(\sqrt{\hat{q}_j})^h$. Thus, the bound $T_f\leq B_1=O_P(1)$ implies that
$$\|E(\sqrt{\overline{r}_1},...,\sqrt{\overline{r}_d})\Delta\|^2=O_P(n^{-1}).$$
Since $\lambda_{\min}(E(\sqrt{{r}_1},...,\sqrt{{r}_d})^TE(\sqrt{{r}_1},...,\sqrt{{r}_d}))$ is a positive constant that is bounded away from $0$, and
$$|\lambda_{\min}(E(\sqrt{\overline{r}_1},...,\sqrt{\overline{r}_d})^TE(\sqrt{\overline{r}_1},...,\sqrt{\overline{r}_d}))-\lambda_{\min}(E(\sqrt{{r}_1},...,\sqrt{{r}_d})^TE(\sqrt{{r}_1},...,\sqrt{{r}_d}))|=o_P(1),$$
we have $\|\Delta\|^2=O_P(n^{-1})$, which further leads to
\begin{equation}
\sum_{h=1}^{d-1}\left(\sum_{j=1}^k(\sqrt{\hat{p}_j})^h-\sum_{g=1}^d|{\mathcal{C}}_g|(\sqrt{r_{g}})^h\right)^2=O_P(n^{-1}),\label{eq:amg-r63}
\end{equation}
by using the fact that $\max_{j\in[k]}\sqrt{n}|\sqrt{\hat{q}_j}-\sqrt{q_j}|=O_P(1)$ and Condition E. The two inequalities (\ref{eq:amg-c63}) and (\ref{eq:amg-r63}), together with the fact that $\sum_{g=1}^d|\hat{\mathcal{C}}_g|=\sum_{g=1}^d|{\mathcal{C}}_g|$, imply
$$\sum_{h=0}^{d-1}\left(\sum_{g=1}^d|\hat{\mathcal{C}}_g|(\sqrt{r_{g}})^h-\sum_{g=1}^d|{\mathcal{C}}_g|(\sqrt{r_{g}})^h\right)^2=O_P(n^{-1}).$$
The same argument used in the proof of Theorem \ref{thm:power-degen-Gaussian} implies that $|\hat{\mathcal{C}}_g|=|\mathcal{C}_g|$ for all $g\in[d]$. Therefore, together with $\max_{j\in[k]}|\sqrt{\hat{p}_j}-\sqrt{r_{\sigma(j)}}|^2=O_P(n^{-1})$, we obtain the conclusion $n\ell(\hat{p},q)=O_P(1)$.

For the other direction, when $n\ell(p,q)^2=O(1)$, the approximations (\ref{eq:camero}) and (\ref{eq:mustang}) in the proofs of Theorem \ref{thm:two-sample-null} hold with bounds at the order of $O_P(1)$. This leads to $T_f=O_P(1)$ and $T_g=O_P(1)$.
\end{proof}

\subsection{Minimax Upper and Lower Bounds}\label{sec:pf-optimal}

In this section, we prove all results in Section \ref{sec:opt}. We first give proofs for the lower bounds, and then for the upper bounds.

\begin{proof}[Proof of Theorem \ref{thm:lower-gauss}]
We first observe an inequality $|\eta_{jl}|+|\eta_{lj}|\geq\frac{2}{|\mu_l-\mu_j|}$,
which has been derived in the proof of Theorem \ref{thm:Gaussian}.
Thus, Condition M1 implies $\sqrt{n}|\mu_j-\mu_l|\rightarrow\infty$ for any $j\neq l$.
Consider the set
$$\bar{\Theta}_{\delta}=\left\{\theta:\|\theta-\mu\|=\frac{\delta}{\sqrt{n}}\right\}.$$
For each $\theta\in\bar{\Theta}_{\delta}$, $|\theta_j-\mu_j|^2\leq \frac{\delta^2}{n}$, which implies $\mu_j$ is the closest element to $\theta_j$ in the set $\{\mu_1,...,\mu_k\}$. Therefore, $\ell(\theta,\mu)=\|\theta-\mu\|=\delta/\sqrt{n}$, which implies $\bar{\Theta}_{\delta}\subset\Theta_{\delta}$. This gives the lower bound
$$R_n(k,\delta)\geq \inf_{0\leq\phi\leq 1}\left\{\mathbb{P}_{\mu}\phi+\sup_{\theta\in\bar{\Theta}_{\delta}}\mathbb{P}_{\theta}(1-\phi)\right\}.$$
Consider the uniform distribution $\Pi$ on $\bar{\Theta}_{\delta}$. Then,
$$R_n(k,\delta)\geq \inf_{0\leq\phi\leq 1}\left\{\mathbb{P}_{\mu}\phi+\int\mathbb{P}_{\theta}(1-\phi)d\Pi(\theta)\right\}.$$
By Neyman-Pearson lemma, the optimal testing function $\phi$ is given by
$$\phi=\left\{\frac{d\int\mathbb{P}_{\theta}d\Pi(\theta)}{d\mathbb{P}_{\mu}}>1\right\}.$$
Using the property of $\Pi$, we have
\begin{eqnarray*}
\frac{d\int\mathbb{P}_{\theta}d\Pi(\theta)}{d\mathbb{P}_{\mu}} &=& \int \frac{d\mathbb{P}_{\theta}}{d\mathbb{P}_{\mu}}d\Pi(\theta) \\
&=& \int \exp\left(-\frac{n}{2}\|\theta-\mu\|^2+n\iprod{X-\mu}{\theta-\mu}\right)d\Pi(\theta) \\
&=& e^{-\delta^2/2}\int\exp\left(n\iprod{X-\mu}{\theta-\mu}\right)d\Pi(\theta).
\end{eqnarray*}
Let $\bar{\Pi}$ be the uniform distribution on the unit sphere $\{\theta:\|\theta\|=1\}$, and then we have
$$\int\exp\left(n\iprod{X-\mu}{\theta-\mu}\right)d\Pi(\theta)=\int\exp\left(\delta\sqrt{n}\iprod{X-\mu}{\theta}\right)d\bar{\Pi}(\theta).$$
Let $f$ be the marginal density of the first coordinate of $\theta\sim\bar{\Pi}$. Then, $f(t)\propto (1-t^2)^{\frac{k-3}{2}}$. The uniformity of $\bar{\Pi}$ implies that
\begin{equation}
\int\exp\left(\delta\sqrt{n}\iprod{X-\mu}{\theta}\right)d\bar{\Pi}(\theta)=\frac{\int_{-1}^1\exp\left(\delta\sqrt{n}\|X-\mu\|t\right)(1-t^2)^{\frac{k-3}{2}}dt}{\int_{-1}^1(1-t^2)^{\frac{k-3}{2}}dt}.\label{eq:int-sphere}
\end{equation}
Therefore, we can write the quantity in the above display as $F(\sqrt{n}\|X-\mu\|)$. Since
$$F'(x)=\frac{\int_0^1(e^{\delta xt}-e^{-\delta xt})\delta t(1-t^2)^{\frac{k-3}{2}}dt}{\int_{-1}^1(1-t^2)^{\frac{k-3}{2}}dt}>0,\text{ for }x>0,$$
the testing statistic $\frac{d\int\mathbb{P}_{\theta}d\Pi(\theta)}{d\mathbb{P}_{\mu}}$ is an increasing function of $\|X-\mu\|^2$. This implies
$$\phi=\left\{n\|X-\mu\|^2\geq t\right\},$$
for some $t>0$. Note that $n\|X-\mu\|^2\sim\chi_k^2$ under $\mathbb{P}_{\mu}$, and $n\|X-\mu\|^2\sim\chi_{k,\delta^2}^2$ under any $\mathbb{P}_{\theta}$ with $\theta\in\bar{\Theta}_{\delta}$. Hence,
$$R_n(k,\delta)\geq \inf_{t>0}\left\{\mathbb{P}(\chi_k^2\geq t)+\mathbb{P}(\chi_{k,\delta^2}^2<t)\right\}.$$
This completes the proof.
\end{proof}

\begin{proof}[Proof of Theorem \ref{thm:lower-gauss-degen}]
Since $|\bar{\eta}_{gh}|+|\bar{\eta}_{hg}|\geq\frac{2}{|\nu_g-\nu_h|}$, it is implied by Condition M2 that $\sqrt{n}|\nu_g-\nu_h|\rightarrow\infty$ for any $g\neq h$. Moreover, for any $j\in\mathcal{C}_h$, $|\mu_j-\nu_h|=o(n^{-1/2})$. Under these assumptions, for any $\theta$ such that $\|\theta-\mu\|=\frac{\delta}{\sqrt{n}}$, there exists a $\pi\in S_k$ that depends on $\theta$ and $\|\theta_{\pi}-\mu\|=\ell(\theta,\mu)=\frac{\delta}{\sqrt{n}}(1+\epsilon_{\theta})$. Moreover, $|\epsilon_{\theta}|=o(1)$ uniformly over all $\theta$ that satisfies $\|\theta-\mu\|=\frac{\delta}{\sqrt{n}}$. Define
\begin{equation}
\theta'=\mu+\frac{1}{1+\epsilon_{\theta}}(\theta-\mu).\label{eq:def-opt}
\end{equation}
Then, $\|\theta'-\mu\|=\frac{\delta_{\theta}}{\sqrt{n}}$ and $\ell(\theta',\mu)=\frac{\delta}{\sqrt{n}}$, where $\delta_{\theta}=\frac{\delta}{1+\epsilon_{\theta}}$. We use the notation $R$ to denote the operator $R:\theta\mapsto R(\theta)=\theta'$ defined by (\ref{eq:def-opt}). By the definition, a useful property is $\frac{R(\theta)-\mu}{\|R(\theta)-\mu\|}=\frac{\theta-\mu}{\|\theta-\mu\|}$. Consider the set
$$\bar{\Theta}_{\delta}=\left\{R(\theta):\|\theta-\mu\|=\frac{\delta}{\sqrt{n}}\right\}.$$
This definition immediately implies $\bar{\Theta}_{\delta}\subset\Theta_{\delta}$. Note that each element in $\bar{\Theta}_{\delta}$ can be represented as
$$R(\theta)=\mu+\frac{\delta_{\theta}}{\sqrt{n}}\frac{\theta-\mu}{\|\theta-\mu\|}.$$
Since there is a one-to-one relation between $\frac{\theta-\mu}{\|\theta-\mu\|}$ and a unit vector $v$, we can also write each element in $\bar{\Theta}_{\delta}$ as $\mu+\frac{\delta_v}{\sqrt{n}} v$. Consider a uniform probability measure $\bar{\Pi}$ on $\{v:\|v\|=1\}$. Then, by the same argument in the proof of Theorem \ref{thm:lower-gauss},
$$R_n(k,\delta)\geq \inf_{0\leq\phi\leq 1}\left\{\mathbb{P}_{\mu}\phi+\int\mathbb{P}_{\mu+\frac{\delta_v}{\sqrt{n}}v}(1-\phi)d\bar{\Pi}(v)\right\},$$
and the likelihood ratio is
$$
\mathcal{L}=\int\frac{d\mathbb{P}_{\mu+\frac{\delta_v}{\sqrt{n}}v}}{d\mathbb{P}_{\mu}}d\bar{\Pi}(v) =  \int\exp\left(-\delta_v^2/2+\delta_v\sqrt{n}\iprod{X-\mu}{v}\right)d\bar{\Pi}(v).
$$
Under the assumption, there exist $\delta_-$ and $\delta_+$ such that $\delta_-\leq \delta_v\leq \delta_+$ for all $v$ and $\delta_-/\delta=1+o(1)$ and $\delta_+/\delta=1+o(1)$. We introduce the upper and lower brackets of $\mathcal{L}$ as
\begin{eqnarray*}
\mathcal{L}_- &=& \min\left\{\int\exp\left(-\delta_+^2/2+\delta_-\sqrt{n}\iprod{X-\mu}{v}\right)d\bar{\Pi}(v),\right.\\
&& \left.\int\exp\left(-\delta_+^2/2+\delta_+\sqrt{n}\iprod{X-\mu}{v}\right)d\bar{\Pi}(v)\right\}, \\
\mathcal{L}_+ &=& \min\left\{\int\exp\left(-\delta_-^2/2+\delta_-\sqrt{n}\iprod{X-\mu}{v}\right)d\bar{\Pi}(v),\right.\\
&& \left.\int\exp\left(-\delta_-^2/2+\delta_+\sqrt{n}\iprod{X-\mu}{v}\right)d\bar{\Pi}(v)\right\}.
\end{eqnarray*}
The definitions imply $\mathcal{L}_-\leq \mathcal{L}\leq\mathcal{L}_+$. 
%We also define
%$$\mathcal{L}^*=\int\exp\left(-\delta^2/2+\delta\sqrt{n}\iprod{X-\mu}{v}\right)d\bar{\Pi}(v).$$
Define the function
\begin{equation}
F_{\delta}(x)=\frac{\int_{-1}^1\exp\left(\delta xt\right)(1-t^2)^{\frac{k-3}{2}}dt}{\int_{-1}^1(1-t^2)^{\frac{k-3}{2}}dt}.\label{eq:F-delta}
\end{equation}
By (\ref{eq:int-sphere}), we have
\begin{eqnarray*}
\mathcal{L}_- &=& e^{-\delta_+^2/2}\min\{F_{\delta_-}(\sqrt{n}\|X-\mu\|),F_{\delta_+}(\sqrt{n}\|X-\mu\|)\}, \\
\mathcal{L}_+ &=& e^{-\delta_-^2/2}\max\{F_{\delta_-}(\sqrt{n}\|X-\mu\|),F_{\delta_+}(\sqrt{n}\|X-\mu\|)\}.
\end{eqnarray*}
Define $\phi=\mathbb{I}\{\mathcal{L}>1\}$, $\phi_-=\mathbb{I}\{\mathcal{L}_->1\}$ and $\phi_+=\mathbb{I}\{\mathcal{L}_+>1\}$. We have the inequality $\phi_-\leq \phi\leq\phi_+$. For $\theta=\mathbb{E}X=\mu$, $\|\sqrt{n}(X-\mu)\|^2\sim \chi_k^2$. Thus, let $Z\sim N(0,I_k)$, and then we have
\begin{eqnarray*}
\mathbb{P}_{\mu}\phi &\geq& \mathbb{P}_{\mu}\left(\mathcal{L}_->1\right) \\
&=& \mathbb{P}\left(e^{-\delta_+^2/2}\min\{F_{\delta_-}(\|Z\|),F_{\delta_+}(\|Z\|)\}>1\right) \\
&\rightarrow& \mathbb{P}\left(e^{-\delta^2/2}F_{\delta}(\|Z\|)>1\right).
\end{eqnarray*}
For the alternative $\theta=\mu+\frac{\delta_v}{\sqrt{n}}v\in\bar{\Theta}_{\delta}$, $\|\sqrt{n}(X-\mu)\|^2\sim \chi_{k^2,\delta_v^2}$, where $\delta_v^2\in[\delta_-,\delta_+]$. Then,
\begin{eqnarray*}
\mathbb{P}_{\mu+\frac{\delta_v}{\sqrt{n}}v}(1-\phi) &\geq& \mathbb{P}_{\mu+\frac{\delta_v}{\sqrt{n}}v}\left(\mathcal{L}_+\leq 1\right) \\
&=& \mathbb{P}\left(e^{-\delta_-^2/2}\max\{F_{\delta_-}(\|Z+\delta_v v\|),F_{\delta_+}(\|Z+\delta_v v\|)\}\leq 1\right) \\
&\geq& \mathbb{P}\left(e^{-\delta_-^2/2}\max\{F_{\delta_-}(\|Z+\delta_+ v\|),F_{\delta_+}(\|Z+\delta_+ v\|)\}\leq 1\right) \\
&\rightarrow& \mathbb{P}\left(e^{-\delta^2/2}F_{\delta}(\|Z+\delta v\|)\leq 1\right)
\end{eqnarray*}
Note that $\mathbb{P}\left(e^{-\delta^2/2}F_{\delta}(\|Z+\delta v\|)\leq 1\right)$ is independent of $v$. Therefore, by the fact that $F_{\delta}(x)$ is increasing on $x>0$, we have
\begin{eqnarray*}
R_n(k,n) &\geq& \mathbb{P}_{\mu}\left(\mathcal{L}_->1\right) + \inf_{\|v\|=1}\mathbb{P}_{\mu+\frac{\delta_v}{\sqrt{n}}v}\left(\mathcal{L}_+\leq 1\right) \\
&\geq& (1+o(1))\left\{\mathbb{P}\left(e^{-\delta^2/2}F_{\delta}(\|Z\|)>1\right)+\inf_{\|v\|=1}\mathbb{P}\left(e^{-\delta^2/2}F_{\delta}(\|Z+\delta v\|)\leq 1\right)\right\} \\
&\geq& (1+o(1))\inf_{t>0}\left\{\mathbb{P}(\chi_k^2\geq t)+\mathbb{P}(\chi_{k,\delta^2}^2<t)\right\}.
\end{eqnarray*}
The proof is complete.
\end{proof}

\begin{proof}[Proof of Theorem \ref{thm:lower-cat}]
Note that Condition M3 implies $\sqrt{n}|\sqrt{q_j}-\sqrt{q_l}|\rightarrow\infty$ for any $j\neq l$.
Consider the set
$$\bar{\mathcal{P}}_{\delta}=\left\{p:\|\sqrt{p}-\sqrt{q}\|=\frac{\delta}{\sqrt{n}}\right\}.$$
For each $p\in\bar{\mathcal{P}}_{\delta}$, $|\sqrt{p_j}-\sqrt{q_j}|^2\leq \frac{\delta^2}{n}$, which implies $\sqrt{q_j}$ is the closest element to $\sqrt{p_j}$ in the set $\{\sqrt{q_1},...,\sqrt{q_k}\}$. Therefore, $\ell(p,q)=\|\sqrt{p}-\sqrt{q}\|=\delta/\sqrt{n}$, which implies $\bar{\mathcal{P}}_{\delta}\subset\mathcal{P}_{\delta}$. This gives the lower bound
$$R_n(k,\delta)\geq \inf_{0\leq\phi\leq 1}\left\{\mathbb{P}_{q}\phi+\sup_{p\in\bar{\mathcal{P}}_{\delta}}\mathbb{P}_{p}(1-\phi)\right\}.$$
Let $\Pi$ be the uniform distribution on the sphere $\{v:\|v-\sqrt{q}\|=\delta/\sqrt{n}\}$. Then,
$$R_n(k,\delta)\geq \inf_{0\leq\phi\leq 1}\left\{\mathbb{P}_{q}\phi+\int\mathbb{P}_{p}(1-\phi)d\Pi(\sqrt{p})\right\}.$$
By Neyman-Pearson lemma, the optimal testing function $\phi$ is given by
$$\phi=\left\{\frac{d\int\mathbb{P}_{p}d\Pi(\sqrt{p})}{d\mathbb{P}_{q}}>1\right\}.$$
By the definition, we have
$$\mathcal{L}=\frac{d\int\mathbb{P}_{p}d\Pi(\sqrt{p})}{d\mathbb{P}_{q}}=\int\exp\left(n\sum_{j=1}^kq_j\log\frac{p_j}{q_j}+n\sum_{j=1}^k(\hat{p}_j-q_j)\log\frac{p_j}{q_j}\right) d\Pi(\sqrt{p}),$$
where $\hat{p}_j=\frac{1}{n}\sum_{i=1}^n\mathbb{I}\{X_i=j\}$. Note that $\log\sqrt{\frac{p_j}{q_j}}=\log\left(1+\frac{\sqrt{p_j}-\sqrt{q_j}}{\sqrt{q_j}}\right)$. By Condition M3, $\max_{1\leq j\leq k}\left|\frac{\sqrt{p_j}-\sqrt{q_j}}{\sqrt{q_j}}\right|=o(1)$. Therefore,
\begin{equation}
\max_{1\leq j\leq k}\frac{\left|\log\sqrt{\frac{p_j}{q_j}}-\frac{\sqrt{p_j}-\sqrt{q_j}}{\sqrt{q_j}}\right|}{\left|\frac{\sqrt{p_j}-\sqrt{q_j}}{\sqrt{q_j}}\right|^2}=O(1),\label{eq:log-1}
\end{equation}
and
\begin{equation}
\max_{1\leq j\leq k}\frac{\left|\log\sqrt{\frac{p_j}{q_j}}-\frac{\sqrt{p_j}-\sqrt{q_j}}{\sqrt{q_j}}+\frac{1}{2}\left(\frac{\sqrt{p_j}-\sqrt{q_j}}{\sqrt{q_j}}\right)^2\right|}{\left|\frac{\sqrt{p_j}-\sqrt{q_j}}{\sqrt{q_j}}\right|^3}=O(1).\label{eq:log-2}
\end{equation}
Since
$$\sum_{j=1}^kq_j\left[\frac{\sqrt{p_j}-\sqrt{q_j}}{\sqrt{q_j}}-\frac{1}{2}\left(\frac{\sqrt{p_j}-\sqrt{q_j}}{\sqrt{q_j}}\right)^2\right]=-\|\sqrt{p}-\sqrt{q}\|^2.$$
By (\ref{eq:log-2}), we have
$$\sum_{j=1}^kq_j\log\frac{p_j}{q_j}=-(1+o(1))2\|\sqrt{p}-\sqrt{q}\|^2.$$
Under Condition M3, $\hat{p}_j/p_j=1+o_P(1)$, and this implies $\hat{p}_j/q_j=1+o_P(1)$. Therefore,
\begin{equation}
\sum_{j=1}^k(\hat{p}_j-q_j)\frac{\sqrt{p_j}-\sqrt{q_j}}{\sqrt{q_j}}=2(1+o_P(1))\sum_{j=1}^k(\sqrt{\hat{p}_j}-\sqrt{q_j})(\sqrt{p_j}-\sqrt{q_j}).\label{eq:appr1}
\end{equation}
By (\ref{eq:log-1}), we have
\begin{equation}
\sum_{j=1}^k(\hat{p}_j-q_j)\log\frac{p_j}{q_j}=4(1+o_P(1))\sum_{j=1}^k(\sqrt{\hat{p}_j}-\sqrt{q_j})(\sqrt{p_j}-\sqrt{q_j}).\label{eq:appr2}
\end{equation}
The approximations (\ref{eq:appr1}) and (\ref{eq:appr2}) imply the existence of $\delta_-$ and $\delta_+$ that satisfies $\delta_-=(1+o(1))\delta$, $\delta_+=(1+o(1))\delta$. Moreover, on an event $E$ with probability $1-o(1)$ under both null and alternative, the following inequalities hold:
$$-2\delta_+^2\leq n\sum_{j=1}^kq_j\log\frac{p_j}{q_j}\leq-2\delta_-^2,$$
$$\sum_{j=1}^k(\hat{p}_j-q_j)\log\frac{p_j}{q_j}\leq \max\left\{\frac{4\delta_-}{\sqrt{n}}\iprod{\sqrt{\hat{p}}-\sqrt{q}}{\frac{\sqrt{p}-\sqrt{q}}{\|\sqrt{p}-\sqrt{q}\|}},\frac{4\delta_+}{\sqrt{n}}\iprod{\sqrt{\hat{p}}-\sqrt{q}}{\frac{\sqrt{p}-\sqrt{q}}{\|\sqrt{p}-\sqrt{q}\|}}\right\},$$
and
$$\sum_{j=1}^k(\hat{p}_j-q_j)\log\frac{p_j}{q_j}\geq \min\left\{\frac{4\delta_-}{\sqrt{n}}\iprod{\sqrt{\hat{p}}-\sqrt{q}}{\frac{\sqrt{p}-\sqrt{q}}{\|\sqrt{p}-\sqrt{q}\|}},\frac{4\delta_+}{\sqrt{n}}\iprod{\sqrt{\hat{p}}-\sqrt{q}}{\frac{\sqrt{p}-\sqrt{q}}{\|\sqrt{p}-\sqrt{q}\|}}\right\}.$$
We introduce the upper and lower brackets of the $\mathcal{L}$ as
\begin{eqnarray}
\label{eq:L-} \mathcal{L}_- &=& \min\left\{\int\exp\left(-2\delta_+^2+4\delta_-\sqrt{n}\iprod{\sqrt{\hat{p}}-\sqrt{q}}{v}\right)d\bar{\Pi}(v),\right.\\
\nonumber && \left.\int\exp\left(-2\delta_+^2+4\delta_+\sqrt{n}\iprod{\sqrt{\hat{p}}-\sqrt{q}}{v}\right)d\bar{\Pi}(v)\right\}, \\
\label{eq:L+} \mathcal{L}_+ &=& \max\left\{\int\exp\left(-2\delta_-^2+4\delta_-\sqrt{n}\iprod{\sqrt{\hat{p}}-\sqrt{q}}{v}\right)d\bar{\Pi}(v),\right.\\
\nonumber && \left.\int\exp\left(-2\delta_-^2+4\delta_+\sqrt{n}\iprod{\sqrt{\hat{p}}-\sqrt{q}}{v}\right)d\bar{\Pi}(v)\right\},
\end{eqnarray}
where $\bar{\Pi}$ is the uniform distribution on the unit sphere $\{v:\|v\|=1\}$. 
By (\ref{eq:int-sphere}), we have
\begin{eqnarray*}
\mathcal{L}_- &=& e^{-2\delta_+^2}\min\{F_{2\delta_-}(2\sqrt{n}\|\sqrt{\hat{p}}-\sqrt{q}\|),F_{2\delta_+}(2\sqrt{n}\|\sqrt{\hat{p}}-\sqrt{q}\|)\}, \\
\mathcal{L}_+ &=& e^{-2\delta_-^2}\max\{F_{2\delta_-}(2\sqrt{n}\|\sqrt{\hat{p}}-\sqrt{q}\|),F_{2\delta_+}(2\sqrt{n}\|\sqrt{\hat{p}}-\sqrt{q}\|)\}.
\end{eqnarray*}
where $F_{\delta}(x)$ is defined in (\ref{eq:F-delta}). Note that $4n\|\sqrt{\hat{p}}-\sqrt{q}\|^2\leadsto \chi_{k-1}^2$ under the null and $4n\|\sqrt{\hat{p}}-\sqrt{q}\|^2-\delta_2^2\leadsto \chi_{k-1,\delta_1^2}^2$ , with $\delta_1^2=\delta_1(p)^2$ and $\delta_2^2=\delta_2(p)^2$ defined in (\ref{eq:def-delta1}) and (\ref{eq:def-delta2}), under the alternative.
Define $\phi=\mathbb{I}\{\mathcal{L}>1\}$, $\phi_-=\mathbb{I}\{\mathcal{L}_->1\}$, $\phi_+=\mathbb{I}\{\mathcal{L}_+>1\}$ and $\phi^*=\mathbb{I}\{\mathcal{L}^*>1\}$. Then, we have the inequality $\phi_-\mathbb{I}_E\leq\phi\mathbb{I}_E\leq \phi_+\mathbb{I}_E$.
For $q=p$, we have
\begin{eqnarray*}
\mathbb{P}_q\phi &=& \mathbb{P}_q\phi\mathbb{I}_E + \mathbb{P}_q\phi\mathbb{I}_{E^c} \\
&\geq& \mathbb{P}_q\phi_-\mathbb{I}_E \\
&\geq& \mathbb{P}_q\left(\mathcal{L}_->1\right) - \mathbb{P}_q(E^c) \\
&=& \mathbb{P}_q\left(e^{-2\delta_+^2}\min\{F_{2\delta_-}(2\sqrt{n}\|\sqrt{\hat{p}}-\sqrt{q}\|),F_{2\delta_+}(2\sqrt{n}\|\sqrt{\hat{p}}-\sqrt{q}\|)\}>1\right) - \mathbb{P}_q(E^c) \\
&\rightarrow& \mathbb{P}\left(e^{-2\delta^2}F_{2\delta}\left(\sqrt{\chi_{k-1}^2}\right)>1\right).
\end{eqnarray*}
For the alternative, we have
\begin{eqnarray*}
\mathbb{P}_p(1-\phi) &=& \mathbb{P}_p(1-\phi)\mathbb{I}_E + \mathbb{P}_p(1-\phi)\mathbb{I}_{E^c} \\
&\leq& \mathbb{P}_p(1-\phi_+)\mathbb{I}_E + \mathbb{P}_p(E^c) \\
&\leq& \mathbb{P}_p(1-\phi_+) + \mathbb{P}_p(E^c) \\
&=& \mathbb{P}_p\left(e^{-2\delta_-^2}\max\{F_{2\delta_-}(2\sqrt{n}\|\sqrt{\hat{p}}-\sqrt{q}\|),F_{2\delta_+}(2\sqrt{n}\|\sqrt{\hat{p}}-\sqrt{q}\|)\}\leq 1\right) + \mathbb{P}_p(E^c) \\
&\rightarrow& \mathbb{P}\left(e^{-2\delta^2}F_{2\delta}\left(\sqrt{\chi_{k-1,\delta_1(p)^2}^2+\delta_2(p)^2}\right)\leq 1\right) \\
&\geq& \inf_{\{\delta_1,\delta_2:\delta_1^2+\delta_2^2=\delta^2\}}\mathbb{P}\left(e^{-2\delta^2}F_{2\delta}\left(\sqrt{\chi_{k-1,\delta_1^2}^2+\delta_2^2}\right)\leq 1\right).
\end{eqnarray*}
By the fact that $F_{\delta}(x)$ is increasing on $x>0$, we have
\begin{eqnarray*}
R_n(k,\delta) &\geq& (1+o(1))\left\{\mathbb{P}\left(e^{-2\delta^2}F_{2\delta}\left(\sqrt{\chi_{k-1}^2}\right)>1\right)+\right. \\
&& \left.\inf_{\{\delta_1,\delta_2:\delta_1^2+\delta_2^2=\delta^2\}}\mathbb{P}\left(e^{-2\delta^2}F_{2\delta}\left(\sqrt{\chi_{k-1,\delta_1^2}^2+\delta_2^2}\right)\leq 1\right)\right\} \\
&\geq& (1+o(1))\inf_{t>0}\left(\mathbb{P}\left(\chi_{k-1}^2>t\right)+\inf_{\{\delta_1,\delta_2:\delta_1^2+\delta_2^2=\delta^2\}}\mathbb{P}(\chi_{k-1,\delta_1^2}^2+\delta_2^2\leq t)\right).
\end{eqnarray*}
The proof is complete.
\end{proof}

\begin{proof}[Proof of Theorem \ref{thm:lower-cat-degen}]
It is implied by Condition M4 that $\sqrt{n}|\sqrt{r_g}-\sqrt{r_h}|\rightarrow\infty$ for any $g\neq h$. Moreover, for any $j\in\mathcal{C}_h$, $|\sqrt{q_j}-\sqrt{r_h}|=o(n^{-1/2})$. Under these assumptions, for any $p$ such that $\|\sqrt{p}-\sqrt{q}\|=\frac{\delta}{\sqrt{n}}$, there exists a $\pi\in S_k$ that depends on $p$ and $\|\sqrt{p_{\pi}}-\sqrt{q}\|=\ell(p,q)=\frac{\delta}{\sqrt{n}}(1+\epsilon_{\theta})$. Moreover, $|\epsilon_{\theta}|=o(1)$ uniformly over all $p$ that satisfies $\|\sqrt{p}-\sqrt{q}\|=\frac{\delta}{\sqrt{n}}$. Define
\begin{equation}
{p'}=\left(\sqrt{q}+\frac{1}{1+\epsilon_{\theta}}(\sqrt{p}-\sqrt{q})\right)^2.\label{eq:def-opt-p}
\end{equation}
Then, $\|\sqrt{p}-\sqrt{q}\|=\frac{\delta_{\theta}}{\sqrt{n}}$ and $\ell(p',q)=\frac{\delta}{\sqrt{n}}$, where $\delta_{\theta}=\frac{\delta}{1+\epsilon_{\theta}}$. We use the notation $R$ to denote the operator $R:p\mapsto R(p)$ defined by (\ref{eq:def-opt-p}). By the definition, a useful property is $\frac{\sqrt{R(p)}-\sqrt{q}}{\|\sqrt{R(p)}-\sqrt{q}\|}=\frac{\sqrt{p}-\sqrt{q}}{\|\sqrt{p}-\sqrt{q}\|}$. Consider the set
$$\bar{\mathcal{P}}_{\delta}=\left\{R(p):\|\sqrt{p}-\sqrt{q}\|=\frac{\delta}{\sqrt{n}}\right\}.$$
This definition immediately implies $\bar{\mathcal{P}}_{\delta}\subset\mathcal{P}_{\delta}$. Note that each element in $\bar{\mathcal{P}}_{\delta}$ can be represented as
$$R(p)=\left(\sqrt{q}+\frac{\delta_{\theta}}{\sqrt{n}}\frac{\sqrt{p}-\sqrt{q}}{\|\sqrt{p}-\sqrt{q}\|}\right)^2.$$
Since there is a one-to-one relation between $\frac{\sqrt{p}-\sqrt{q}}{\|\sqrt{p}-\sqrt{q}\|}$ and a unit vector $v$, we can also write each element in $\bar{\mathcal{P}}_{\delta}$ as $\left(\sqrt{q}+\frac{\delta_v}{\sqrt{n}} v\right)^2$. Consider a uniform probability measure $\bar{\Pi}$ on $\{v:\|v\|=1\}$. Then, by the same argument in the proof of Theorem \ref{thm:lower-gauss},
$$R_n(k,\delta)\geq \inf_{0\leq\phi\leq 1}\left\{\mathbb{P}_{q}\phi+\int\mathbb{P}_{\left(\sqrt{q}+\frac{\delta_v}{\sqrt{n}} v\right)^2}(1-\phi)d\bar{\Pi}(v)\right\},$$
and the likelihood ratio is
$$\mathcal{L}=\int \frac{d\mathbb{P}_{\left(\sqrt{q}+\frac{\delta_v}{\sqrt{n}} v\right)^2}}{d\mathbb{P}_{q}} d\bar{\Pi}(v).$$
Using the same arguments in the proofs of Theorem \ref{thm:lower-gauss-degen} and Theorem \ref{thm:lower-cat}, there exist $\delta_-$ and $\delta_+$, with which we can define $\mathcal{L}_-$ and $\mathcal{L}_+$ as in (\ref{eq:L-}) and (\ref{eq:L+}) with the desired properties. Then, the same argument in the proof of Theorem \ref{thm:lower-cat} leads to the desired result.
\end{proof}

\begin{proof}[Proof of Theorem \ref{thm:upper-gauss}]
By studying the proof of Theorem \ref{thm:Gaussian}, the only probabilistic argument in approximation is that $\max_{1\leq j\leq k}Z_j^2\leq C_n$ in probability. Since this event is independent of $\theta$, the in-probability argument can be made uniformly over $\theta\in\Theta_{\delta}$ and $\theta\in\Theta_0$.
\end{proof}

\begin{proof}[Proof of Theorem \ref{thm:upper-gauss-degen}]
By Theorem \ref{thm:M2}, $T_g\geq T_f$ in probability. This implies that $\mathbb{P}_{\theta}\phi=\mathbb{P}_{\theta}(T_g>t^*)$ and $\mathbb{P}_{\theta}(1-\phi)=\mathbb{P}_{\theta}(T_g\leq t^*)$ under both null and alternative distributions. Then, by the same argument in the proof of Theorem \ref{thm:upper-gauss}, we obtain the desired conclusion.
\end{proof}

\begin{proof}[Proofs of Theorem \ref{thm:upper-cat} and Theorem \ref{thm:upper-cat-degen}]
Similar to the argument used in the proof of Theorem \ref{thm:upper-gauss}, the results directly follow the conclusions of Theorem \ref{thm:M3} and Theorem \ref{thm:M4}.
\end{proof}

\bibliographystyle{plainnat}
\bibliography{test}

\begin{thebibliography}{22}
\providecommand{\natexlab}[1]{#1}
\providecommand{\url}[1]{\texttt{#1}}
\expandafter\ifx\csname urlstyle\endcsname\relax
  \providecommand{\doi}[1]{doi: #1}\else
  \providecommand{\doi}{doi: \begingroup \urlstyle{rm}\Url}\fi

\bibitem[Anandkumar et~al.(2012)Anandkumar, Hsu, and
  Kakade]{anandkumar2012method}
Animashree Anandkumar, Daniel Hsu, and Sham~M Kakade.
\newblock A method of moments for mixture models and hidden markov models.
\newblock In \emph{Conference on Learning Theory}, pages 33--1, 2012.

\bibitem[Anandkumar et~al.(2013)Anandkumar, Ge, Hsu, and
  Kakade]{anandkumar2013tensor}
Animashree Anandkumar, Rong Ge, Daniel Hsu, and Sham Kakade.
\newblock A tensor spectral approach to learning mixed membership community
  models.
\newblock In \emph{Conference on Learning Theory}, pages 867--881, 2013.

\bibitem[Anscombe(1948)]{anscombe1948transformation}
Francis~J Anscombe.
\newblock The transformation of poisson, binomial and negative-binomial data.
\newblock \emph{Biometrika}, 35\penalty0 (3/4):\penalty0 246--254, 1948.

\bibitem[Arabshahi and Anandkumar(2016)]{arabshahi2016spectral}
Forough Arabshahi and Animashree Anandkumar.
\newblock Spectral methods for correlated topic models.
\newblock \emph{arXiv preprint arXiv:1605.09080}, 2016.

\bibitem[Arora et~al.(2014)Arora, Ge, and Moitra]{arora2014new}
Sanjeev Arora, Rong Ge, and Ankur Moitra.
\newblock New algorithms for learning incoherent and overcomplete dictionaries.
\newblock In \emph{Conference on Learning Theory}, pages 779--806, 2014.

\bibitem[Banerjee and Ma(2017)]{banerjee2017optimal}
Debapratim Banerjee and Zongming Ma.
\newblock Optimal hypothesis testing for stochastic block models with growing
  degrees.
\newblock \emph{arXiv preprint arXiv:1705.05305}, 2017.

\bibitem[Blanchard et~al.(2017)Blanchard, Carpentier, and
  Gutzeit]{blanchard2017minimax}
Gilles Blanchard, Alexandra Carpentier, and Maurilio Gutzeit.
\newblock Minimax euclidean separation rates for testing convex hypotheses in
  $\mathbb{R}^d$.
\newblock \emph{arXiv preprint arXiv:1702.03760}, 2017.

\bibitem[Crane(2016)]{crane2016ubiquitous}
Harry Crane.
\newblock The ubiquitous ewens sampling formula.
\newblock \emph{Statistical Science}, 31\penalty0 (1):\penalty0 1--19, 2016.

\bibitem[De~Blasi et~al.(2015)De~Blasi, Favaro, Lijoi, Mena, Pr{\"u}nster, and
  Ruggiero]{de2015gibbs}
Pierpaolo De~Blasi, Stefano Favaro, Antonio Lijoi, Rams{\'e}s~H Mena, Igor
  Pr{\"u}nster, and Matteo Ruggiero.
\newblock Are gibbs-type priors the most natural generalization of the
  dirichlet process?
\newblock \emph{IEEE transactions on pattern analysis and machine
  intelligence}, 37\penalty0 (2):\penalty0 212--229, 2015.

\bibitem[Ewens(1972)]{ewens1972sampling}
Warren~J Ewens.
\newblock The sampling theory of selectively neutral alleles.
\newblock \emph{Theoretical population biology}, 3\penalty0 (1):\penalty0
  87--112, 1972.

\bibitem[Gao and Lafferty(2017)]{gao2017testing}
Chao Gao and John Lafferty.
\newblock Testing network structure using relations between small subgraph
  probabilities.
\newblock \emph{arXiv preprint arXiv:1704.06742}, 2017.

\bibitem[Gnedin and Pitman(2006)]{gnedin2006exchangeable}
Alexander Gnedin and Jim Pitman.
\newblock Exchangeable gibbs partitions and stirling triangles.
\newblock \emph{Journal of Mathematical sciences}, 138\penalty0 (3):\penalty0
  5674--5685, 2006.

\bibitem[Horn and Johnson(2012)]{horn2012matrix}
Roger~A Horn and Charles~R Johnson.
\newblock \emph{Matrix analysis}.
\newblock Cambridge university press, 2012.

\bibitem[Hsu and Kakade(2013)]{hsu2013learning}
Daniel Hsu and Sham~M Kakade.
\newblock Learning mixtures of spherical gaussians: moment methods and spectral
  decompositions.
\newblock In \emph{Proceedings of the 4th conference on Innovations in
  Theoretical Computer Science}, pages 11--20. ACM, 2013.

\bibitem[Kingman(1975)]{kingman1975random}
John~FC Kingman.
\newblock Random discrete distributions.
\newblock \emph{Journal of the Royal Statistical Society. Series B
  (Methodological)}, pages 1--22, 1975.

\bibitem[Kocherlakota and Kocherlakota(1986)]{kocherlakota1986goodness}
S~Kocherlakota and K~Kocherlakota.
\newblock Goodness of fit tests for discrete distributions.
\newblock \emph{Communications in statistics-theory and methods}, 15\penalty0
  (3):\penalty0 815--829, 1986.

\bibitem[Kulperger and Singh(1982)]{kulperger1982random}
RJ~Kulperger and AC~Singh.
\newblock On random grouping in goodness of fit tests of discrete
  distributions.
\newblock \emph{Journal of Statistical Planning and Inference}, 7\penalty0
  (2):\penalty0 109--115, 1982.

\bibitem[M{\"u}ller et~al.(2013)M{\"u}ller, Rodriguez,
  et~al.]{muller2013random}
Peter M{\"u}ller, Abel Rodriguez, et~al.
\newblock Random partition models.
\newblock In \emph{Nonparametric Bayesian Inference}, pages 87--92. IMS and
  ASA, 2013.

\bibitem[Pearson(1900)]{pearson1900x}
Karl Pearson.
\newblock On the criterion that a given system of deviations from the probable
  in the case of a correlated system of variables is such that it can be
  reasonably supposed to have arisen from random sampling.
\newblock \emph{The London, Edinburgh, and Dublin Philosophical Magazine and
  Journal of Science}, 50\penalty0 (302):\penalty0 157--175, 1900.

\bibitem[Perry et~al.(2017)Perry, Weed, Bandeira, Rigollet, and
  Singer]{perry2017sample}
Amelia Perry, Jonathan Weed, Afonso Bandeira, Philippe Rigollet, and Amit
  Singer.
\newblock The sample complexity of multi-reference alignment.
\newblock \emph{arXiv preprint arXiv:1707.00943}, 2017.

\bibitem[Pitman(1995)]{pitman1995exchangeable}
Jim Pitman.
\newblock Exchangeable and partially exchangeable random partitions.
\newblock \emph{Probability theory and related fields}, 102\penalty0
  (2):\penalty0 145--158, 1995.

\bibitem[Pitman(1996)]{pitman1996random}
Jim Pitman.
\newblock Random discrete distributions invariant under size-biased
  permutation.
\newblock \emph{Advances in Applied Probability}, 28\penalty0 (2):\penalty0
  525--539, 1996.

\end{thebibliography}

%\end{raggedright}           % Comment this out if you don't want ragged edges.

\end{document}